\documentclass[11pt,reqno]{amsart}
\usepackage[foot]{amsaddr}
\usepackage[margin=1in]{geometry}

\usepackage{amsmath}
\usepackage{amssymb}
\usepackage{amsthm}
\usepackage{hyperref}
\usepackage{bbm}
\usepackage{enumerate}
\usepackage{booktabs}
\usepackage{multirow}
\usepackage{siunitx}
\usepackage[section]{placeins}
\usepackage{graphicx}

\allowdisplaybreaks

\theoremstyle{plain}
\newtheorem{theorem}{Theorem}[section]
\newtheorem*{theorem*}{Theorem}
\newtheorem{proposition}[theorem]{Proposition}
\newtheorem{lemma}[theorem]{Lemma}
\newtheorem{corollary}[theorem]{Corollary}

\theoremstyle{remark}
\newtheorem{definition}[theorem]{Definition}
\newtheorem{example}[theorem]{Example}
\newtheorem{assumption}[theorem]{Assumption}
\newtheorem{remark}[theorem]{Remark}

\newcommand{\R}{\mathbb{R}}
\newcommand{\C}{\mathbb{C}}
\newcommand{\E}{\mathbb{E}}
\renewcommand{\P}{\mathbb{P}}
\newcommand{\N}{\mathcal{N}}
\newcommand{\eps}{\varepsilon}
\newcommand{\1}{\mathbbm{1}}
\newcommand{\Id}{\operatorname{Id}}
\renewcommand{\Im}{\operatorname{Im}}
\renewcommand{\Re}{\operatorname{Re}}

\renewcommand{\O}{O_{\prec}}

\newcommand{\Tr}{\operatorname{Tr}}
\newcommand{\col}{\operatorname{col}}
\renewcommand{\vec}{\operatorname{vec}}
\newcommand{\diag}{\operatorname{diag}}
\newcommand{\sign}{\operatorname{sign}}
\newcommand{\rank}{\operatorname{rank}}
\newcommand{\supp}{\operatorname{supp}}
\newcommand{\dist}{\operatorname{dist}}

\newcommand{\HS}{{\text{HS}}}

\renewcommand{\v}{\mathbf{v}}

\newcommand{\y}{\mathbf{y}}

\newcommand{\bD}{\mathbf{D}}

\newcommand{\bmu}{\boldsymbol{\mu}}
\newcommand{\balpha}{\boldsymbol{\alpha}}

\newcommand{\beps}{\boldsymbol{\varepsilon}}
\newcommand{\one}{\mathbf{1}}

\renewcommand{\a}{\alpha}
\renewcommand{\b}{\beta}

\renewcommand{\t}{\mathfrak{t}}
\newcommand{\X}{\mathfrak{X}}
\newcommand{\Y}{\mathfrak{Y}}

\newcommand{\vT}{\check{T}}

\newcommand{\vG}{\check{G}}

\newcommand{\vU}{\check{U}}
\newcommand{\vV}{\check{V}}
\newcommand{\vR}{\check{R}}

\newcommand{\vg}{\check{\gamma}}
\newcommand{\vt}{\check{t}}
\newcommand{\vs}{\check{s}}
\newcommand{\vf}{\check{f}}
\newcommand{\vm}{\check{m}}

\newcommand{\vz}{\check{z}}
\newcommand{\vE}{\check{E}}
\newcommand{\vX}{\check{\mathfrak{X}}}

\renewcommand{\t}[1]{\widetilde{#1}}
\newcommand{\tG}{\tilde{G}}
\newcommand{\tX}{\tilde{\mathfrak{X}}}

\newcommand{\ty}{\tilde{y}}
\newcommand{\tm}{\tilde{m}}
\newcommand{\pa}{^{(\a)}}
\newcommand{\tvG}{\tilde{\check{G}}}

\newcommand{\G}{\mathcal{G}}

\newcommand{\I}{\mathcal{I}}

\newcommand{\cP}{\mathcal{P}}
\newcommand{\Q}{\mathcal{Q}}
\newcommand{\cR}{\mathcal{R}}

\newcommand{\cZ}{\mathcal{Z}}

\newcommand{\hsigma}{\hat{\sigma}}

\newcommand{\hSigma}{\widehat{\Sigma}}

\title[Tracy-Widom for covariance matrices]
{Tracy-Widom at each edge of real covariance and MANOVA estimators}
\author{Zhou Fan}
\address{Department of Statistics and Data Science, Yale University}
\email{zhou.fan@yale.edu}
\author{Iain M. Johnstone}
\address{Department of Statistics, Stanford University}
\email{imj@stanford.edu}

\begin{document}

\maketitle

\begin{abstract}
We study the sample covariance matrix for real-valued data with general
population covariance, as well as MANOVA-type covariance estimators in variance
components models under null hypotheses of global sphericity. 
In the limit as matrix dimensions increase proportionally,
the asymptotic spectra of such estimators may
have multiple disjoint intervals of support, possibly intersecting the
negative half line. We show that the distribution of the extremal eigenvalue
at each regular edge of the support has a GOE Tracy-Widom limit. Our proof
extends a comparison argument of Ji Oon Lee and Kevin Schnelli, replacing a
continuous Green function flow by a discrete Lindeberg swapping scheme.
\end{abstract}

\section{Introduction}

Consider a matrix $\hSigma=X'TX$, where $X \in \R^{M \times N}$ has
random independent entries, and $T \in \R^{M \times M}$
is deterministic. We study eigenvalue fluctuations at 
the edges of the spectrum of $\hSigma$, when $M \asymp N$ are both large.

At the largest edge and for $T \succ 0$, a substantial literature,
reviewed below, shows that the fluctuations of the largest eigenvalue of
$\hSigma$ follow the Tracy-Widom distribution.
In this paper, we extend the validity of this Tracy-Widom limit
to matrices $T$ with both positive and negative eigenvalues, and to
all ``regular'' edges of the spectrum of $\hSigma$. 
Our main result is stated informally as follows:
\begin{theorem*}[Informal]
Let $\hSigma=X'TX$, where $\sqrt{N}X \in \R^{M \times N}$
has independent entries with mean 0, variance 1, and bounded higher moments,
and $T \in \R^{M \times M}$ is diagonal with bounded entries. Let $\mu_0$
be the deterministic approximation for the spectrum of $\hSigma$ and
let $E_*$ be any regular edge of the support of $\mu_0$. Then for
$\lambda(\hSigma)$ the extremal eigenvalue of $\hSigma$ near $E_*$, 
and for a scale constant $\gamma>0$,
\[\pm (\gamma N)^{2/3}(\lambda(\hSigma)-E_*) \overset{L}{\to} \mu_{TW}.\]
\end{theorem*}
Here, $\mu_{TW}$ is the GOE Tracy-Widom law \cite{tracywidomorthogonal}.
A formal statement is provided in Theorem \ref{thm:TW}, and we comment on the
assumption of diagonal $T$ in Remark \ref{remark:diagonal} below.

Our study of this model is motivated by two applications in statistics and
genetics. In the first well-studied setting, $\y_1,\ldots,\y_n \in \R^p$ are
observations of $p$ variables, or ``traits'',
in $n$ independent samples. When the traits
are distributed with mean 0 and covariance $\Sigma \in \R^{p \times p}$, the
sample covariance matrix $\tilde{\Sigma}=n^{-1}Y'Y$
provides an unbiased estimate of $\Sigma$,
where $Y \in \R^{n \times p}$ is a row-wise stacking of $\y_1,\ldots,\y_n$.
Writing $Y=n^{1/2}X'\Sigma^{1/2}$, this takes the form
\begin{equation}\label{eq:samplecovariance}
\tilde{\Sigma}=\Sigma^{1/2}XX'\Sigma^{1/2}.
\end{equation}
The non-zero eigenvalues of $\tilde{\Sigma}$ are the same as those of its
``companion'' matrix $\hSigma = X'\Sigma X$.
Here $T = \Sigma$ is positive definite, and since $\y_1,\ldots,\y_n$ are
independent and identically distributed, there is a single level of variation.

In the second setting, we consider models with multiple levels of
variation which induce \textit{dependence} among the observations. 
For example, suppose the samples are divided into $I$ groups of size
$J=n/I$, and modeled by a random effects linear model
where the traits for sample $j$ of group $i$ are given by
\[\y_{i,j}=\balpha_i+\beps_{i,j} \in \R^p.\]
Here, $\balpha_i,\beps_{i,j}$
are independent vectors capturing variation at the group and
individual levels, with mean 0 and respective
covariances $\Sigma_1,\Sigma_2 \in \R^{p \times p}$.
The traditional (MANOVA) estimate of the variance component $\Sigma_1$ is
\begin{equation}\label{eq:MANOVA1}
    \hSigma=Y'BY,
\end{equation}
where again $Y \in \R^{n \times p}$ is a row-wise stacking of
the observations $\y_{i,j}$. The matrix $B$ is \textit{not} positive definite,
having $n - I$ negative eigenvalues: Loosely speaking, one subtracts a
scaled estimate of the second-level noise $\Sigma_2$ to estimate $\Sigma_1$.
Under a null hypothesis of ``global sphericity'' where $\Sigma_1,\Sigma_2
\propto \Id$, and
introducing a representation $Y=UX$ detailed in Section \ref{subsec:MANOVA},
we obtain $\hSigma=X'TX$ with $T=U'BU$ having positive and negative
eigenvalues in non-vanishing proportions.
\cite[Boxes 1 and 2]{blowsmcguigan} has an example from quantitative
genetics, and our main result resolves an open question stated there
about Tracy-Widom limits and scaling constants in this model.

Returning to the general discussion,
when $M,N \to \infty$ proportionally, the
empirical spectrum of $\hSigma$ is well approximated by a deterministic
law $\mu_0$ \cite{marcenkopastur,yinyq,silverstein,silversteinbai}. Under a
``sphericity'' null hypothesis that $T=\Id$, the law $\mu_0$ is
the Marcenko-Pastur distribution, and the largest and smallest
eigenvalues of $\hSigma$ converge to the edges of the support of
$\mu_0$ \cite{geman,yinetal,baiyin} and have asymptotic
GUE/GOE Tracy-Widom fluctuations
\cite{johansson,johnstone,soshnikov,peche,feldheimsodin,pillaiyin}. 
In statistics and genetics,
these results have enabled the application of Roy's largest root
test in high-dimensional principal components analysis
\cite{johnstone,pattersonetal}.

In this paper, we study $\hSigma$ in the setting $T \neq \Id$.
For $T \succeq 0$, 
\cite{baisilverstein} showed that all eigenvalues of $\hSigma$ converge
to the support of $\mu_0$, and \cite{baisilversteinexact,knowlesyin}
proved exact separation of eigenvalues and eigenvalue rigidity.
For complex Gaussian $X$ and $T \succ 0$,
\cite{elkaroui,onatski} established GUE Tracy-Widom fluctuations
of the largest eigenvalue,
under an edge regularity condition introduced in
\cite{elkaroui}. For complex Gaussian $X$, this was extended to each regular
edge of the support in \cite{hachemetal}. For real $X$
and diagonal $T \succ 0$, \cite{leeschnelli} established GOE Tracy-Widom
fluctuations of the largest eigenvalue, using different techniques based on
earlier work for the deformed Wigner model
in \cite{leeschnelliwigner}. Universality results
of \cite{baopanzhouTW,knowlesyin} lift these assumptions that
$X$ is Gaussian and/or $T$ is diagonal.

We build on the proof in \cite{leeschnelli} to extend the above picture
in two directions: First, we establish a GOE Tracy-Widom limit at each regular
edge of the support for real $X$, including the interior edges.
This extension is new even in the Gaussian setting.
Second, we extend the notion of edge regularity and associated analysis
to $T$ having both positive and negative eigenvalues.
This is important for our study of random effects models with multiple
levels of variation, whose edge behavior is obtained here for the
first time.

\begin{remark}\label{remark:diagonal}
We restrict attention as in
\cite{leeschnelli} to diagonal $T$. By rotational invariance, this
encompasses the case of non-diagonal $T$ and real Gaussian $X$.
Existing universality results of \cite{baopanzhouTW,knowlesyin}
imply that our conclusions hold also for non-diagonal $T \succeq 0$.
We believe that, with minor modifications to the proof, the
results of \cite{knowlesyin} may be further extended to $T$ having
negative eigenvalues, but we will not pursue this extension here.
\end{remark}

\subsection{Strategy of proof}
Our proof generalizes the resolvent comparison argument of
\cite{leeschnelli} for the largest eigenvalue.
Let $E_*$ denote an edge of the deterministic
spectral support of $\hSigma$. (We define this formally in
Section \ref{sec:model}.) We will consider
\[\hSigma^{(L)}=X'T^{(L)}X\]
for a different matrix $T^{(L)}$, and
compare the eigenvalue behavior of $\hSigma$ near $E_*$ with that
of $\hSigma^{(L)}$ near an edge $E_*^{(L)}$.

In \cite{leeschnelli}, $E_*$ is the rightmost edge of support.
The comparison between $T$
and $T^{(L)}$ is achieved by a continuous interpolation over $l \in
[0,L]$, where $T^{(0)}=T$ and each $T^{(l)}$ has diagonal entries
$\{t_\a^{(l)}:\a=1,\ldots,M\}$ given by
\begin{equation}\label{eq:leeschnelliflow}
(t_\a^{(l)})^{-1}=e^{-l}(t_\a^{(0)})^{-1}+(1-e^{-l}).
\end{equation}
(See \cite[Eq.\ (6.1)]{leeschnelli}.) 
Taking $L=\infty$, $T^{(\infty)}$ is a multiple of the identity,
and Tracy-Widom fluctuations are known for $\hSigma^{(\infty)}$. Along this
interpolation, the edge $E_*^{(l)}$ evolves continuously. Defining a smooth
resolvent approximation
\begin{equation}\label{eq:roughresolventapprox}
    \P\left[\hSigma^{(l)} \text{ has no eigenvalues in }
    E_*^{(l)}+[s_1,s_2]\right]
\approx \E\left[K(\X^{(l)}(s_1,s_2))\right],
\end{equation}
\cite{leeschnelli} establishes the bound
\begin{equation}\label{eq:leeschnellibound}
\big|\tfrac{d}{dl}\E\big[K(\X^{(l)}(s_1,s_2))\big]\big| \leq N^{-1/3+\eps}
\end{equation}
for a small constant $\eps>0$ and $s_1,s_2$ on the $N^{-2/3}$ scale.
This is applied to compare the
probability in (\ref{eq:roughresolventapprox}) for $l=0$ and
$l=\infty$.

We extend this argument by showing that 
the continuous interpolation in (\ref{eq:leeschnelliflow}) may be replaced by a
discrete Lindeberg sequence
$T^{(0)},T^{(1)},\ldots,T^{(L)}$ for an integer $L \leq O(N)$,
swapping one diagonal entry of $T$ at a time.
Letting $E_*$ be any regular edge of $\hSigma$,
each matrix $\hSigma^{(l)}
\equiv X'T^{(l)}X$ will have a corresponding edge $E_*^{(l)}$ such that
\begin{equation}\label{eq:roughEclose}
|E_*^{(l+1)}-E_*^{(l)}| \leq O(1/N).
\end{equation}
Each of these $L$ discrete steps may be thought of as
corresponding to a time interval $\Delta l=O(N^{-1})$ in the continuous
interpolation (\ref{eq:leeschnelliflow}).
We show that the above conditions are sufficient
to establish a discrete analogue of (\ref{eq:leeschnellibound}),
\begin{equation}\label{eq:roughresolventcompare}
\left|\E\left[K(\X^{(l+1)}(s_1,s_2))\right]
-\E\left[K(\X^{(l)}(s_1,s_2))\right]\right| \leq N^{-4/3+\eps}.
\end{equation}
As $L \leq O(N)$, summing over $l=0,\ldots,L-1$ establishes the desired
comparison between $T^{(0)}$ and $T^{(L)}$.

In contrast to the continuous flow (\ref{eq:leeschnelliflow}),
our swapping sequence is well-defined even for negative $t_\a^{(0)}$.
Furthermore, by swapping the diagonal entries of $T$ from one support
interval to another without continuously evolving them between such intervals,
we may preserve an interior edge $E_*$ even as the other
intervals of support disappear.

Section \ref{sec:tools} reviews prerequisite proof ingredients.
Section \ref{sec:lindeberg} constructs the interpolating sequence.
Finally, Section \ref{sec:resolventcompare} establishes
(\ref{eq:roughresolventcompare}). The main step of Section
\ref{sec:resolventcompare} is to generalize the
``decoupling lemma'' of \cite[Lemma 6.2]{leeschnelli} to a setting involving
two different resolvents $G$ and $\vG$ corresponding to
$T \equiv T^{(l)}$ and $\vT \equiv T^{(l+1)}$.

\subsection*{Acknowledgments}
We are indebted to geneticist Mark Blows for
asking the question about Tracy-Widom for random effects models that led
to this paper, and for many stimulating discussions. We would like to also thank
Kevin Schnelli for helpful conversations about \cite{leeschnelli}.
ZF was supported
in part by a Hertz Foundation Fellowship and an NDSEG Fellowship (DoD AFOSR 32
CFR 168a). IMJ is supported in part by NIH R01 EB001988 and NSF DMS 1407813.

\section{Model and results}\label{sec:model}
\subsection{Deterministic spectral law}\label{subsec:support}
Let $T=\diag(t_1,\ldots,t_M) \in \R^{M \times M}$ be a deterministic diagonal
matrix, whose
diagonal values $t_1,\ldots,t_M$ may be positive, negative, or zero. Let $X \in
\R^{M \times N}$ be a random matrix with independent entries of
mean 0 and variance $1/N$. We study the matrix
\[\hSigma=X'TX\]
in the limit as $N,M \to \infty$ proportionally. In this limit, the
empirical spectrum of $\hSigma$ is well-approximated by a
deterministic law $\mu_0$.\footnote{We define $\mu_0$ as an $N$-dependent law
depending directly on $M/N$ and $T$, rather than assuming that
$M/N$ and the spectrum of $T$ converge to certain limiting quantities.}
We review in this section the definition of $\mu_0$ and its relevant properties.

When $T=\Id$, $\mu_0$ is the Marcenko-Pastur law
\cite{marcenkopastur}. More generally, the law $\mu_0$ may be defined by
a fixed-point equation in its Stieltjes transform:
For each $z \in \C^+$, there is a unique value $m_0(z) \in \C^+$ which satisfies
\begin{equation}\label{eq:MPdiag}
z=-\frac{1}{m_0(z)}+\frac{1}{N}\sum_{\a=1}^M \frac{t_\a}{1+t_\a m_0(z)}.
\end{equation}
This is oftentimes called the Marcenko-Pastur equation, and it defines
implicitly the Stieltjes transform $m_0:\C^+ \to \C^+$
of a law $\mu_0$ on $\R$ \cite{marcenkopastur,silverstein,silversteinbai}.
This law $\mu_0$ admits a continuous density $f_0$ at each $x \in \R_*$, given
by
\begin{equation}\label{eq:f0}
f_0(x)=\lim_{z \in \C^+ \to x} \frac{1}{\pi} \Im m_0(z),
\end{equation}
where
\begin{equation}\label{eq:Rstar}
\R_*=\begin{cases} \R & \text{ if }\rank(T)>N\\
\R \setminus \{0\}& \text{ if }\rank(T) \leq N.
\end{cases}
\end{equation}
For $x \neq 0$, this is shown in \cite{silversteinchoi}; we extend this to 
$x=0$ when $\rank(T)>N$ in Appendix \ref{appendix:deterministic}.

This law $\mu_0$ may have multiple disjoint intervals of
support, and two such cases are depicted in Figures \ref{fig:example1} and
\ref{fig:example2} of Appendix \ref{appendix:deterministic}.
We denote the support of $\mu_0$ by $\supp(\mu_0)$,
and we call $E_* \in \R$ a right (or left) edge if it is a right (or left)
endpoint of one of the disjoint intervals constituting $\supp(\mu_0)$.
When 0 is a point mass of $\mu_0$, we do not consider it an edge.

The support intervals and edge locations of $\mu_0$ are described in a simple
way by (\ref{eq:MPdiag}), as explained in \cite{silversteinchoi,knowlesyin}:
Define $P=\{0\} \cup \{-t_\a^{-1}:t_\a \neq 0\}$,
and consider $\bar{\R}=\R \cup \{\infty\}$. Consider the formal inverse of
$m_0(z)$,
\begin{equation}\label{eq:z0}
z_0(m)=-\frac{1}{m}+\frac{1}{N}\sum_{\a=1}^M \frac{t_\a}{1+t_\a m},
\end{equation}
as a real-valued function on $\bar{\R} \setminus P$ with the convention
$z_0(\infty)=0$. Two examples are also plotted in Figures \ref{fig:example1}
and \ref{fig:example2} of Appendix \ref{appendix:deterministic}.
Then the local extrema of $z_0$ are in 1-to-1
correspondence with edges of $\mu_0$, with the scale of
square-root decay at each edge inversely related to the curvature of $z_0$.

\begin{proposition}\label{prop:edges}
Let $m_1,\ldots,m_n \in \bar{\R} \setminus P$
denote the local minima and local maxima\footnote{$m_* \in \bar{\R}
\setminus P$ is a local minimum of $z_0$ if
$z_0(m) \geq z_0(m_*)$ for all $m$ in a sufficiently small neighborhood of
$m_*$, with the convention that $m_*=\infty$ is a local minimum if
$z_0$ is positive over $(C,\infty) \cup (-\infty,-C)$ for some $C>0$.
Local maxima are defined similarly.} of $z_0$,
ordered such that $0>m_1>\ldots>m_k>-\infty$ and
$\infty \geq m_{k+1}>\ldots>m_n>0$. Let
$E_j=z_0(m_j)$ for each $j=1,\ldots,n$. Then:
\begin{enumerate}[(a)]
\item $\mu_0$ has exactly $n/2$ support intervals and $n$ edges,
which are given by $E_1,\ldots,E_n$.
\item $E_j$ is a right edge if $m_j$ is a local minimum, and a left edge if
$m_j$ is a local maximum.
\item The edges are ordered as $E_1>\ldots>E_k>E_{k+1}>\ldots>E_n$.
\item For each $E_j$ where $m_j \neq \infty$, we have
$E_j \in \R_*$ and $z_0''(m_j) \neq 0$.
Defining $\gamma_j=\sqrt{2/|z_0''(m_j)|}$, the density of $\mu_0$
satisfies $f_0(x) \sim
(\gamma_j/\pi)\sqrt{|E_j-x|}$ as $x \to E_j$ with $x \in \supp(\mu_0)$.
\end{enumerate}
\end{proposition}

\begin{definition}\label{def:mvaluescale}
For an edge $E_*$ of $\mu_0$, the local minimum/maximum $m_*$ of $z_0$ such
that $z_0(m_*)=E_*$ is its {\bf $\pmb{m}$-value}. The edge is {\bf soft}
if $m_* \neq \infty$ and {\bf hard} if $m_*=\infty$. For a soft edge,
$\gamma=\sqrt{2/|z_0''(m_*)|}$ is its associated {\bf scale}.
\end{definition}

The statements of Proposition \ref{prop:edges} are known for $T \succeq 0$,
and we describe the extension to
general $T$ in Appendix \ref{appendix:deterministic}.
When $T \succeq 0$, an edge at 0 is usually called hard and all other 
edges soft. Definition \ref{def:mvaluescale} extends this to general $T$:
A hard edge is always 0 and can occur when $\rank(T)=N$.
If $T$ has negative eigenvalues, then
a soft edge may also be 0 when $\rank(T)>N$. We thus distinguish hard edges by
the m-value rather than the edge location.

\subsection{Edge regularity and extremal eigenvalues}

We state our assumptions on $T$ and $X$. We also introduce
the notion of a regular edge, which is similar to the definitions
of \cite{elkaroui,hachemetal,knowlesyin} for $T \succeq 0$.

\begin{assumption}\label{assump:dT}
$T=\diag(t_1,\ldots,t_M) \in \R^{M \times M}$, where $|t_\a|<C$
for some constant $C>0$ and each $\a=1,\ldots,M$.
\end{assumption}
\begin{assumption}\label{assump:X}
$X \in \R^{M \times N}$ is random with independent entries.
For all indices $(\a,i)$, all $\ell \geq 1$,
and some constants $C,C_1,C_2,\ldots>0$,
\[C^{-1}<M/N<C, \quad \E[X_{\a i}]=0,\quad \E[X_{\a i}^2]=1/N,\quad
    \E[|\sqrt{N}X_{\a i}|^\ell] \leq C_\ell.\]
\end{assumption}

\begin{definition}\label{def:regular}
Let $E_* \in \R$ be a soft edge of $\mu_0$ with $m$-value $m_*$ and
scale $\gamma$. Then $E_*$ is {\bf regular} if there is a constant
$\tau>0$ such that $|m_*|<\tau^{-1}$, $\gamma<\tau^{-1}$, and
$|m_*+t_\a^{-1}|>\tau$ for all $\a \in \{1,\ldots,M\}$ such that $t_\a \neq 0$.
\end{definition}
A smaller constant $\tau$ indicates a weaker assumption.
We will say $E_*$ is $\tau$-regular if we wish to emphasize the role of
$\tau$. All subsequent constants may implicitly depend on $\tau$.

The existence of any regular edge will imply that the 
average value of $|t_\a|$ is of constant order; see
Proposition \ref{prop:basicregbounds}.
An interpretation of regularity
is the following, whose proof we defer to Appendix \ref{appendix:m0}.
\begin{proposition}\label{prop:regular}
Suppose Assumption \ref{assump:dT} holds and the edge $E_*$ is regular. Then
there exist constants $C,c,\delta>0$ (independent of $N$) such that
\begin{enumerate}[(a)]
\item (Separation) The interval $(E_*-\delta,E_*+\delta)$ belongs to
$\R_*$ and contains no edge other than $E_*$.
\item (Square-root decay) For all
$x \in \supp(\mu_0) \cap (E_*-\delta,E_*+\delta)$, the density $f_0$ of $\mu_0$
satisfies $c\sqrt{|E_*-x|} \leq f_0(x) \leq C\sqrt{|E_*-x|}$.
\end{enumerate}
\end{proposition}

We will study the extremal eigenvalue of $\hSigma$ at each regular edge.
This is well-defined by the following results establishing
closeness of eigenvalues of $\hSigma$ to the support of $\mu_0$.
Such results were shown in \cite{baisilverstein,knowlesyin} for $T \succ 0$,
and we discuss the extension to general $T$ in Appendix \ref{appendix:locallaw}.

\begin{theorem}[No eigenvalues outside support]\label{thm:sticktobulk}
Suppose Assumptions \ref{assump:dT} and \ref{assump:X} hold. Fix any
constants $\delta,D>0$. There exists a constant $N_0 \equiv N_0(\delta,D)$
such that for all $N \geq N_0$, with probability at least $1-N^{-D}$,
all eigenvalues of $\hSigma$ are within distance $\delta$ of $\supp(\mu_0)$.
\end{theorem}

\begin{theorem}[$N^{-2/3}$ concentration]\label{thm:regedgeconcentration}
Suppose Assumptions \ref{assump:dT} and \ref{assump:X} hold,
and $E_*$ is a regular right edge.
Then there exists a constant $\delta>0$ such that for any
$\eps,D>0$, some $N_0 \equiv N_0(\eps,D)$, and all $N \geq N_0$,
\[\P\Big[\text{no eigenvalue of } \hSigma \text{ belongs to }
[E_*+N^{-2/3+\eps},E_*+\delta]\Big]>1-N^{-D}.\]
The analogous statement holds if $E_*$ is a
regular left edge, with no eigenvalue of $\hSigma$ belonging to
$[E_*-\delta,E_*-N^{-2/3+\eps}]$.
\end{theorem}

\subsection{Tracy-Widom fluctuations}\label{subsec:TW}
The following is our main result.
\begin{theorem}\label{thm:TW}
Let $\hSigma=X'TX$.
Suppose that Assumptions \ref{assump:dT} and \ref{assump:X} hold for $T$ and
$X$, and that $E_*$ is a $\tau$-regular edge of the law
$\mu_0$. Let $E_*$ have scale $\gamma$ as defined in
Definition \ref{def:mvaluescale}. Then
there exists a $\tau$-dependent constant $\delta>0$ such that as
$N,M \to \infty$,
\begin{enumerate}[(a)]
\item For $E_*$ a right edge and $\lambda_{\max}$ the largest eigenvalue
of $\hSigma$ in $E_*+[-\delta,\delta]$,
\[(\gamma N)^{2/3}(\lambda_{\max}-E_*) \overset{L}{\to} \mu_{TW}.\]
\item For $E_*$ a left edge and $\lambda_{\min}$ the smallest eigenvalue
of $\hSigma$ in $E_*+[-\delta,\delta]$,
\[(\gamma N)^{2/3}(E_*-\lambda_{\min}) \overset{L}{\to} \mu_{TW}.\]
\end{enumerate}
\end{theorem}

Here, $\mu_{TW}$ is the GOE Tracy-Widom law.
The notation $\overset{L}{\to}$ indicates convergence in law. As $E_*$
is $N$-dependent, let us clarify that this means
\[\Big|\P[(\gamma N)^{2/3}(\lambda_{\max}-E_*) \leq x]-\mu_{TW}((-\infty,x])
\Big| \leq o(1)\]
for any fixed $x \in \R$, where $E_*$ is any (deterministic) choice of
$\tau$-regular edge, and $o(1)$ denotes a term vanishing as $N,M \to \infty$ and
depending only on $x$, $\tau$, and the constants
in Assumptions \ref{assump:dT} and \ref{assump:X}.

When $T \succeq 0$, the above result holds also for
the sample covariance matrix with the same values of
$E_*$ and $\gamma$, since this has the same eigenvalues as
$\hSigma$ except for a set of $|N-M|$ zeros.
\begin{corollary}
Under the conditions of Theorem \ref{thm:TW}, suppose $T \succeq 0$,
and let $\tilde{\Sigma}=T^{1/2}XX'T^{1/2}$. Then Theorem \ref{thm:TW} holds
also for $\tilde{\Sigma}$.
\end{corollary}

When $T=\Id$, the equation $0=z_0'(m_*)$ may be solved explicitly to yield
\[m_*=-\sqrt{N}/(\sqrt{N}\pm \sqrt{M}),
\qquad E_*=(\sqrt{N}\pm \sqrt{M})^2/N\]
\[(\gamma N)^{-2/3}=
\frac{|\sqrt{N}\pm \sqrt{M}|}{N}\left|\frac{1}{\sqrt{M}}\pm \frac{1}{\sqrt{N}}
\right|^{1/3}\]
for the upper and lower edges.
These centering and scaling constants are the same
as those of \cite{soshnikov,peche,feldheimsodin} and differ from those of
\cite{johnstone,ma} in small $O(1)$ adjustments to $N$ and $M$.
These adjustments do not affect the
validity of Theorem \ref{thm:TW}, although the proper adjustments are shown
in \cite{ma} to lead to an improved second-order rate of convergence.

\subsection{Application to linear mixed models}\label{subsec:MANOVA}
Consider $Y \in \R^{n \times p}$ representing $p$
traits in $n$ samples, modeled by a Gaussian random effects linear model
\begin{equation}\label{eq:MANOVAmodel}
Y = U_1\alpha_1+\ldots+U_k\alpha_k.
\end{equation}
Each random effect matrix $\alpha_r \in \R^{m_r \times p}$ has independent rows
with distribution $\N(0,\Sigma_r)$.
The deterministic incidence matrix $U_r \in \R^{n \times m_r}$
determines how the
random effect contributes to the observations $Y$. For simplicity, we omit
here possible additional fixed effects, and we present an example with a fixed
mean effect in Example \ref{ex:oneway} of Appendix \ref{appendix:testing}.

In many examples,
a canonical unbiased MANOVA estimator exists for each covariance
$\Sigma_r$ and takes the form (\ref{eq:MANOVA1}),
where $B \equiv B_r \in \R^{n \times n}$ is a symmetric matrix 
that is constructed based on $U_1,\ldots,U_k$. Spectral properties of MANOVA
estimators in the regime $n,p,m_1,\ldots,m_k \to
\infty$ were studied in \cite{fanjohnstone,fanjohnstonespikes}, which contain
additional discussion and examples.

Theorem \ref{thm:TW} provides the basis for an asymptotic test of the global
sphericity null hypothesis
\begin{equation}\label{eq:nullhypothesis}
    H_0:\Sigma_r=\sigma_r^2\Id \text{ for every } r=1,\ldots,k
\end{equation}
in this model, based on the largest observed eigenvalue of $\hSigma$.
While this test may be performed using any matrix $B$ in (\ref{eq:MANOVA1}),
to yield power against non-isotropic alternatives for a
particular covariance $\Sigma_r$, we suggest choosing $B \equiv B_r$ such that $\hSigma$
is the MANOVA estimator for $\Sigma_r$. Under $H_0$, let us set $N=p$ and write
$\alpha_r=\sqrt{N}\sigma_r X_r$
where $X_r \in \R^{m_r \times N}$ has independent $\N(0,1/N)$ entries. Defining
$M=m_1+\ldots+m_k$, $F_{rs}=N\sigma_r\sigma_s U_r'BU_s \in \R^{m_r \times m_s}$,
and
\begin{equation}\label{eq:F}
X=\begin{pmatrix} X_1 \\ \vdots \\ X_k \end{pmatrix} \in \R^{M \times N},
\qquad 
    F=\begin{pmatrix} F_{11} & \cdots & F_{1k} \\ \vdots & \ddots & \vdots \\
F_{k1} & \cdots & F_{kk} \end{pmatrix} \in \R^{M \times M},
\end{equation}
the MANOVA estimator (\ref{eq:MANOVA1}) takes the form
\[\hSigma=Y'BY=\sum_{r,s=1}^k \alpha_r'U_r'BU_s\alpha_s=X'FX.\]
Rotational invariance of $X$ implies
$\hSigma\overset{L}{=}X'TX$ where $T=\diag(t_1,\ldots,t_M)$ is
the diagonal matrix of eigenvalues of $F$. Under mild conditions
for the model, as discussed in \cite{fanjohnstone,fanjohnstonespikes},
Assumptions \ref{assump:dT} and \ref{assump:X} hold for $\hSigma$.

In detail, a test based on the largest
eigenvalue of $\hSigma$ may be performed as follows:
\begin{enumerate}[1.]
\item Construct the above matrix $F$. Let $t_1,\ldots,t_M$ be its eigenvalues.
\item Plot the function $z_0(m)$ from (\ref{eq:z0})
over $m \in \R$, and locate the value $m_*$ closest to 0 such that
$z_0'(m_*)=0$ and $m_*<0$.
\item Compute the center and scale
$E_*=z_0(m_*)$ and $\gamma=\sqrt{2/z_0''(m_*)}$.
\item Compare $(\gamma N)^{2/3}(\lambda_{\max}-E_*)$ to the
GOE Tracy-Widom law $\mu_{TW}$.
\end{enumerate}

Asymptotic validity of this test requires regularity of the rightmost edge
of $\mu_0$. We provide a sufficient condition for this in Proposition
\ref{prop:balancedregular}, which encompasses many
balanced classification designs. More generally, edge regularity
is quantified by the separation between $m_*$ and the poles of $z_0(m)$,
and by the curvature of $z_0(m)$ at $m_*$. One may visually inspect the plot of
$z_0(m)$ for a qualitative diagnostic check of this assumption.

Constructing $F$ and computing $z_0(m)$ requires knowledge of
$\sigma_1^2,\ldots,\sigma_k^2$. If any $\sigma_r^2$ is unknown, it may be
replaced by the $1/n$-consistent estimate
\[\hsigma_r^2=p^{-1}\Tr \hSigma_r,\]
where $\hSigma_r$ is an unbiased MANOVA estimator for $\Sigma_r$.
We verify this in Appendix \ref{appendix:testing}, where we also
discuss the concrete example of the balanced one-way design, and provide
numerical simulation results to assess approximation accuracy in finite samples.

\section{Preliminaries and tools}\label{sec:tools}
The remainder of this paper is devoted to the proof of Theorem \ref{thm:TW}.
We collect here some tools for the proof.

\subsection{Notation}
We denote
$\I_M=\{1,\ldots,M\}$, $\I_N=\{1,\ldots,N\}$, and
$\I \equiv \I_N \sqcup \I_M$ considering $\I_N$ and $\I_M$ as disjoint.
We index rows and columns of
$\C^{(N+M) \times (N+M)}$ by $\I$ and consistently use lower-case
Roman letters $i,j$, etc.\ for indices in $\I_N$, Greek letters
$\a,\b$, etc.\ for indices in $\I_M$, and upper-case Roman letters $A,B$,
etc.\ for general indices in $\I$.

We typically write $z=E+i\eta$ where $E=\Re z$ and $\eta=\Im z$.
$\C^+$ and $\overline{\C^+}$ denote the open and closed upper-half complex
planes. $X'$ denotes the transpose of a matrix $X$. $\|\v\|$ denotes the
Euclidean norm for vectors, and $\|X\|=\sup_{\v:\|\v\|=1} \|X\v\|$ the
operator norm for matrices. $C,c>0$ denote constants changing from instance to 
instance and may depend on $\tau$ in the context of a regular edge.
$a_N \asymp b_N$ means $cb_N \leq a_N \leq Ca_N$.

\subsection{Stochastic domination}
For a non-negative scalar $\Psi$ (either random or deterministic), we write
\[\xi \prec \Psi \qquad \text{and} \qquad \xi=\O(\Psi)\]
if, for any constants $\eps,D>0$ and all $N \geq N_0(\eps,D)$,
\begin{equation}\label{eq:domination}
\P\left[|\xi|>N^\eps \Psi \right]<N^{-D}.
\end{equation}
Here, $N_0(\eps,D)$ may depend on $\eps,D$, and quantities
which are explicitly constant in the context of the statement.

Several known elementary properties of stochastic domination pertaining to
union bounds and expectations are
reviewed in Appendix \ref{appendix:tools}.

\subsection{Edge regularity}\label{subsec:regularity}
The following are consequences of edge regularity. Similar
properties were established for $T \succeq 0$ in
\cite{baopanzhou,knowlesyin}, and we defer proofs for general $T$ to Appendix
\ref{appendix:deterministic}.

\begin{proposition}\label{prop:basicregbounds}
Suppose Assumption \ref{assump:dT} holds, and $E_*$ is a regular edge with
$m$-value $m_*$ and scale $\gamma$. Then there exist constants $C,c>0$ such that
for all $\a=1,\ldots,M$,
\[c<|m_*|<C,\qquad c<\gamma<C,\qquad |E_*|<C, \qquad
|1+t_\a m_*|>c.\]
Furthermore, if any regular edge $E_*$ exists, then $T$ satisfies
\begin{equation}\label{eq:nondegenerate}
|\{\a \in \{1,\ldots,M\}:\,|t_\a|>c\}|>cM
\end{equation}
for a constant $c>0$, and if $T \succeq 0$, then also $E_*>c>0$.
\end{proposition}

\begin{proposition}\label{prop:z0secondderivative}
Suppose Assumption \ref{assump:dT} holds and $E_*$ is a regular edge with
$m$-value $m_*$. Then there exist constants $c,\delta>0$ such that for all $m
\in (m_*-\delta,m_*+\delta)$, if $E_*$ is a right edge then
$z_0''(m)>c$, and if $E_*$ is a left edge then $z_0''(m)<-c$.
\end{proposition}

\begin{proposition}\label{prop:m0estimates}
Suppose Assumption \ref{assump:dT} holds and $E_*$ is a regular edge.
Then there exist constants $C,c,\delta>0$ such that the following hold:
Define
\[\bD_0=\{z \in \C^+:\Re z \in (E_*-\delta,E_*+\delta),\,\Im z
\in (0,1]\}.\]
Then for all $z \in \bD_0$ and $\alpha \in \{1,\ldots,M\}$,
\[c<|m_0(z)|<C, \qquad c<|1+t_\a m_0(z)|<C.\]
Furthermore, for all $z \in \bD_0$, denoting $z=E+i\eta$ and $\kappa=|E-E_*|$,
\[c\sqrt{\kappa+\eta} \leq |m_0(z)-m_*| \leq C\sqrt{\kappa+\eta},\qquad
cf(z) \leq \Im m_0(z) \leq Cf(z)\]
where
\[f(z)=\begin{cases}
\sqrt{\kappa+\eta} & \text{ if } E \in \supp(\mu_0) \\
\frac{\eta}{\sqrt{\kappa+\eta}} & \text{ if } E \notin \supp(\mu_0).
\end{cases}\]
\end{proposition}

\subsection{Resolvent bounds and identities}
For $z \in \C^+$, denote the resolvent and Stieltjes transform of
$\hSigma$ by
\begin{equation}\label{eq:mN}
G_N(z)=(\hSigma-z\Id)^{-1} \in \C^{N \times N},\qquad
m_N(z)=N^{-1}\Tr G_N(z).
\end{equation}
These satisfy the basic properties
\begin{align}
    |m_N(z)| \leq 1/\eta,&\quad |G_{ij}(z)| \leq 1/\eta,\label{eq:mNbound}\\
    |m_N(z)-m_N(z')| \leq |z-z'|/\eta^2,&\quad
|G_{ij}(z)-G_{ij}(z')| \leq |z-z'|/\eta^2.\label{eq:mNlipschitz}
\end{align}

As in \cite{leeschnelli,knowlesyin}, define the linearized resolvent $G(z)$ by
\[H(z)=\begin{pmatrix} -z\Id & X' \\ X & -T^{-1} \end{pmatrix}
\in \C^{(N+M) \times (N+M)},\qquad
G(z)=H(z)^{-1}.\]
The Schur-complement formula yields the alternative form
\begin{equation}\label{eq:Galt}
G(z)=\begin{pmatrix} G_N(z) & G_N(z)X'T \\ TXG_N(z) & TXG_N(z)X'T-T
\end{pmatrix},
\end{equation}
which is understood as the definition of $G(z)$ when $T$ is not invertible.
We will omit the argument $z$ in $m_0,m_N,G_N,G$ when the meaning is
clear.

For any $A \in \I$, define $H^{(A)}$ as the submatrix of $H$ with row and
column $A$ removed, and define $G^{(A)}=(H^{(A)})^{-1}$.
When $T$ is not invertible, $G^{(A)}$ is defined by the alternative form
analogous to (\ref{eq:Galt}). We index $G^{(A)}$ by $\I \setminus \{A\}$.

Note that $G$ and $G^{(A)}$ are symmetric, in the sense $G'=G$ and
$(G^{(A)})'=G^{(A)}$ without complex conjugation.
The entries of $G$ and $G^{(A)}$ are
related by the following Schur-complement identities from
\cite[Lemma 4.4]{knowlesyin}.

\begin{lemma}[Resolvent identities]\label{lemma:resolventidentities}
Fix $z \in \C^+$.
\begin{enumerate}[(a)]
\item For any $i \in \I_N$ and $\a \in \I_M$,
\[G_{ii}=-\,\frac{1}{z+\sum_{\a,\b \in \I_M} G^{(i)}_{\a\b}X_{\a i}X_{\b i}},
\quad
G_{\alpha\alpha}=-\,\frac{t_\alpha}{1+t_\alpha \sum_{i,j \in \I_N}
G^{(\alpha)}_{ij}X_{\a i}X_{\a j}}.\]
\item For any $i \neq j \in \I_N$ and
$\alpha \neq \beta \in \I_M$,
\[G_{ij}=-G_{ii}\sum_{\b \in \I_M} G^{(i)}_{\b j}X_{\b i},
\quad G_{\a\b}=-G_{\a\a}\sum_{j \in \I_N} G^{(\a)}_{j\b}X_{\a j}.\]
For any $\a \in \I_M$ and $i \in \I_N$,
\[G_{i\a}=-G_{ii}\sum_{\b \in \I_M} G^{(i)}_{\b \a}X_{\b i}
=-G_{\a\a}\sum_{j \in \I_N} G^{(\a)}_{ij}X_{\a j} .\]
\item For any $A,B,C \in \I$ with $A \neq C$ and $B \neq C$,
\[G_{AB}^{(C)}=G_{AB}-\frac{G_{AC}G_{CB}}{G_{CC}}.\]
\end{enumerate}
\end{lemma}

\subsection{Local law}
We will require a local law for entries of $G(z)$, when $z \in \C^+$ close to
a regular edge $E_*$. This was established
in \cite{knowlesyin} for $T \succeq 0$, and we discuss the extension to general
$T$ in Appendix \ref{appendix:locallaw}.

\begin{theorem}[Entrywise local law at regular edges]\label{thm:locallaw}
Suppose Assumptions \ref{assump:dT} and \ref{assump:X} hold, and $E_*$
is a $\tau$-regular edge. Then for a $\tau$-dependent constant
$\delta>0$, the following holds: Fix any constant $a>0$ and define
\begin{equation}\label{eq:bD}
\bD=\{z \in \C^+:\Re z \in (E_*-\delta,E_*+\delta),
\;\Im z \in [N^{-1+a},1]\}.
\end{equation}
For $A \in \I$, denote $t_A=1$ if $A \in \I_N$ and $t_A=t_\a$ if $A=\a \in
\I_M$. Set
\begin{equation}\label{eq:Pi0}
\Pi(z)=\begin{pmatrix} m_0(z)\Id & 0 \\ 0 & -T(\Id+m_0(z)T)^{-1}
\end{pmatrix} \in \C^{(N+M) \times (N+M)}.
\end{equation}
Then for all $z \equiv E+i\eta \in \bD$ and $A,B \in \I$,
\begin{equation}\label{eq:anisotropiclaw}
(G_{AB}(z)-\Pi_{AB}(z))\Big/(t_At_B) \prec 
\sqrt{(\Im m_0(z))/(N\eta)}+1/(N\eta),
\end{equation}
and also
\[m_N(z)-m_0(z) \prec 1/(N\eta).\]
\end{theorem}

\begin{corollary}\label{cor:locallawunionbound}
Under the assumptions of Theorem \ref{thm:locallaw}, for any $\eps,D>0$ and all
$N \geq N_0(\eps,D)$, with probability at least $1-N^{-D}$,
\[|G_{AB}(z)-\Pi_{AB}(z)|\Big/|t_A t_B| \leq 
N^\eps \left(\sqrt{\Im m_0(z)/(N\eta)}+1/(N\eta)\right)\]
holds simultaneously for every $z \in \bD$ and $A,B \in \I$.
\end{corollary}

Here, $N_0(\eps,D)$ may depend on the constant $a$ defining $\bD$.
It is verified from (\ref{eq:Galt}) that the quantity on the left of
(\ref{eq:anisotropiclaw}) is alternatively written as
\begin{equation}\label{eq:locallawalt}
\frac{G_{AB}-\Pi_{AB}}{t_A t_B}
=\begin{pmatrix} G_N-m_0\Id & G_NX' \\ XG_N
& XG_NX'-m_0(\Id+m_0T)^{-1} \end{pmatrix}_{AB}.
\end{equation}
This is understood as its definition when either $t_A$ and/or $t_B$ is 0.

\subsection{Resolvent approximation}\label{subsec:resolventapprox}

Fix a regular edge $E_*$. For $s_1,s_2 \in \R$ and $\eta>0$, define
\begin{equation}\label{eq:X}
\X(s_1,s_2,\eta)=N \int_{E_*+s_1}^{E_*+s_2} \Im m_N(y+i\eta)dy.
\end{equation}
For $\eta$ much smaller than $N^{-2/3}$ and $s_1,s_2$ on the $N^{-2/3}$ scale,
we expect
\[\#(E_*+s_1,E_*+s_2) \approx \pi^{-1} \X(s_1,s_2,\eta)\]
where the left side denotes the number of eigenvalues of $\hSigma$
in this interval. The following is a version of this approximation,
similar to \cite[Corollary 6.2]{erdosyauyin}.
We provide a self-contained proof in Appendix \ref{appendix:tools}.

\begin{lemma}\label{lemma:resolventapprox}
Suppose Assumptions \ref{assump:dT} and \ref{assump:X} hold, and 
$E_*$ is a regular right edge. Let $K:\R \to [0,1]$
be such that $K(x)=1$ for all $x \leq 1/3$ and $K(x)=0$
for all $x \geq 2/3$.
Then for sufficiently small constants $\delta,\eps>0$:

Let $\lambda_{\max}$ be the maximum eigenvalue of $\hSigma$ in
$(E_*-\delta,E_*+\delta)$. Set $s_+=N^{-2/3+\eps}$, $l=N^{-2/3-\eps}$,
and $\eta=N^{-2/3-9\eps}$. For any $D>0$, all $N \geq N_0(\eps,D)$,
and all $s \in [-s_+,s_+]$,
\begin{align*}
\E\left[K(\pi^{-1}\X(s-l,s_+,\eta))\right]
-N^{-D} &\leq \P\left[\lambda_{\max} \leq E_*+s\right]\\
&\leq \E\left[K(\pi^{-1}\X(s+l,s_+,\eta))\right]+N^{-D}.
\end{align*}
\end{lemma}

\section{The interpolating sequence}\label{sec:lindeberg}
In this section, we construct the interpolating sequence
$T^{(0)},\ldots,T^{(L)}$ described in the introduction.
We consider only the
case of a right edge; this is without loss of generality, as the edge can have
arbitrary sign and we may take the reflection $T \mapsto -T$. 
For each pair $T \equiv T^{(l)}$ and $\vT \equiv T^{(l+1)}$,
the following definition captures the
relevant property that will be needed in the subsequent computation. 

\begin{definition}
Let $T,\vT \in \R^{M \times M}$ be two diagonal matrices
satisfying Assumption \ref{assump:dT}. Let $E_*$ be a right
edge of the law $\mu_0$ defined by $T$, and let $\vE_*$ be a
right edge of $\check{\mu}_0$ defined by $\vT$. $(T,E_*)$ and
$(\vT,\vE_*)$ are {\bf swappable} if, for a constant $\phi>0$, both of the
following hold.
\begin{itemize}
\item Letting $t_\a,\vt_\a$ be the diagonal entries of $T,\vT$, we have
$\sum_\alpha |t_\alpha-\vt_\alpha|<\phi$.
\item The $m$-values $m,\vm_*$ of $E_*,\vE_*$ satisfy $|m_*-\vm_*|<\phi/N$.
\end{itemize}
\end{definition}
We say that $(T,E_*)$ and $(\vT,\vE_*)$ are $\phi$-swappable if we wish to
emphasize the role of $\phi$. All subsequent constants may implicitly depend
on $\phi$.

One method to construct a swappable pair $T,\vT$ is to ensure
$|t_\a-\vt_\a| \leq \phi/M$ for every $\a=1,\ldots,M$, and such a condition
would hold for each pair $T^{(l)},T^{(l+1)}$ of a suitable discretization
of the continuous flow in \cite{leeschnelli}.
However, to study interior edges of the spectrum, we will instead consider
swappable pairs of a ``Lindeberg'' form where there is an $O(1)$ difference
between $t_\a$ and $\vt_\a$ for a single index $\a$.

We first establish some basic deterministic properties of a swappable pair,
including closeness of the edges $E_*,\vE_*$ as claimed in
(\ref{eq:roughEclose}).

\begin{lemma}\label{lemma:Escaleclose}
Suppose $T,\vT$ are diagonal matrices
satisfying Assumption \ref{assump:dT}, $E_*,\vE_*$ are regular
right edges, and $(T,E_*)$ and $(\vT,\vE_*)$ are swappable. Let $m_*,\gamma$ and
$\vm_*,\vg$ be the $m$-values and scales of $E_*,\vE_*$. Denote
$s_\a=(1+t_\a m_*)^{-1}$ and $\vs_\a=(1+\vt_\a \vm_*)^{-1}$.
Then there exists a constant $C>0$ such that all of the following hold:
\begin{enumerate}[(a)]
\item For all integers $i,j \geq 0$ satisfying $i+j \leq 4$,
\[\left|\frac{1}{N}\sum_{\a=1}^M t_\a^i s_\a^i \vt_\a^j\vs_\a^j
-\frac{1}{N}\sum_{\a=1}^M t_\a^{i+j}s_\a^{i+j}\right| \leq C/N.\]
\item (Closeness of edge location) $|E_*-\vE_*| \leq C/N$ and
\begin{equation}\label{eq:Eplusdiff}
\left|(E_*-\vE_*)-\frac{1}{N}
\sum_{\a=1}^M (t_\a-\vt_\a)s_\a \vs_\a\right|
\leq C/N^2.
\end{equation}
\item (Closeness of scale) $|\gamma-\vg| \leq C/N$.
\end{enumerate}
\end{lemma}
\begin{proof}
By Proposition \ref{prop:basicregbounds},
$|t_\a|,|s_\a|,\gamma<C$, $c<|m_*|<C$
and similarly for $\vt_\a,\vs_\a,\vm_*,\vg$.
From the definitions of $s_\a$ and $\vs_\a$, we verify
\begin{equation}\label{eq:tsdiff}
t_\a s_\a-\vt_\a \vs_\a=(t_\a-\vt_\a)s_\a\vs_\a+(\vm_*-m_*)t_\a s_\a \vt_\a
\vs_\a.
\end{equation}
Then, denoting
$A_{i,j}=N^{-1}\sum_\a t_\a^i s_\a^i \vt_\a^j \vs_\a^j$,
swappability implies
\[|A_{i,j}-A_{i+1,j-1}| \leq \frac{1}{N}\sum_{\a=1}^M
|t_\a^i s_\a^i \vt_\a^{j-1}\vs_\a^{j-1}||\vt_\a \vs_\a-t_\a s_\a| \leq C/N.\]
Iteratively applying this yields (a).
For (b), note by (\ref{eq:tsdiff}) that
\begin{align*}
E_*-\vE_*&=-\frac{1}{m_*}+\frac{1}{\vm_*}
+\frac{1}{N}\sum_{\a=1}^M (t_\a s_\a-\vt_\a \vs_\a)\\
&=(m_*-\vm_*)\left(\frac{1}{m_*\vm_*}-A_{1,1}\right)+\frac{1}{N}
\sum_{\a=1}^M (t_\a-\vt_\a)s_\a \vs_\a.
\end{align*}
Recall $0=z_0'(m_*)=m_*^{-2}-A_{2,0}$. Then
part (b) follows from the definition of swappability, together with
$|A_{1,1}-m_*^{-2}|=|A_{1,1}-A_{2,0}| \leq C/N$ and
$|m_*^{-2}-m_*^{-1}\vm_*^{-1}| \leq C/N$.
For (c), we have $\gamma^{-2}=z_0''(m_*)/2=-m_*^{-3}+A_{3,0}$.
Then (c) follows from
$|\gamma^{-2}-\vg^{-2}| \leq |m_*^{-3}-\vm_*^{-3}|+|A_{3,0}-A_{0,3}|
\leq C/N$.
\end{proof}

In the rest of this section, we prove the existence of an interpolating
sequence. Note that to ensure the final edge $E_*^{(L)}$ is not a hard edge at
0, we allow the final matrix $T^{(L)}$ to have two distinct values $\{0,t\}$.

\begin{lemma}\label{lemma:lindebergconstruction}
Suppose $T$ is diagonal and satisfies Assumption \ref{assump:dT},
and $E_*$ is a $\tau$-regular right edge with scale $\gamma=1$.
Then there exist $\tau$-dependent constants $C',\tau',\phi>0$,
a sequence of
diagonal matrices $T^{(0)},T^{(1)},\ldots,T^{(L)}$ in $\R^{M \times M}$ for
$L \leq 2M$, and a sequence of right edges
$E_*^{(0)},E_*^{(1)},\ldots,E_*^{(L)}$ of the corresponding laws $\mu_0^{(l)}$
defined by $T^{(l)}$, such that:
\begin{enumerate}[1.]
\item $T^{(0)}=T$ and $E_*^{(0)}=E_*$.
\item $T^{(L)}$ has at most two distinct diagonal entries 0 and $t$,
for some $t \in \R$.
\item Each $T^{(l)}$ satisfies Assumption \ref{assump:dT} with constant $C'$.
\item Each $E_*^{(l)}$ is $\tau'$-regular.
\item $(T^{(l)},E_*^{(l)})$ and $(T^{(l+1)},E_*^{(l+1)})$ are $\phi$-swappable
for each $l=0,\ldots,L-1$.
\item (Scaling) Each $E_*^{(l)}$ has associated scale $\gamma^{(l)}=1$.
\end{enumerate}
\end{lemma}

We first ignore the scaling property 6, and construct
$T^{(0)},\ldots,T^{(L)}$ and $E_*^{(0)},\ldots,E_*^{(L)}$
satisfying properties 1--5. We will use a Lindeberg swapping construction,
where each $T^{(l+1)}$ differs from $T^{(l)}$ in only one diagonal entry.
It is useful to write $z_0'$ and $z_0''$ as
\[z_0'(m)=\frac{1}{m^2}-\frac{1}{N}\sum_{\a:t_\a \neq 0}
\frac{1}{(m+t_\a^{-1})^2}, \qquad
z_0''(m)=-\frac{2}{m^3}+\frac{2}{N}\sum_{\a:t_\a \neq 0}
\frac{1}{(m+t_\a^{-1})^3},\]
and to think about swapping entries of $T$ as swapping or removing poles of
$z_0'$ and $z_0''$. In particular, for each fixed $m \in \R$, we can easily
deduce from the above whether a given swap increases or decreases 
$z_0'(m)$ and $z_0''(m)$.

Upon defining a swap $T \to \vT$, the identification of the new right edge
$\vE_*$ for $\vT$ uses the following continuity lemma.
\begin{lemma}\label{lemma:singleswap}
Suppose $T$ is a diagonal matrix satisfying Assumption \ref{assump:dT}, and
$E_*$ is a $\tau$-regular right edge with $m$-value $m_*$. Let $\vT$ be a
matrix that replaces a single diagonal entry $t_\a$ of $T$ by a value $\vt_\a$,
such that $|\vt_\a| \leq \|T\|$ and either $\vt_\a=0$ or
$|m_*+\vt_\a^{-1}|>\tau$.
Let $z_0,\vz_0$ denote the function (\ref{eq:z0}) defined by $T,\vT$.
Then there exist $\tau$-dependent constants $N_0,\phi>0$ such that whenever $N
\geq N_0$:
\begin{itemize}
\item $\vT$ has a right edge $\vE_*$ with $m$-value $\vm_*$ satisfying
$|m_*-\vm_*|<\phi/N$.
\item The interval between $m_*$ and $\vm_*$ does not contain any pole
of $z_0$ or $\vz_0$.
\item $\sign(m_*-\vm_*)=\sign(\vz_0'(m_*))$.
\end{itemize}
(We define $\sign(x)=1$ if $x>0$, $-1$ if $x<0$, and 0 if $x=0$.)
\end{lemma}
\begin{proof}
By Proposition \ref{prop:basicregbounds}, $|m_*|>\nu$ for a constant $\nu$.
Take $\delta<\min(\tau/2$, $\nu/2)$.
Then the given conditions for $\vt_\a$ imply
that $(m_*-\delta,m_*+\delta)$ does not
contain any pole of $z_0$ or $\vz_0$, and
\[|z_0'(m)-\vz_0'(m)|<C/N\]
for some $C>0$ and all $m \in (m_*-\delta,m_*+\delta)$. For sufficiently small
$\delta$, Proposition \ref{prop:z0secondderivative} also ensures
$z_0''(m)>c$ for all $m \in (m_*-\delta,m_*+\delta)$. If
$\vz_0'(m_*)<0=z_0'(m_*)$, this implies
$\vz_0$ must have a local minimum
in $(m_*,m_*+C/N)$, for a constant $C>0$ and all $N \geq N_0$.
Similarly, if $\vz_0'(m_*)>0$, then $\vz_0$ has a local
minimum in $(m_*-C/N,m_*)$, and if $\vz_0'(m_*)=0$, then $\vz_0$ has a local
minimum at $m_*$. The result follows from Proposition \ref{prop:edges} upon
setting $\vE_*=\vz_0(\vm_*)$.
\end{proof}

The basic idea for proving Lemma \ref{lemma:lindebergconstruction} is to
take a Lindeberg sequence $T^{(0)},\ldots,T^{(L)}$ and apply the above lemma for
each swap. We cannot do this naively for any Lindeberg sequence, because in
general if $E_*^{(l)}$ is $\tau_l$-regular, then the above lemma only
guarantees that $E_*^{(l+1)}$ is $\tau_{l+1}$-regular for
$\tau_{l+1}=\tau_l-C/N$ and a $\tau_l$-dependent constant $C>0$. 
Thus edge regularity, as well as the edge itself, may vanish after $O(N)$ swaps.

To circumvent this, we consider a specific construction of the Lindeberg
sequence, apply Lemma \ref{lemma:singleswap} along this sequence
to identify an edge $\vE_*$ for each successive $\vT$,
and use a separate argument to
show that $\vE_*$ must be $\tau'$-regular for a fixed constant $\tau'>0$.
Hence we may continue to apply Lemma \ref{lemma:singleswap} along the whole
sequence.

We consider separately the cases $m_*<0$ and $m_*>0$.

\begin{lemma}\label{lemma:lindebergmneg}
Suppose (the right edge) $E_*$ has $m$-value $m_*<0$. Then for some
$\tau$-dependent constant $N_0$, whenever $N \geq N_0$,
Lemma \ref{lemma:lindebergconstruction} holds without the scaling
condition, property 6.
\end{lemma}
\begin{proof}
We construct a Lindeberg sequence that first reflects about
$m_*$ each pole of $z_0$ to the right of $m_*$, and then replaces each pole by
the one closest to $m_*$.

Suppose, first, that there are $K_1$ non-zero diagonal entries $t_\a$
of $T$ (positive or negative) where $-t_\a^{-1}>m_*$. Consider
a sequence of matrices $T^{(0)}$, $T^{(1)}$, $\ldots$, $T^{(K_1)}$ where
$T^{(0)}=T$, and each $T^{(k+1)}$ replaces one such diagonal entry $t_\a$ of
$T^{(k)}$ by the value $\vt_\a$ such that $-\vt_\a^{-1}<m_*$ and
$|m_*+\vt_\a^{-1}|=|m_*+t_\a^{-1}|$. For each such swap $T \to \vT$,
we verify $|\vt_\a| \leq |t_\a| \leq \|T\|$, $\vz_0'(m_*)=z_0'(m_*)=0$,
and $\vz_0''(m_*)>z_0''(m_*)>0$. Thus we may take $\vm_*=m_*$ in Lemma
\ref{lemma:singleswap}, and the new edge $\vE_*=\vz_0(m_*)$ remains
$\tau$-regular for the same constant $\tau$.

All diagonal entries of $T^{(K_1)}$ are now nonnegative. Let
$t=\|T^{(K_1)}\|$ be the maximal such entry. By the above construction,
$-t^{-1}<m_*<0$. Since $E_*^{(K_1)}$ is $\tau$-regular,
(\ref{eq:nondegenerate}) implies $t>c$ for a constant $c>0$.
Let $K_2$ be the number of positive diagonal entries of $T^{(K_1)}$
strictly less than $t$, and consider a sequence
$T^{(K_1+1)},\ldots,T^{(K_1+K_2)}$
where each $T^{(k+1)}$ replaces one such diagonal entry in $T^{(k)}$ by $t$.
Applying Lemma \ref{lemma:singleswap} to each such swap $T \to \vT$,
we verify $\vz_0'(m_*)<z_0(m_*)=0$, so $m_*<\vm_*<0$.
Then $|\vm_*|<|m_*|$ and $\min_\a |\vm_*+\vt_\a^{-1}|>\min_\a |m_*+t_\a^{-1}|$.
Also $\vm_*+\vt_\a^{-1}>0$ for all $\vt_\a \neq 0$, so
$\vz_0''(\vm_*)>-2/\vm_*^3>2t^3$. This verifies $\vE_*=\vz_0(\vm_*)$ is
$\tau'$-regular for a fixed constant $\tau'>0$.
(We may take any $\tau'<\min(\tau,t^{3/2})$.)

The total number of swaps $L=K_1+K_2$ is at most $2M$, and all diagonal entries
of $T^{(L)}$ belong to $\{0,t\}$. This concludes the proof, with property 5
verified by Lemma \ref{lemma:singleswap}.
\end{proof}

\begin{lemma}\label{lemma:lindebergmpos}
Lemma \ref{lemma:lindebergmneg} holds also when $E_*$ has $m$-value $m_*>0$.
\end{lemma}
\begin{proof}
Proposition \ref{prop:edges} implies $m_*$ is a local minimum of $z_0$.
The interval $(0,m_*)$ must contain a pole of $z_0$---otherwise,
by the boundary condition of $z_0$ at 0,
there would exist a local maximum $m$ of $z_0$ in $(0,m_*)$ satisfying
$z_0(m)>z_0(m_*)$, which would contradict the edge ordering
in Proposition \ref{prop:edges}(c). Let
$-t^{-1}$ be the pole in $(0,m_*)$ closest to $m_*$. Note that $t<0$ and
$|t|>|m_*|^{-1}>\tau$. We construct a
Lindeberg sequence that first replaces a small but constant fraction of entries
of $T$ by $t$, then replaces all non-zero $t_\a>t$ by 0, and
finally replaces all $t_\a<t$ by 0.

First, fix a small constant $c_0>0$, let $K_1=\lfloor c_0 M \rfloor$,
and consider a sequence of matrices $T^{(0)},T^{(1)},\ldots,T^{(K_1)}$ where
$T^{(0)}=T$ and each $T^{(k+1)}$ replaces a different (arbitrary) diagonal
entry of $T^{(k)}$ by $t$. For $c_0$ sufficiently small, it is easy to check
that we may apply Lemma \ref{lemma:singleswap} to identify an edge $E_*^{(k)}$
for each $k=1,\ldots,K_1$, such that each $E_*^{(k)}$ remains $\tau/2$-regular.

$T^{(K_1)}$ now has at least $c_0M$ diagonal entries equal to $t$.
By the condition in Lemma \ref{lemma:singleswap} that the swap $m_* \to \vm_*$
does not cross any pole of $z_0$ or $\vz_0$, we have that $-t^{-1}$ is still
the pole in $(0,m_*^{(K_1)})$ closest to $m_*^{(K_1)}$.
Let $K_2$ be the number of non-zero diagonal entries $t_\a$
of $T^{(K_1)}$ (positive or negative) such that $t_\a>t$.
Consider a sequence $T^{(K_1+1)},\ldots,T^{(K_1+K_2)}$ where each $T^{(k+1)}$
replaces one such entry in $T^{(k)}$ by 0. Note that each swap $T \to \vT$ of
this sequence satisfies $\vz_0'(m)>z_0'(m)$ at every value $m$.
Then in particular, $\vz_0'(m_*)>z_0'(m_*)=0$, so
Lemma \ref{lemma:singleswap} yields a new edge $\vE_*$ for which
$-t^{-1}<\vm_*<m_*$. For every
$\alpha$ such that $-\vt_\alpha^{-1}>-t^{-1}$, we have
$-\vt_\alpha^{-1}>m_*$ because $-t^{-1}$ is the closest pole to the left of
$m_*$. Then, since $\vm_*<m_*$, this shows
$\min_{\a:-\vt_\a^{-1}>-t^{-1}} |\vm_*+\vt_\a^{-1}|
>\min_{\a:-t_\a^{-1}>-t^{-1}} |m_*+t_\a^{-1}|>\tau/2$.
The conditions $\vm_*>|t|^{-1}>c$ and
\[0=\vz_0'(\vm_*)\leq\frac{1}{\vm_*^2}
-\frac{c_0M}{N}\frac{1}{(\vm_*+t^{-1})^2}\]
ensure that $\vm_*+t^{-1}>\nu$ for a constant $\nu>0$, and hence
$\min_\a |\vm_*+\vt_\a^{-1}|>\min(\nu,\tau/2)$ for the minimum over all
$\alpha$.
To bound $\vz_0''(\vm_*)$, let us introduce the function
\[f(m)=-\frac{2}{N}\sum_{\a=1}^M \frac{t_\a^2 m^3}{(1+t_\a m)^3}\]
and define analogously $\vf(m)$ for $\vT$.
We have $f'(m)<0$ for all $m$, so $f(\vm_*)>f(m_*)$. 
Furthermore, if $t_\a$ was the value which was replaced by 0, then
$1+t_\a \vm_*>0$. (This is obvious for positive $t_\a$; for negative $t_\a$, it
follows from $-t^{-1}<\vm_*<m_*<-t_\a^{-1}$, as $-t^{-1}$ is the closest pole to
the left of $\vm_*$.) Then $\vf(\vm_*)>f(\vm_*)>f(m_*)$.
Applying the condition $0=z_0'(m_*)$, we
verify $f(m_*)=m_*^4z_0''(m_*)$. Then
\[\vz_0''(\vm_*)>\frac{m_*^4}{\vm_*^4}z_0''(m_*)>z_0''(m_*).\]
This shows that $\vE_*=\vz_0(\vm_*)$ is $\tau'$-regular for a fixed constant
$\tau'>0$. (We may take $\tau'=\min(\nu,\tau/2)$ as above.)

Finally, $T^{(K_1+K_2)}$ now has at least $c_0M$ diagonal entries equal to $t$,
and all non-zero diagonal entries $t_\a$ satisfy $t_\a<t<0$.
Let $K_3$ be the number of such
entries and consider a sequence $T^{(K_1+K_2+1)},\ldots,T^{(K_1+K_2+K_3)}$ where
each $T^{(k+1)}$ replaces one such entry of $T^{(k)}$ by 0. Again,
each such swap satisfies
$\vz_0'(m_*)>z_0'(m_*)=0$, so by Lemma \ref{lemma:singleswap},
$-t^{-1}<\vm_*<m_*$. As in the $K_2$ swaps above,
this implies $\min_\a |\vm_*+\vt_\a^{-1}|>c$ for a
constant $c>0$. The condition $\vt_\a<t$ for all non-zero $\vt_\a$
implies that $1+\vt_\a\vm_*<0$ for all non-zero $\vt_\a$, so we have
\[\vf(\vm_*) \geq -\frac{2c_0M}{N}\frac{t^2\vm_*^3}{(1+t\vm_*)^3}>c\]
for a constant $c>0$, by Proposition \ref{prop:basicregbounds}. Applying again $\vf(\vm_*)=\vm_*^4\vz_0''(\vm_*)$, this
yields $\vz_0''(\vm_*)>c'>0$, so $\vE_*$ is $\tau'$-regular for a constant
$\tau'>0$.

The total number of swaps $L=K_1+K_2+K_3$ is at most $2M$.
All diagonal entries of $T^{(L)}$ belong to $\{0,t\}$,
so this concludes the proof.
\end{proof}

We now establish Lemma \ref{lemma:lindebergconstruction} for all properties
1--6 by rescaling.

\begin{proof}[Proof of Lemma \ref{lemma:lindebergconstruction}]
By Lemmas \ref{lemma:lindebergmneg} and
\ref{lemma:lindebergmpos}, there exist sequences
$T^{(0)},\ldots,T^{(L)}$ and $E_*^{(0)},\ldots,E_*^{(L)}$
satisfying conditions 1--5. By Lemma \ref{lemma:Escaleclose}, the associated
scales $\gamma_0,\ldots,\gamma_L$ satisfy $|\gamma_{l+1}-\gamma_l| \leq C/N$
for a $\phi,\tau'$-dependent constant $C>0$ and each $l=0,\ldots,L-1$.

We verify from the definitions of $E_*,m_*,\gamma$ that under the rescaling
$T \mapsto cT$ for any $c>0$, we have
\[E_* \mapsto cE_*,\qquad m_* \mapsto c^{-1}m_*,\qquad \gamma \mapsto
c^{-3/2}\gamma.\]
Consider then the matrices $\tilde{T}^{(l)}=\gamma_l^{2/3}T^{(l)}$ and edges
$\tilde{E}_*^{(l)}=\gamma_l^{2/3}E_*^{(l)}$.
We check properties 1--6 for $\tilde{T}^{(l)}$ and $\tilde{E}_*^{(l)}$:
Properties 1, 2, and 6 are obvious. Since
$T^{(0)},\ldots,T^{(L)}$ are all $\tau'$-regular, Proposition
\ref{prop:basicregbounds} implies $c<\gamma_l<C$ for constants $C,c>0$ and
every $l$. Then it is easy to check that properties 3, 4, and 5 also hold
with adjusted constants.
\end{proof}

\section{Resolvent comparison and proof of Theorem
\ref{thm:TW}}\label{sec:resolventcompare}

We will conclude the proof of Theorem \ref{thm:TW} by
establishing the following estimate.
\begin{theorem}[Resolvent comparison]\label{thm:resolventcompare}
Fix $\eps>0$ a sufficiently small constant, and
let $s_1,s_2,\eta \in \R$ be such that
$|s_1|,|s_2|<N^{-2/3+\eps}$ and $\eta \in [N^{-2/3-\eps},N^{-2/3}]$.
Let $T,\vT \in \R^{M \times M}$ be two diagonal matrices and
$E_*,\vE_*$ two corresponding regular right edges, such that $(T,E_*)$ and
$(\vT,\vE_*)$ are swappable and their scales satisfy $\gamma=\vg=1$.
Suppose Assumptions \ref{assump:dT} and \ref{assump:X} hold.

Let $m_N,\vm_N$ be the Stieltjes transforms as in (\ref{eq:mN}) corresponding
to $T,\vT$, and define
\[\X=N\int_{E_*+s_1}^{E_*+s_2} \Im m_N(y+i\eta)dy,\qquad
\vX=N\int_{\vE_*+s_1}^{\vE_*+s_2} \Im \vm_N(y+i\eta)dy.\]
Let $K:\R \to \R$
be any function such that $K$ and its first four derivatives are uniformly
bounded by a constant. Then
\begin{equation}\label{eq:resolventcompare}
\E[K(\X)-K(\vX)] \prec N^{-4/3+16\eps}.
\end{equation}
\end{theorem}

\begin{proof}[Proof of Theorem \ref{thm:TW}]
By symmetry under $T \mapsto -T$, it suffices to consider a right edge.
By rescaling $T \mapsto \gamma^{2/3}T$, it suffices to
consider $\gamma=1$.

Let $T^{(0)},\ldots,T^{(L)},E_*^{(0)},\ldots,E_*^{(L)}$ satisfy
Lemma \ref{lemma:lindebergconstruction}. Define $\X^{(k)}(s_1,s_2,\eta)$
as in (\ref{eq:X}) for each $(T^{(k)},E_*^{(k)})$.
For a small constant $\eps>0$, let
$\eta,s_+,l$ and $K:[0,\infty) \to [0,1]$ be as in
Lemma \ref{lemma:resolventapprox}, where $K$ has bounded derivatives of all
orders.
Fix $x \in \R$ and let $s=xN^{-2/3}$.
Applying Lemma \ref{lemma:resolventapprox},
\[\P[\lambda_{\max}(\hSigma) \leq E_*+s] \leq
\E[K(\pi^{-1}\X^{(0)}(s+l,s_+,\eta)]+N^{-1}.\]
Setting $\eps'=9\eps$ and applying Theorem \ref{thm:resolventcompare},
\[\E[K(\pi^{-1}\X^{(k)}(s+l,s_+,\eta)]
\leq \E[K(\pi^{-1}\X^{(k+1)}(s+l,s_+,\eta)]+N^{-4/3+17\eps'}\]
for each $k=0,\ldots,L-1$.
Finally, defining $\hSigma^{(L)}=X'T^{(L)}X$ and
$\lambda_{\max}(\hSigma^{(L)})$ as its largest eigenvalue
in $(E_*^{(L)}-\delta',E_*^{(L)}+\delta')$ for some
$\delta'>0$, applying Lemma \ref{lemma:resolventapprox} again yields
\[\E[K(\pi^{-1}\X^{(L)}(s+l,s_+,\eta)]
\leq \P[\lambda_{\max}(\hSigma^{(L)}) \leq E_*^{(L)}+s+2l]+N^{-1}.\]
Recalling $L \leq 2M$ and combining the above bounds,
\[\P[N^{2/3}(\lambda_{\max}(\hSigma)-E_*) \leq x]
\leq \P[N^{2/3}(\lambda_{\max}(\hSigma^{(L)})-E_*^{(L)})
\leq x+2N^{-\eps}]+o(1).\]

The matrix $T^{(L)}$ has all diagonal entries 0 or $t$, so
$\hSigma^{(L)}=t\tilde{X}'\tilde{X}$ for $\tilde{X} \in
\R^{\tilde{M} \times N}$ having $\N(0,1/N)$ entries.
The corresponding law
$\mu_0^{(L)}$ has a single support interval and a unique right edge, so
$E_*^{(L)}$ must be this edge. 
Regularity of $E_*^{(L)}$ and (\ref{eq:nondegenerate}) imply
$|t| \asymp 1$ and $\tilde{M}/N \asymp 1$.
If $E_*^{(L)}>0$, then $t>0$. If $E_*^{(L)}<0$, then $t<0$,
    and edge regularity implies
$\tilde{M}/N$ is bounded away from 1. Then we obtain
\begin{equation}\label{eq:finalTW}
\P[N^{2/3}(\lambda_{\max}(\hSigma^{(L)})-E_*^{(L)}) \leq x+2N^{-\eps}]
=F_1(x)+o(1)
\end{equation}
where $F_1$ is the distribution function of $\mu_{TW}$,
by applying the results of
    \cite{feldheimsodin,knowlesyin} to either the largest eigenvalue of
    $\hSigma^{(L)}$ or the smallest positive eigenvalue of $-\hSigma^{(L)}$.
Combining the above, we obtain
\[\P[N^{2/3}(\lambda_{\max}(\hSigma)-E_*) \leq x] \leq F_1(x)+o(1).\]
The reverse bound is analogous, concluding the proof.
\end{proof}

In the remainder of this section, we prove Theorem
\ref{thm:resolventcompare}.

\subsection{Individual resolvent bounds}\label{subsec:basicbounds}
For diagonal $T$ and for $z=y+i\eta$ as appearing in Theorem
\ref{thm:resolventcompare}, we record here simple resolvent bounds
that follow from the local law. Similar bounds were used in
\cite{erdosyauyin,leeschnelli}. We also introduce the shorthand notation
that will be used in the computation.

Let $E_*$ be a regular right edge. Fix a small constant $\eps>0$, and
fix $s_1,s_2,\eta$ such that $|s_1|,|s_2| \leq N^{-2/3+\eps}$ and
$\eta \in [N^{-2/3-\eps},N^{-2/3}]$. Changing variables, we write
\[\X \equiv \X(s_1,s_2,\eta)=N\int_{s_1}^{s_2} \Im m_N(y+E_*+i\eta)dy.\]
For $y \in [s_1,s_2]$, we write as shorthand
\[z \equiv z(y)=y+E_*+i\eta, \quad
G \equiv G(z(y)), \quad m_N \equiv m_N(z(y)),\quad
G^{(\a)} \equiv G^{(\a)}(z(y)),\]
\[m_N^{(\a)} \equiv \frac{1}{N}\sum_{i \in \I_N} G_{ii}^{(\a)}(z(y)),
\quad \X^{(\a)} \equiv N\int_{s_1}^{s_2}
\Im m_N^{(\a)}(\ty+E_*+i\eta)d\ty.\]

We use the simplified summation notation
\[\sum_{i,j} \equiv \sum_{i,j \in \I_N},\qquad
\sum_{\a,\b} \equiv \sum_{\a,\b \in \I_M}\]
where sums over lower-case Roman indices are over $\I_N$ and
sums over Greek indices are over $\I_M$. We use also the
simplified integral notation
\[\int \tG_{AB} \equiv \int_{s_1}^{s_2} G(z(\ty))_{AB} d\ty,\qquad
\int \tm_N \equiv \int_{s_1}^{s_2} m_N(z(\ty))d\ty,\]
so that integrals are implicitly over $[s_1,s_2]$, and we denote
by $\tilde{F}$ the function $F$ evaluated at $F(z(\ty))$ for $\ty$ the
variable of integration. In this notation, $\X$ and $\X\pa$ are simply
\[\X=\sum_i \Im \int \tG_{ii},\qquad
\X\pa=\sum_i \Im \int \tG_{ii}\pa.\]

We introduce the fundamental small parameter
\begin{equation}
\Psi=N^{-1/3+3\eps}.
\end{equation}
We will eventually bound all quantities in the computation by powers of $\Psi$.
In fact, as shown in Lemmas \ref{lemma:resolventbounds} and
\ref{lemma:concentration}
below, non-integrated resolvent entries are controlled by powers of the smaller 
quantity $N^{-1/3+\eps}$.
However, integrated quantities will require the additional slack of $N^{2\eps}$.
We will pass to using $\Psi$ for
all bounds after this distinction is no longer needed.

We have the following corollaries of Proposition \ref{prop:m0estimates}
and Theorem \ref{thm:locallaw}:

\begin{lemma}\label{lemma:resolventbounds}
Under the assumptions of Theorem \ref{thm:resolventcompare},
for all $y \in [s_1,s_2]$, $i \neq j \in \I_N$, and $\a \neq \b \in \I_M$,
\[G_{ii} \prec 1,\quad \frac{1}{G_{ii}} \prec 1,\quad
\frac{G_{\a\a}}{t_\a} \prec 1,\quad \frac{t_\a}{G_{\a\a}} \prec 1,
\quad G_{ij} \prec N^{-1/3+\eps},\]
\[\frac{G_{i\a}}{t_\a}
\prec N^{-1/3+\eps},\quad \frac{G_{\a\b}}{t_\a t_\b} \prec
N^{-1/3+\eps},\quad m_N-m_* \prec N^{-1/3+\eps}.\]
If $T$ is singular, these are defined by continuity and
the form (\ref{eq:Galt}) for $G$.
\end{lemma}
\begin{proof}
Proposition \ref{prop:m0estimates} implies
$\Im m_0(z(y)) \leq C\sqrt{\kappa+\eta} \leq CN^{-1/3+\eps/2}$, while $\eta \geq
N^{-2/3-\eps}$ by assumption. Then Theorem \ref{thm:locallaw} 
yields $(t_At_B)^{-1}(G-\Pi)_{AB} \prec N^{-1/3+\eps}$ for all $A,B \in \I$.
Proposition \ref{prop:m0estimates} also implies
$|m_0(z)| \asymp 1$ and $|1+t_\a m_0(z)| \asymp 1$, from which
all of the entrywise bounds on $G$ follow.
The bound on $m_N$ follows from
$|m_0-m_*| \leq C\sqrt{\kappa+\eta} \leq CN^{-1/3+\eps/2}$
and $|m_N-m_0| \prec N^{-1/3+\eps}$.
\end{proof}

\begin{lemma}\label{lemma:concentration}
Under the assumptions of Theorem \ref{thm:resolventcompare}, 
for all $i \in \I_N$ and $\a \in \I_M$,
\[\sum_k G_{ik}\pa X_{\a k} \prec N^{-1/3+\eps},
\qquad \sum_{p,q} G_{pq}\pa X_{\a p}X_{\a q}-m_* \prec N^{-1/3+\eps}.\]
\end{lemma}
\begin{proof}
Applying Lemmas \ref{lemma:resolventidentities}(b) and
\ref{lemma:resolventbounds},
\[\sum_k G_{ik}\pa X_{\a k}=-G_{i\a}/G_{\a\a} \prec N^{-1/3+\eps}.\]
Similarly, applying Lemma \ref{lemma:resolventidentities}(a)
and Theorem \ref{thm:locallaw},
\[\sum_{p,q} G_{pq}\pa X_{\a p}X_{\a q}-m_*
=-\frac{1}{G_{\a\a}}-\frac{1}{t_\a}-m_*
=\frac{1}{\Pi_{\a\a}}-\frac{1}{G_{\a\a}}+(m_0-m_*)
\prec N^{-1/3+\eps}.\]
\end{proof}

\begin{remark}
All probabilistic bounds 
such as the above are derived from Theorem \ref{thm:locallaw}. Thus they in
fact hold in the uniform sense of
Corollary \ref{cor:locallawunionbound}. We continue to use the notation
$\prec$ for convenience, with the understanding that we may take
union bounds and integrals over $y \in [s_1,s_2]$.
\end{remark}

We record one trivial bound for an integral that will be repeatedly used,
and which explains the appearance of $\Psi$.
\begin{lemma}\label{lemma:integralbound}
Suppose the assumptions of Theorem \ref{thm:resolventcompare} hold,
$F(z(y)) \prec N^{a(-1/3+\eps)}$ for some $a \geq 2$,
and we may take a union bound of this statement over $y \in [s_1,s_2]$
(in the sense of Lemma \ref{lemma:infiniteunion}). Then, with
$\Psi=N^{-1/3+3\eps}$,
\[N\int \tilde{F} \prec \Psi^{a-1}.\]
\end{lemma}
\begin{proof}
We have
$N(s_2-s_1)N^{a(-1/3+\eps)} \leq 2N^{1/3+\eps}N^{a(-1/3+\eps)} \leq
2\Psi^{a-1}$.
\end{proof}

The next lemma allows us to ``remove the superscript'' in the computation.
\begin{lemma}\label{lemma:removesuper}
Under the assumptions of Theorem \ref{thm:resolventcompare},
for any $y \in [s_1,s_2]$, $i,j \in \I_N$ (possibly equal), and $\a \in \I_M$,
\[G_{ij}-G_{ij}\pa \prec N^{2(-1/3+\eps)},\quad
m_N-m_N\pa \prec N^{2(-1/3+\eps)}, \quad
\X-\X\pa \prec \Psi.\]
\end{lemma}
\begin{proof}
Applying the last resolvent identity from Lemma \ref{lemma:resolventidentities},
\[G_{ij}-G_{ij}\pa=\frac{G_{i\a}G_{j\a}}{G_{\a\a}}
=G_{i\a}\frac{G_{j\a}}{t_\a}\frac{t_\a}{G_{\a\a}},\]
so the first statement follows from Lemma \ref{lemma:resolventbounds}.
Taking $i=j$ and averaging over $\I_N$ yields
the second statement. The third statement follows from
Lemma \ref{lemma:integralbound} and
$\X-\X\pa=\Im N\int (\tm_N-\tm_N\pa)$.
\end{proof}

\subsection{Resolvent bounds for a swappable pair}\label{subsec:swappablebounds}
We now record bounds for a swappable pair $(T,E_*)$ and $(\vT,\vE_*)$, where
$E_*,\vE_*$ are both regular. We denote by $\vm_N,\vG,\vX$ the analogues
of $m_N,G,\X$ for $\vT$. For $\eps,s_1,s_2,\eta$ and $y \in [s_1,s_2]$ as in
Section \ref{subsec:basicbounds}, we write as shorthand
\[\vz \equiv \vz(y)=y+\vE_*+i\eta,\qquad \vG \equiv \vG(\vz(y)),
\qquad \vm_N \equiv \vm_N(\vz(y)).\]
The results of the preceding section hold equally for $\vG$, $\vm_N$, and $\vX$.

The desired bound (\ref{eq:resolventcompare}) arises from the following
identity: Suppose first that $T$ and $\vT$ are invertible.
Applying $A^{-1}-B^{-1}=A^{-1}(B-A)B^{-1}$,
\[G-\vG=G\begin{pmatrix} (-\vz+z)\Id
& 0 \\ 0 & -\vT^{-1}+T^{-1} \end{pmatrix} \vG.\]
Hence, as $z-\vz=E_*-\vE_*$,
\begin{equation}\label{eq:GijvGij}
G_{ij}-\vG_{ij}=\sum_k G_{ik}\vG_{jk}(E_*-\vE_*)-\sum_\a \frac{G_{i\a}}{t_\a}
\frac{\vG_{j\a}}{\vt_\a}(t_\a-\vt_\a).
\end{equation}
This holds by continuity when $T$ is singular, using the form (\ref{eq:Galt}). 

The following lemma allows us to ``remove the check'' in the computation.
\begin{lemma}\label{lemma:removecheck}
Suppose the assumptions of Theorem \ref{thm:resolventcompare} hold.
Let $\Psi=N^{-1/3+3\eps}$. Then for any $y \in [s_1,s_2]$,
$i,j \in \I_N$ (possibly equal), and $\a \in \I_M$,
\[G_{ij}-\vG_{ij} \prec N^{2(-1/3+\eps)}, \quad
m_N-\vm_N \prec N^{2(-1/3+\eps)}, \quad
\X-\vX \prec \Psi.\]
\end{lemma}
\begin{proof}
Applying Lemma \ref{lemma:resolventbounds} for both $G$ and $\vG$, and also the
definition of swappability and Lemma \ref{lemma:Escaleclose}, we have from
(\ref{eq:GijvGij})
\[G_{ij}-\vG_{ij} \prec |E_*-\vE_*| \cdot N \cdot N^{2(-1/3+\eps)}
+\sum_\a |t_\a-\vt_\a| N^{2(-1/3+\eps)} \prec N^{2(-1/3+\eps)}.\]
(The contribution from $k=i$ or $k=j$ in the first sum of (\ref{eq:GijvGij})
is of lower order.)
Taking $i=j$ and averaging over $\I_N$ yields the second statement,
and integrating over $y \in [s_1,s_2]$ and applying Lemma
\ref{lemma:integralbound} yields the third.
\end{proof}

In many cases, we may strengthen the above lemma by an additional factor
of $\Psi$ if we take an expectation. (This may be seen by taking $Y=Y\pa=1$ and
$a=0$ in Lemma \ref{lemma:removecheck2} below.)
To take expectations of remainder terms, we will invoke Lemma
\ref{lemma:expectationdomination} combined with the following basic bound:
\begin{lemma}\label{lemma:crudeGXbound}
Under the assumptions of Theorem \ref{thm:resolventcompare},
let $P \equiv P(z(y))$ be any polynomial in the entries of $X$ and $G$ with
bounded degree, bounded (possibly random)
coefficients, and at most $N^C$ terms for a constant
$C>0$. Then for a constant $C'>0$ and all $y \in [s_1,s_2]$,
we have $\E[|P|] \leq N^{C'}$.
\end{lemma}
\begin{proof}
By the triangle inequality and Holder's inequality, it suffices to consider a
bounded power of a single entry of $G$ or $X$. Then the result follows from
(\ref{eq:mNbound}) and the form (\ref{eq:Galt}) for $G$.
\end{proof}

\begin{lemma}\label{lemma:removecheck2}
Under the assumptions of Theorem \ref{thm:resolventcompare},
let $Y$ be any quantity such that $Y \prec \Psi^a$ for some constant
$a \geq 0$. Suppose that for each $\a \in \I_M$, there exists a quantity $Y\pa$
such that $Y-Y\pa \prec \Psi^{a+1}$, and $Y\pa$ is independent of row $\a$
of $X$. Suppose furthermore that $\E[|Y|^\ell] \leq N^{C_\ell}$ for each integer
$\ell>0$ and some constants $C_1,C_2,\ldots>0$.

Then, for all $i,j \in \I_N$ (possibly equal) and $y \in [s_1,s_2]$,
\[\E[(G_{ij}-\vG_{ij})Y] \prec N^{2(-1/3+\eps)}\Psi^{a+1} \prec \Psi^{a+3},\]
\[\E[(m_N-\vm_N)Y] \prec N^{2(-1/3+\eps)}\Psi^{a+1} \prec \Psi^{a+3}.\]
\[\E[(\X-\vX)Y] \prec \Psi^{a+2}.\]
\end{lemma}
\begin{proof}
Applying (\ref{eq:Eplusdiff}), the bound $N^{-1} \prec \Psi^3$, and
Lemma \ref{lemma:resolventbounds} to (\ref{eq:GijvGij}),
\begin{align*}
(G_{ij}-\vG_{ij})Y&=\sum_k G_{ik}\vG_{jk}(E_*-\vE_*)Y
-\sum_\a \frac{G_{i\a}}{t_\a}\frac{\vG_{j\a}}{\vt_\a}(t_\a-\vt_\a)Y\\
&=\sum_\a (t_\a-\vt_\a) \left(s_\a\vs_\a\frac{1}{N}\sum_k
G_{ik}\vG_{jk}-\frac{G_{i\a}}{t_\a}\frac{\vG_{j\a}}{\vt_\a}\right)Y
+\O(\Psi^{a+5}).
\end{align*}
By swappability and Lemma \ref{lemma:resolventbounds},
the explicit term on the right is of size
$\O(N^{2(-1/3+\eps)}\Psi^a)$. (The contributions from $k=i$ and $k=j$ in the
summation are of lower order.) Applying
the assumption $Y-Y\pa \prec \Psi^{a+1}$
as well as Lemma \ref{lemma:removesuper},
we may replace $Y$ with $Y\pa$, $G_{ik}$ with $G_{ik}\pa$, and $\vG_{jk}$ with
$\vG_{jk}\pa$ above while introducing an $\O(N^{2(-1/3+\eps)}\Psi^{a+1})$
error. Hence,
\begin{align}
(G_{ij}-\vG_{ij})Y&=\sum_\a (t_\a-\vt_\a) \left(s_\a\vs_\a\frac{1}{N}\sum_k
G_{ik}\pa\vG_{jk}\pa-\frac{G_{i\a}}{t_\a}\frac{\vG_{j\a}}{\vt_\a}\right)Y\pa
\nonumber\\
&\hspace{1in}+\O(N^{2(-1/3+\eps)}\Psi^{a+1}).\label{eq:GijvGijY}
\end{align}
Applying the resolvent identities from Lemma \ref{lemma:resolventidentities},
\[\frac{G_{i\a}}{t_\a}
=\frac{G_{\a\a}}{t_\a}\sum_k G_{ik}\pa X_{\a k}
=-\frac{1}{1+t_\a \sum_{p,q} G_{pq}\pa X_{\a p}X_{\a q}}
\sum_k G_{ik}\pa X_{\a k}.\]
Recalling $s_\a=(1+t_\a m_*)^{-1}$, and applying Lemma
\ref{lemma:concentration} and a Taylor expansion of 
$(1+t_\a x)^{-1}$ around $x=m_*$,
\[\frac{G_{i\a}}{t_\a}=-s_\a\sum_k G_{ik}\pa X_{\a k}+\O(N^{2(-1/3+\eps)}),\]
where the explicit term on the right is of size $\O(N^{-1/3+\eps})\prec \Psi$.
A similar expansion holds for $\vG_{j\a}/\vt_\a$. Substituting into
(\ref{eq:GijvGijY}),
\begin{align*}
(G_{ij}-\vG_{ij})Y&=\sum_\a (t_\a-\vt_\a)s_\a\vs_\a \Bigg(\frac{1}{N}\sum_k
G_{ik}\pa\vG_{jk}\pa
-\sum_{k,l} G_{ik}\pa X_{\a k}
\vG_{jl}\pa X_{\a l}\Bigg)Y\pa\\
&\qquad +\O(N^{2(-1/3+\eps)}\Psi^{a+1}).
\end{align*}
Denoting by $\E_\a$ the partial expectation over only row $\a$ of $X$
(i.e.\ conditional on $X_{\beta j}$ for all $\beta \neq \alpha$), we have
\[\E_\a\left[\frac{1}{N}\sum_k G_{ik}\pa\vG_{jk}\pa
-\sum_{k,l} G_{ik}\pa X_{\a k} \vG_{jl}\pa X_{\a l}\right]=0,\]
while the remainder term remains $\O(N^{2(-1/3+\eps)}\Psi^{a+1})$
by Lemma \ref{lemma:expectationdomination}, where the moment condition
of Lemma \ref{lemma:expectationdomination} is
verified by Lemma \ref{lemma:crudeGXbound},
the moment assumption on $Y$, and Cauchy-Schwarz.
Then the first statement follows.
The second statement follows from applying this with $i=j$ and
averaging over $i \in \I_N$. The third statement follows from integrating
over $y \in [s_1,s_2]$ and noting $N^{1/3+\eps}N^{2(-1/3+\eps)}=\Psi$ as in
Lemma \ref{lemma:integralbound}. (If $Y$ also depends on the spectral parameter
$z(y)$, we evaluate $m_N$ and $\vm_N$ at a different parameter $\ty$
and integrate over $\ty$.)
\end{proof}

Finally, we derive a deterministic consequence of swappability and the
scaling condition $\gamma=\vg=1$. In the proof of
\cite{leeschnelli} for a continuous interpolation $T^{(l)}$,
denoting $\dot{t}_\a$ and $\dot{m}_*$ the derivatives with respect to $l$, the
differential analogue of the following lemma is the pair of identities
\[\sum_\a \dot{t}_\a t_\a s_\a^3=N\dot{m}_*,\qquad
\sum_\a \dot{t}_\a t_\a^2 s_\a^4=N\dot{m}_*(A_4-m_*^{-4}).\]

\begin{lemma}\label{lemma:sumrules}
Suppose $T,\vT$ satisfy Assumption \ref{assump:dT}, $E_*,\vE_*$ are associated
regular right edges with scales $\gamma=\vg=1$, and
$(T,E_*)$ and $(\vT,\vE_*)$ are swappable. Define $s_\a=(1+t_\a m_*)^{-1}$,
$\vs_\a=(1+\vt_\a \vm_*)^{-1}$, $A_4=N^{-1}\sum_\a t_\a^4 s_\a^4$,
    \begin{equation}\label{eq:PQ}
        \cP_\a=s_\a \vs_\a (t_\a s_\a+\vt_\a \vs_\a),\qquad
\Q_\a=s_\a \vs_\a (t_\a^2 s_\a^2+t_\a s_\a \vt_\a \vs_\a
+\vt_\a^2 \vs_\a^2).
    \end{equation}
Then for some constant $C>0$, both of the following hold:
\begin{align}
    \left|2N(m_*-\vm_*)-\sum_{\a=1}^M (t_\a-\vt_\a)\cP_\a\right|&\leq C/N
\label{eq:sumrule1}\\
\left|3N(m_*-\vm_*)(A_4-m_*^{-4})
    -\sum_{\a=1}^M (t_\a-\vt_\a)\Q_\a\right|&\leq C/N.\label{eq:sumrule2}
\end{align}
\end{lemma}
\begin{proof}
For (\ref{eq:sumrule1}),
we have from $0=z_0'(m_*)$ applied to $T$ and $\vT$
\begin{equation}\label{eq:sumrule11}
m_*^{-2}-\vm_*^{-2}=\frac{1}{N}\sum_\a t_\a^2 s_\a^2-\vt_\a^2\vs_\a^2.
\end{equation}
The left side may be written as
\begin{equation}\label{eq:sumrule12}
m_*^{-2}-\vm_*^{-2}=(\vm_*-m_*)(\vm_*+m_*)m_*^{-2}\vm_*^{-2}
=2(\vm_*-m_*)m_*^{-3}+O(N^{-2}),
\end{equation}
where the second equality applies $|m_*|,|\vm_*| \asymp 1$ and
$|\vm_*-m_*| \leq C/N$. The right side may be written as
\[\frac{1}{N}\sum_\a t_\a^2 s_\a^2-\vt_\a^2 \vs_\a^2
=\frac{1}{N}\sum_\a (t_\a-\vt_\a)t_\a s_\a^2
+(s_\a^2-\vs_\a^2)t_\a \vt_\a+(t_\a-\vt_\a) \vt_\a \vs_\a^2.\]
Including the identities $(1+t_\a m_*)s_\a=1$ and $(1+\vt_\a \vm_*)\vs_\a=1$,
\begin{align}
&\frac{1}{N}\sum_\a t_\a^2 s_\a^2-\vt_\a^2 \vs_\a^2\nonumber\\
&=\frac{1}{N}\sum_\a (t_\a-\vt_\a)(t_\a s_\a^2(1+\vt_\a \vm_*)\vs_\a
+\vt_\a \vs_\a^2 (1+t_\a m_*)s_\a)+(s_\a^2-\vs_\a^2)t_\a \vt_\a\nonumber\\
&=\frac{1}{N}\sum_\a (t_\a-\vt_\a)s_\a \vs_\a
(t_\a s_\a+\vt_\a \vs_\a+t_\a s_\a \vt_\a \vm_*
+\vt_\a \vs_\a t_\a m_*)+(s_\a^2-\vs_\a^2)t_\a \vt_\a\nonumber\\
&\equiv \frac{1}{N}\sum_\a (t_\a-\vt_\a)s_\a \vs_\a
(t_\a s_\a+\vt_\a \vs_\a)+R_\a,\label{eq:sumrule13}
\end{align}
where we define $R_\a$ as the remainder term. Noting that
\[s_\a^2-\vs_\a^2=(s_\a-\vs_\a)(s_\a+\vs_\a)
=(\vt_\a \vm_*-t_\a m_*)s_\a\vs_\a(s_\a+\vs_\a),\]
we have
\begin{align*}
R_\a&=t_\a\vt_\a s_\a \vs_\a(t_\a s_\a\vm_*+t_\a\vs_\a m_*
-\vt_\a s_\a\vm_*-\vt_\a\vs_\a m_*\\
&\qquad +\vt_\a s_\a \vm_*+\vt_\a\vs_\a\vm_*
-t_\a s_\a m_*-t_\a \vs_\a m_*)\\
&=t_\a s_\a \vt_\a \vs_\a(\vm_*-m_*)(t_\a s_\a+\vt_\a \vs_\a).
\end{align*}
Then, denoting
$A_{i,j}=N^{-1}\sum_\a t_\a^i s_\a^i \vt_\a^j \vs_\a^j$
and applying Lemma \ref{lemma:Escaleclose}(a),
\[\frac{1}{N}\sum_\a R_\a=(\vm_*-m_*)(A_{2,1}+A_{1,2})
=2(\vm_*-m_*)A_{3,0}+O(N^{-2}).\]
By the scaling $\gamma=1$, we have $A_{3,0}=1+m_*^{-3}$.
Combining this with (\ref{eq:sumrule11}), (\ref{eq:sumrule12}), and
(\ref{eq:sumrule13}) and multiplying by $N$ yields (\ref{eq:sumrule1}).

The identity (\ref{eq:sumrule2}) follows similarly:
The condition $\gamma=\vg$ implies
\[m_*^{-3}-\vm_*^{-3}=\frac{1}{N}\sum_\a t_\a^3 s_\a^3-\vt_\a^3\vs_\a^3.\]
The left side is
\[(\vm_*-m_*)(m_*^2+m_*\vm_*+\vm_*^2)m_*^{-3}\vm_*^{-3}
=3(\vm_*-m_*)m_*^{-4}+O(N^{-2}),\]
while the right side is
\begin{align*}
&\frac{1}{N}\sum_\a t_\a^3 s_\a^3-\vt_\a^3\vs_\a^3\\
&=\frac{1}{N}\sum_\a (t_\a-\vt_\a)t_\a^2 s_\a^3
+(s_\a^2-\vs_\a^2)t_\a^2 s_\a \vt_\a
+(t_\a-\vt_\a)t_\a s_\a \vt_\a \vs_\a^2\\
&\hspace{0.5in}+(s_\a-\vs_\a) t_\a \vt_\a^2 \vs_\a^2
+(t_\a-\vt_\a) \vt_\a^2 \vs_\a^3\\
&=\frac{1}{N}\sum_\a (t_\a-\vt_\a)\Big(t_\a^2 s_\a^3(1+\vt_\a \vm_*)\vs_\a
+t_\a s_\a \vt_\a \vs_\a^2(1+t_\a m_*)s_\a
+\vt_\a^2 \vs_\a^3(1+t_\a m_*)s_\a\Big)\\
&\hspace{0.5in}+(s_\a-\vs_\a)((s_\a+\vs_\a)t_\a^2 s_\a \vt_\a
+t_\a \vt_\a^2 \vs_\a^2)\\
&=\frac{1}{N}\sum_\a (t_\a-\vt_\a)s_\a \vs_\a(t_\a^2 s_\a^2
+t_\a s_\a \vt_\a \vs_\a+\vt_\a^2 \vs_\a^2)\\
&\hspace{0.5in}
+t_\a s_\a \vt_\a \vs_\a(\vm_*-m_*)(t_\a^2 s_\a^2+t_\a s_\a \vt_\a \vs_\a
+\vt_\a^2 \vs_\a^2)\\
&=\left(\frac{1}{N}\sum_\a (t_\a-\vt_\a)\Q_\a\right)+3(\vm_*-m_*)A_4+O(N^{-2}).
\end{align*}
Combining the above and multiplying by $N$ yields (\ref{eq:sumrule2}).
\end{proof}

\subsection{Proof of resolvent comparison}
We use the notation of Sections \ref{subsec:basicbounds} and
\ref{subsec:swappablebounds}.

The proof of Theorem \ref{thm:resolventcompare} is a lengthy computation
using the preceding lemmas. To help organize the various terms which appear in
this computation, we denote them as $\X_{k,*}$ for $k=3,4$ and $*$ a label
describing the form of this term. Each $\X_{k,*}$ is of size 
at most $\O(\Psi^k)$, as may be verified from Lemmas \ref{lemma:resolventbounds}
and \ref{lemma:integralbound}. In the label $*$: 1 indicates a term $m_N-m_*$,
2, 3, or 4 indicate a product of 2, 3, or 4 resolvent entries $G_{ij}$,
the mark
$'$ indicates that a resolvent entry is squared, and the superscript $\sim$
denotes that this quantity is contained inside $\Im \int$.
All of these terms depend implicitly on a fixed index
$i \in \I_N$ and $y \in [s_1,s_2]$, which we omit for notational brevity.
{\allowdisplaybreaks
\begin{align*}
\X_{3,12'}&=K'(\X)(m_N-m_*)\frac{1}{N}\sum_k G_{ik}^2\\
\X_{3,3}&=K'(\X)\frac{1}{N^2}\sum_{k,l} G_{ik}G_{kl}G_{il}\\
\X_{3,2\t{2}}&=K''(\X)\frac{1}{N^2}\sum_{j,k,l} G_{ik}G_{il}
\Im \int \tG_{jk}\tG_{jl}\\
\X_{3,2'\t{2}'}&=K''(\X)\frac{1}{N^2}\sum_{j,k,l} G_{ik}^2
\Im \int \tG_{jl}^2\\
\X_{4,22'}&=K'(\X)(m_N-m_*)^2\frac{1}{N}\sum_k G_{ik}^2\\
\X_{4,13}&=K'(\X)(m_N-m_*)\frac{1}{N^2}\sum_{k,l}G_{ik}G_{kl}G_{il}\\
\X_{4,4}&=K'(\X)\frac{1}{N^3}\sum_{j,k,l}G_{ij}G_{jk}G_{kl}G_{il}\\
\X_{4,4'}&=K'(\X)\frac{1}{N^3}\sum_{j,k,l}G_{ik}^2G_{jl}^2\\
\X_{4,12\t{2}}&=K''(\X)(m_N-m_*)\frac{1}{N^2}\sum_{j,k,l}G_{ik}G_{il}
\Im\int\tG_{jk}\tG_{jl}\\
\X_{4,12'\t{2}'}&=K''(\X)(m_N-m_*)\frac{1}{N^2}\sum_{j,k,l}G_{ik}^2
\Im\int\tG_{jl}^2\\
\X_{4,3\t{2}}&=K''(\X)\frac{1}{N^3}\sum_{j,p,q,r}G_{ip}G_{iq}G_{pr}
\Im\int\tG_{jq}\tG_{jr}\\
\X_{4,3'\t{2}}&=K''(\X)\frac{1}{N^3}\sum_{j,p,q,r}G_{ir}^2G_{pq}
\Im\int\tG_{jp}\tG_{jq}\\
\X_{4,3\t{2}'}&=K''(\X)\frac{1}{N^3}\sum_{j,p,q,r}G_{iq}G_{ir}G_{qr}
\Im\int\tG_{jp}^2\\
\X_{4,2\t{12}}&=K''(\X)\frac{1}{N^2}\sum_{j,k,l}G_{ik}G_{il}
\Im\int(\tm_N-m_*)\tG_{jk}\tG_{jl}\\
\X_{4,2'\t{12}'}&=K''(\X)\frac{1}{N^2}\sum_{j,k,l}G_{il}^2
\Im\int(\tm_N-m_*)\tG_{jk}^2\\
\X_{4,2\t{3}}&=K''(\X)\frac{1}{N^3}\sum_{j,p,q,r}G_{ip}G_{iq}
\Im\int\tG_{jp}\tG_{jr}\tG_{qr}\\
\X_{4,2\t{3}'}&=K''(\X)\frac{1}{N^3}\sum_{j,p,q,r}G_{ip}G_{iq}
\Im\int\tG_{jr}^2\tG_{pq}\\
\X_{4,2'\t{3}}&=K''(\X)\frac{1}{N^3}\sum_{j,p,q,r}G_{ip}^2
\Im\int\tG_{jq}\tG_{jr}\tG_{qr}\\
\X_{4,2\t{2}\t{2}}&=K'''(\X)\frac{1}{N^3}\sum_{j,k,p,q,r}G_{ip}G_{iq}
\left(\Im\int\tG_{jp}\tG_{jr}\right)\left(\Im \int \tG_{kq}\tG_{kr}\right)\\
\X_{4,2'\t{2}\t{2}}&=K'''(\X)\frac{1}{N^3}\sum_{j,k,p,q,r}G_{ip}^2
\left(\Im\int\tG_{jq}\tG_{jr}\right)\left(\Im \int \tG_{kq}\tG_{kr}\right)\\
\X_{4,2\t{2}\t{2}'}&=K'''(\X)\frac{1}{N^3}\sum_{j,k,p,q,r}G_{ip}G_{iq}
\left(\Im\int\tG_{jp}\tG_{jq}\right)\left(\Im \int \tG_{kr}^2\right)\\
\X_{4,2'\t{2}'\t{2}'}&=K'''(\X)\frac{1}{N^3}\sum_{j,k,p,q,r}G_{ip}^2
\left(\Im\int\tG_{jq}^2\right)\left(\Im \int \tG_{kr}^2\right)
\end{align*}}
Define the aggregate quantities
\begin{align*}
\X_3&=\X_{3,12}+\X_{3,3}+\X_{3,2\t{2}}\\
\X_4&=3\X_{4,22'}+6\X_{4,13}+12\X_{4,4}+3\X_{4,4'}+4\X_{4,12\t{2}}+8\X_{4,3\t{2}}+4\X_{4,3'\t{2}}\\
&\hspace{0.5in}
+2\X_{4,2\t{12}}+2\X_{4,2\t{3}'}+4\X_{4,2\t{3}}+4\X_{4,2\t{2}\t{2}},\\
\X_4^-&=\X_{4,2\t{12}}+\X_{4,2\t{3}'}+2\X_{4,2\t{3}}-
\X_{4,12\t{2}}-\X_{4,3'\t{2}}-2\X_{4,3\t{2}}.
\end{align*}
Theorem \ref{thm:resolventcompare} is a consequence of the following two
technical results.

\begin{lemma}[Decoupling]\label{lemma:decoupling}
Under the assumptions of Theorem \ref{thm:resolventcompare},
denote $\X_\lambda=\lambda \X+(1-\lambda)\vX$ for $\lambda \in [0,1]$.
For fixed $i \in \I_N$ and $y \in [s_1,s_2]$, define $\X_3$, $\X_4$, and
$\X_4^-$ as above. For fixed $\a \in \I_M$, let
$s_\a=(1+t_\a m_*)^{-1}$ and $\vs_\a=(1+\vt_\a \vm_*)^{-1}$, define
$\cP_\a$ and $\Q_\a$ as in (\ref{eq:PQ}), and
\[\cR_\a=s_\a\vs_\a(t_\a s_\a-\vt_\a \vs_\a)^2.\]
Then
\begin{align*}
&\int_0^1 \E\left[K'(\X_\lambda)\frac{G_{i\a}}{t_\a}
\frac{\vG_{i\a}}{\vt_\a}\right]\,d\lambda\\
&=s_\a\vs_\a \int_0^1 \E\left[K'(\X_\lambda)\frac{1}{N}\sum_k G_{ik}
\vG_{ik}\right]\,d\lambda
-\cP_\a\E[\X_3]
+\tfrac{1}{3}\Q_\a\E[\X_4]+\tfrac{1}{3}\cR_\a\E[\X_4^-]+\O(\Psi^5).
\end{align*}
\end{lemma}
\begin{lemma}[Optical theorems]\label{lemma:optical}
Under the assumptions of Theorem \ref{thm:resolventcompare}, for fixed $i \in
\I_N$ and $y \in [s_1,s_2]$, define $\X_3$ and $\X_4$ as above.
Let $A_4=N^{-1}\sum_\a t_\a^4 s_\a^4$. Then
\[2\Im \E[\X_3]=(A_4-m_*^{-4})\Im \E[\X_4]+\O(\Psi^5).\]
\end{lemma}

Lemma \ref{lemma:decoupling} generalizes \cite[Lemma 6.2]{leeschnelli}
to a swappable pair. We will present its proof in
Section \ref{subsec:decoupling}. We introduce the interpolation
$\X_\lambda=\lambda \X+(1-\lambda)\vX$ as a device to bound $K(\X)-K(\vX)$.
(This is different from a continuous interpolation between the entries of
$T$ and $\vT$.) Let us make several additional remarks:
\begin{enumerate}[1.]
\item The proof in \cite{leeschnelli} requires this lemma
in ``differential form'', where $T=\vT$. In this case, we have
$G=\vG$, $\X_\lambda=\X$ for every $\lambda \in [0,1]$,
$s_\a=\vs_\a$, and $t_\a=\vt_\a$. Then the integral over $\lambda$ is
irrelevant, and Lemma \ref{lemma:decoupling} reduces to the full version of
\cite[Lemma 6.2]{leeschnelli}.
\item The term $\X_4^-$ does not appear in
\cite{leeschnelli} and is not canceled by the optical theorems of Lemma
\ref{lemma:optical}. (When $T=\vT$, we have $\cR_\a=0$ so this term is not
present.) The cancellation instead occurs by symmetry of its definition,
upon integrating over $y$: Momentarily writing $\X_{k,*}$ as $\X_{k,*}(y)$,
and noting that $K(\X)$ is real-valued, we obtain
\begin{equation}\label{eq:4minuscancellation}
\Im \int \X_{4,2\t{12}}(\ty) d\ty=\Im \int \X_{4,12\t{2}}(\ty)d\ty
\end{equation}
from the symmetric definition of these two terms. A similar cancellation
occurs for the pairs $(\X_{4,2\t{3}'},\X_{4,3'\t{2}})$ and
$(\X_{4,2\t{3}},\X_{4,3\t{2}})$ which comprise $\X_4^-$.
\item An important simplification in the proof
is that we may use Lemmas \ref{lemma:removecheck} and
\ref{lemma:removecheck2} to convert $\O(\Psi^3)$ and $\O(\Psi^4)$ terms to
involve only $G$ and not $\vG$---hence $\X_3,\X_4,\X_4^-$ are defined only
by $T$ and not $\vT$.
\end{enumerate}

The other technical ingredient, Lemma \ref{lemma:optical}, is identical to
the full version of \cite[Lemma B.1]{leeschnelli}, as the terms $\X_3$ and
$\X_4$ depend only on the single matrix $T$. We briefly discuss the
breakdown of its proof in Section \ref{subsec:optical}.

In \cite{leeschnelli}, for expositional
clarity, these lemmas were stated and proven only in the special case
$K' \equiv 1$. Full proofs were presented for an analogous
deformed Wigner model in \cite{leeschnelliwigner}. Although more cumbersome, we 
will demonstrate the full proof of Lemma \ref{lemma:decoupling}
for general $K$ in Section \ref{subsec:decoupling},
as much of the additional complexity in our calculation
due to two resolvents $G$ and $\vG$ arises
from the interpolation $\X_\lambda$ and the Taylor expansion of $K'$.

We establish Theorem \ref{thm:resolventcompare} using the above two results:
\begin{proof}[Proof of Theorem \ref{thm:resolventcompare}]
We write
\begin{equation}\label{eq:KXKcX}
K(\X)-K(\vX)=\int_0^1 \frac{d}{d\lambda} K(\X_\lambda) d\lambda
=\int_0^1 K'(\X_\lambda)(\X-\vX) d\lambda.
\end{equation}
Recalling $\X=\sum_i \Im \int \tG_{ii}$
and applying (\ref{eq:GijvGij}),
\[\X-\vX=\sum_i \Im \int \left(\sum_k \tG_{ik} \tvG_{ik}(E_*-\vE_*)
-\sum_\a \frac{\tG_{i\a}}{t_\a}\frac{\tvG_{i\a}}{\vt_\a}(t_\a-\vt_\a)\right).\]
($\tG$ and $\tvG$ denote $G$ and $\vG$ evaluated at the variable of integration
$\ty$.) Further applying (\ref{eq:Eplusdiff}),
Lemma \ref{lemma:resolventbounds}, and the trivial bound
$N^{-2/3+\eps} \prec \Psi^2$,
\[\X-\vX=\sum_i \Im \int \sum_\a (t_\a-\vt_\a)\left(s_\a \vs_\a \frac{1}{N}
\sum_k \tG_{ik} \tvG_{ik}-\frac{\tG_{i\a}}{t_\a}\frac{\tvG_{i\a}}{\vt_\a}
\right)+\O(\Psi^4).\]
Applying this to (\ref{eq:KXKcX}), taking the expectation,
exchanging orders of summation and integration, and noting that
$K'(\X_\lambda)$ is real,
\begin{align*}
&\E[K(\X)-K(\vX)]\\
&=\sum_i \sum_\a (t_\a-\vt_\a)
\Im \int \int_0^1 \E\Bigg[K'(\X_\lambda)
\Bigg(s_\a \vs_\a \frac{1}{N}
\sum_k \tG_{ik} \tvG_{ik}
-\frac{\tG_{i\a}}{t_\a}\frac{\tvG_{i\a}}{\vt_\a}\Bigg)
\Bigg]\,d\lambda\,d\ty+\O(\Psi^4),
\end{align*}
where the expectation of the remainder term is still $\O(\Psi^4)$ by Lemmas
\ref{lemma:expectationdomination} and \ref{lemma:crudeGXbound}.
Denoting by $\tX_3(i)$, $\tX_4(i)$, and $\tX_4^-(i)$ the quantities
$\X_3$, $\X_4$, and $\X_4^-$ defined by $\ty$ and the outer index of summation
$i$, Lemma \ref{lemma:decoupling} implies
\begin{align*}
&\E[K(\X)-K(\vX)]\\
&=\sum_i \sum_\a (t_\a-\vt_\a)
\Im \int (\cP_\a\E[\tX_3(i)]
-\tfrac{1}{3}\Q_\a\E[\tX_4(i)]-\tfrac{1}{3}
\cR_\a\E[\vX_4^-(i)])d\ty+\O(N^{1/3+\eps}\Psi^5),
\end{align*}
where the error is $N^{1/3+\eps}\Psi^5$ because
$\sum_\a |t_\a-\vt_\a| \leq C$ and the range of integration is contained
in $[-N^{-2/3+\eps},N^{-2/3+\eps}]$. We note, from the identity
(\ref{eq:4minuscancellation}) and the analogous cancellation for the other
two pairs of terms, that $\Im \int \tX_4^-(i) d\ty=0$,
so this term vanishes. Then, applying Lemma \ref{lemma:optical},
\begin{align}
&\E[K(\X)-K(\vX)]\nonumber\\
&=\sum_i \sum_\a (t_\a-\vt_\a)\left(\cP_\a
\frac{A_4-m_*^{-4}}{2}-\frac{\Q_\a}{3}\right)
\Im \int \E[\tX_4(i)] d\ty+\O(N^{1/3+\eps}\Psi^5).
\label{eq:finaleq}
\end{align}
Finally, applying Lemma \ref{lemma:sumrules}, we have
\begin{equation}\label{eq:finalsumrule}
\sum_\a (t_\a-\vt_\a)\left(\cP_\a
\frac{A_4-m_*^{-4}}{2}-\frac{\Q_\a}{3}\right) \leq C/N.
\end{equation}
Thus the first term of (\ref{eq:finaleq}) is of size
$\O(N \cdot 1/N \cdot N^{-2/3+\eps} \cdot \Psi^4)$, which is of smaller
order than the remainder $N^{1/3+\eps}\Psi^5$. (In 
\cite{leeschnelli} for the differential version of Lemma \ref{lemma:decoupling},
this first term is zero due to the exact cancellation
of the analogue of (\ref{eq:finalsumrule}).)
Hence $\E[K(\X)-K(\vX)] \prec N^{1/3+\eps}\Psi^5=N^{-4/3+16\eps}$.
\end{proof}

\subsection{Proof of decoupling lemma}\label{subsec:decoupling}
In this section, we prove Lemma \ref{lemma:decoupling}.
We will implicitly use the resolvent bounds of
Lemma \ref{lemma:resolventbounds} throughout.

{\bf Step 1:} Consider first a fixed value $\lambda \in [0,1]$.
Let $\E_\a$ denote the partial expectation over row $\a$ of $X$
(i.e.\ conditional on all $X_{\beta j}$ for $\beta \neq \alpha$).
In anticipation of computing $\E_\a$ for the quantity on the left,
we expand
\[K'(\X_\lambda)\frac{G_{i\a}}{t_\a}\frac{\vG_{i\a}}{\vt_\a}\]
as a polynomial of entries of row $\a$ of $X$,
with coefficients independent of all entries in this row.

Applying the resolvent identities,
\[\frac{G_{i\a}}{t_\a}=\frac{G_{\a\a}}{t_\a}\sum_k G_{ik}\pa X_{\a k}
=-\frac{1}{1+t_\a \sum_{p,q} G_{pq}\pa X_{\a p}X_{\a q}}\sum_k G_{ik}\pa
X_{\a k}.\]
Applying Lemma \ref{lemma:concentration} and
a Taylor expansion of the function $(1+t_\a x)^{-1}$ around $x=m_*$,
\begin{align}
\frac{G_{i\a}}{t_\a}&=-s_\a\sum_k G_{ik}\pa X_{\a k}
+t_\a s_\a^2\left(\sum_{p,q} G_{pq}\pa X_{\a p}X_{\a
q}-m_*\right)\sum_k G_{ik}\pa X_{\a k}\nonumber\\
&\quad -t_\a^2s_\a^3
\left(\sum_{p,q} G_{pq}\pa X_{\a p}X_{\a q}-m_*\right)^2
\sum_k G_{ik}\pa X_{\a k}+\O(\Psi^4)\nonumber\\
&\equiv U_1+U_2+U_3+\O(\Psi^4),\label{eq:Giata}
\end{align}
where we defined the three explicit terms of sizes $\O(\Psi),\O(\Psi^2),\O(\Psi^3)$
as $U_1,U_2,U_3$. Similarly
\begin{align}
\frac{\vG_{i\a}}{\vt_\a}=\vU_1+\vU_2+\vU_3+\O(\Psi^4),\label{eq:vGiata}
\end{align}
where $\vU_i$ are defined analogously with $\vs_\a,\vt_\a,\vm_*,\vG$ in place of
$s_\a,t_\a,m_*,G$.

For $K'(\X_\lambda)$, define
$\X_\lambda\pa=\lambda\X\pa+(1-\lambda)\vX\pa$ and note from Lemma
\ref{lemma:removesuper} that
$\X_\lambda-\X_\lambda\pa \prec \Psi$. Taylor expanding $K'(x)$ around
$x=\X_\lambda\pa$,
\begin{align}
K'(\X_\lambda)&=K'(\X_\lambda\pa)+K''(\X_\lambda\pa)
(\X_\lambda-\X_\lambda\pa)
+\frac{K'''(\X_\lambda\pa)}{2}
(\X_\lambda-\X_\lambda\pa)^2+\O(\Psi^3).\label{eq:KXlambda}
\end{align}
Applying the definition of $\X,\X\pa$ and the resolvent identities,
\begin{align*}
\X-\X\pa&=\Im \int \sum_j(\tG_{jj}-\tG_{jj}\pa)
=\Im \int \sum_j \frac{\tG_{j\a}^2}{\tG_{\a\a}}
=\Im \int \tG_{\a\a} \sum_{j,p,q} \tG_{jp}\pa X_{\a p}
\tG_{jq}\pa X_{\a q}.
\end{align*}
Further applying the resolvent identity for $\tG_{\a\a}$, a Taylor
expansion as above, and Lemma \ref{lemma:integralbound},
\begin{align}
\X-\X\pa&= -t_\a s_\a \Im \int \sum_{j,p,q}\tG_{jp}\pa X_{\a p}
\tG_{jq}\pa X_{\a q}\nonumber\\
&\quad +t_\a^2 s_\a^2 \Im \int \sum_{r,s} \left(\tG_{rs}\pa X_{\a r}X_{\a
s}-m_*\right)\sum_{j,p,q}\tG_{jp}\pa X_{\a p}
\tG_{jq}\pa X_{\a q}+\O(\Psi^3)\nonumber\\
&\equiv V_1+V_2+\O(\Psi^3),\label{eq:XXa}
\end{align}
where $V_1 \prec \Psi$ and $V_2 \prec \Psi^2$.
Analogously we may write
\begin{equation}
\vX-\vX\pa=\vV_1+\vV_2+\O(\Psi^3),\label{eq:cXXa}
\end{equation}
where $\vV_1,\vV_2$ are defined with $\vs_\a,\vt_\a,\vm_*,\vG$
in place of $s_\a,t_\a,m_*,G$.
Substituting (\ref{eq:XXa}) and (\ref{eq:cXXa})
into (\ref{eq:KXlambda}), and combining with
(\ref{eq:Giata}) and (\ref{eq:vGiata}), we obtain
\begin{equation}\label{eq:Wexpansion}
K'(\X_\lambda)\frac{G_{i\a}}{t_\a}\frac{\vG_{i\a}}{\vt_\a}
=W_2+W_3+W_4+\O(\Psi^5)
\end{equation}
where the $\O(\Psi^2),\O(\Psi^3),\O(\Psi^4)$ terms are respectively
\begin{align*}
    W_2&=K'(\X_\lambda\pa)U_1\vU_1,\\
    W_3&=K'(\X_\lambda\pa)(U_2\vU_1+U_1\vU_2)+K''(\X_\lambda\pa)
(\lambda V_1+(1-\lambda)\vV_1)U_1\vU_1,\\
W_4&=K'(\X_\lambda\pa)(U_3\vU_1+U_2\vU_2+U_1\vU_3)
+K''(\X_\lambda\pa)(\lambda V_1+(1-\lambda) \vV_1)
(U_2\vU_1+U_1\vU_2)\\
&\quad +\left[K''(\X_\lambda\pa)(\lambda V_2+(1-\lambda)\vV_2)
+\frac{K'''(\X_\lambda\pa)}{2}(\lambda V_1
+(1-\lambda) \vV_1)^2\right]U_1\vU_1.
\end{align*}

{\bf Step 2:} We compute $\E_\a$ of $W_2,W_3,W_4$ above.
Note that $\X\pa,\vX\pa,G\pa,\vG\pa$ are independent of
row $\a$ of $X$. Then for $W_2$, we have
\begin{align}
\E_\a[W_2]&=s_\a \vs_\a K'(\X_\lambda\pa)
\sum_{k,l} G_{ik}\pa \vG_{il}\pa \E_\a[X_{\a k}X_{\a l}]\nonumber\\
&=s_\a\vs_\a K'(\X_\lambda\pa)\frac{1}{N}\sum_k G_{ik}\pa
\vG_{ik}\pa,\label{eq:EW2}
\end{align}
where we have used $\E[X_{\a k}X_{\a l}]=1/N$ if $k=l$ and 0 otherwise.

For $W_3$, let us introduce
\begin{align*}
\Y_{3,12'}\pa&=K'(\X_\lambda\pa)(m_N\pa-m_*)
\frac{1}{N}\sum_k G_{ik}\pa \vG_{ik}\pa,\\
\cZ_{3,12'}\pa&=K'(\X_\lambda\pa)(\vm_N\pa-\vm_*)
\frac{1}{N}\sum_k G_{ik}\pa \vG_{ik}\pa,\\
\Y_{3,3}\pa&=K'(\X_\lambda\pa)
\frac{1}{N^2}\sum_{k,l} G_{ik}\pa G_{kl}\pa \vG_{il}\pa\\
\cZ_{3,3}\pa&=K'(\X_\lambda\pa)\frac{1}{N^2}\sum_{k,l} G_{ik}\pa
\vG_{kl}\pa\vG_{il}\pa\\
\Y_{3,2'\t{2}'}\pa&=K''(\X_\lambda\pa)\frac{1}{N^2}\sum_{j,k,l} G_{ik}\pa
\vG_{ik}\pa \Im \int (\tG_{jl}\pa)^2\\
\cZ_{3,2'\t{2}'}\pa&=K''(\X_\lambda\pa)\frac{1}{N^2}\sum_{j,k,l} G_{ik}\pa
\vG_{ik}\pa \Im \int (\tvG_{jl}\pa)^2\\
\Y_{3,2\t{2}}\pa&=K''(\X_\lambda\pa)\frac{1}{N^2}\sum_{j,k,l} G_{ik}\pa
\vG_{il}\pa \Im \int \tG_{jk}\pa\tG_{jl}\pa\\
\cZ_{3,2\t{2}}\pa&=K''(\X_\lambda\pa)\frac{1}{N^2}\sum_{j,k,l}
G_{ik}\pa\vG_{il}\pa \Im \int \tvG_{jk}\pa\tvG_{jl}\pa,
\end{align*}
which are versions of $\X_{3,*}$ that don't depend on row $\a$ of $X$
and with various
instances of $m_N,m_*,G,\X$ replaced by $\vm_N,\vm_*,\vG,\X_\lambda$.
Consider the first term of $W_3$ and write
\begin{align*}
&\E_\a[K'(\X_\lambda\pa)U_2\vU_1]\\
&=\E_\a\left[-t_\a s_\a^2 \vs_\a
K'(\X_\lambda\pa)
\left(\sum_{p,q} G_{pq}\pa X_{\a p}X_{\a q}-m_*\right)
\sum_{k,l}G_{ik}\pa X_{\a k}\vG_{il}\pa X_{\a l}\right]\\
&=-t_\a s_\a^2 \vs_\a K'(\X_\lambda\pa)\sum_{k,l,p,q}
\Bigg(G_{pq}\pa \E_\a[X_{\a p}X_{\a q}
X_{\a k}X_{\a l}]
-\frac{1}{N}m_*\1\{p=q\}\E_\a[X_{\a k}X_{\a l}]
\Bigg)G_{ik}\pa \vG_{il}\pa.
\end{align*}
The summand corresponding to $(k,l,p,q)$ is 0 unless each distinct
index appears at least twice in $(k,l,p,q)$. Furthermore, the case where
all four indices are equal is negligible:
\[\sum_k \left(G_{kk}\pa \E_\a[X_{\a k}^4]
-\frac{1}{N}m_*\E_\a[X_{\a k}^2]\right)G_{ik}\pa \vG_{ik}\pa
\prec N \cdot N^{-2} \cdot \Psi^2 \prec \Psi^5.\]
(The $k=i$ case of the sum may be bounded separately as $\O(N^{-2})$.)
Thus up to $\O(\Psi^5)$, we need only consider summands
where each distinct index appears
exactly twice. Considering the one case where $k=l$ and the two cases where
$k=p$ and $k=q$,
\begin{align*}
&\E_\a[K'(\X_\lambda\pa)U_2\vU_1]\\
&=-t_\a s_\a^2 \vs_\a
K'(\X_\lambda\pa)\Bigg(
\frac{1}{N^2}\sum_k \sum_p^{(k)} \left(G_{pp}\pa-m_*\right)
G_{ik}\pa \vG_{ik}\pa
+\frac{2}{N^2}\sum_k \sum_l^{(k)} G_{ik}\pa \vG_{il}\pa
G_{kl}\pa\Bigg)+\O(\Psi^5).
\end{align*}
Re-including $p=k$ and $l=k$ into the double summations
introduces an additional $\O(\Psi^5)$ error; hence
we obtain for the first term of $W_3$
\begin{align}
\E_\a[K'(\X_\lambda\pa)U_2\vU_1]=-t_\a s_\a^2 \vs_\a
(\Y_{3,12'}\pa+2\Y_{3,3}\pa)+\O(\Psi^5).\label{eq:W31}
\end{align}

Similar arguments apply for the remaining three terms of $W_3$. For the
terms involving an integral, we may apply Lemma \ref{lemma:integralbound} and
also move $X_{\a k}$ outside of the integral and imaginary part because $X$ is
real and does not depend on the variable of integration $\ty$. We obtain
\begin{align}
\E_\a[K'(\X_\lambda\pa)U_1\vU_2]&=-\vt_\a \vs_\a^2 s_\a
(\cZ_{3,12'}\pa+2\cZ_{3,3}\pa)+\O(\Psi^5),\\
\E_\a[\lambda K''(\X_\lambda\pa)V_1U_1\vU_1]&=-\lambda t_\a s_\a^2 \vs_\a
(\Y_{3,2'\t{2}'}\pa+2\Y_{3,2\t{2}}\pa)+\O(\Psi^5),\\
\E_\a[(1-\lambda)K''(\X_\lambda\pa)\vV_1U_1\vU_1]&=-(1-\lambda)
\vt_\a \vs_\a^2 s_\a (\cZ_{3,2'\t{2}'}\pa+2\cZ_{3,2\t{2}}\pa)+\O(\Psi^5),
\label{eq:W34}
\end{align}
and $\E_\a[W_3]$ is the sum of (\ref{eq:W31}--\ref{eq:W34}).

For $W_4$, consider the first term and write
\begin{align*}
&\E_\a[K'(\X_\lambda\pa)U_3\vU_1]\\
&=\E_\a\left[t_\a^2 s_\a^3\vs_\a K'(\X_\lambda\pa)
\left(\sum_{p,q} G_{pq}\pa X_{\a p}X_{\a q}-m_*\right)^2
\sum_{k,l} G_{ik}\pa X_{\a k}\vG_{il}\pa X_{\a l}\right]\\
&=t_\a^2 s_\a^3\vs_\a K'(\X_\lambda\pa)
\sum_{p,q,r,s,k,l}\Bigg(G_{pq}\pa G_{rs}\pa
\E_\a[X_{\a p}X_{\a q}X_{\a r}X_{\a s}X_{\a k}X_{\a l}]\\
&\hspace{0.2in}-\frac{1}{N}m_*\1\{p=q\}G_{rs}\pa
\E_\a[X_{\a r}X_{\a s}X_{\a k}X_{\a l}]
-\frac{1}{N}m_*\1\{r=s\}G_{pq}\pa
\E_\a[X_{\a p}X_{\a q}X_{\a k}X_{\a l}]\\
&\hspace{0.2in}+\frac{1}{N^2}m_*^2\1\{p=q\}\1\{r=s\}\E[X_{\a k}X_{\a l}]
\Bigg)G_{ik}\pa \vG_{il}\pa.
\end{align*}
A summand corresponding to $(k,l,p,q,r,s)$ is 0 unless each distinct index in
$(k,l,p,q,r,s)$ appears at least twice. Furthermore, as in the computations for
$W_3$ above, all summands for which $(k,l,p,q,r,s)$ do not form three distinct
pairs may be omitted and reincluded after taking $\E_\a$,
introducing an $\O(\Psi^5)$ error. Considering all pairings of these indices,
\begin{align*}
&\E_\a[K'(\X_\lambda\pa)U_3\vU_1]\\
&=t_\a^2 s_\a^3 \vs_\a K'(\X_\lambda\pa)\Bigg(
(m_N\pa-m_*)^2\frac{1}{N}\sum_k G_{ik}\pa \vG_{ik}\pa
+4(m_N\pa-m_*)\frac{1}{N^2}\sum_{k,l} G_{ik}\pa G_{kl}\pa \vG_{il}\pa\\
&\hspace{0.2in}+8\frac{1}{N^3}\sum_{j,k,l} G_{ik}\pa G_{jk}\pa G_{jl}\pa
\vG_{il}\pa
+2\frac{1}{N^3}\sum_{j,k,l} G_{ik}\pa \vG_{ik}\pa (G_{jl}\pa)^2\Bigg)
+\O(\Psi^5).
\end{align*}

At this point, let us apply Lemmas \ref{lemma:removesuper} and
\ref{lemma:removecheck} to remove each superscript $(\a)$ above and to convert
each $\vG$ to $G$, introducing an $\O(\Psi^5)$ error.
We may also remove the superscript $(\a)$ and convert
$\X_\lambda$ to $\X$ in $K'(\X_\lambda\pa)$, via
the second-derivative bounds
\[K'(\X_\lambda\pa)-K'(\X_\lambda) \leq \|K''\|_\infty
|\X_\lambda\pa-\X_\lambda| \prec \Psi.\]
\[K'(\X_\lambda)-K'(\X) \leq \|K''\|_\infty
|\X_\lambda-\X| \prec \Psi.\]
We thus obtain
\[\E_\a[K'(\X_\lambda\pa)U_3\vU_1]
=t_\a^2s_\a^3 \vs_\a(\X_{4,22'}+4\X_{4,13}
+8\X_{4,4}+2\X_{4,4'})+\O(\Psi^5).\]

Applying a similar computation to each term of $W_4$, we obtain
\begin{align}
&\E_\a[K'(\X_\lambda\pa)(U_3\vU_1+U_2\vU_2+U_1\vU_3)]\nonumber\\
&=s_\a\vs_\a(t_\a^2s_\a^2+t_\a
s_\a\vt_\a\vs_\a+\vt_\a^2\vs_\a^2)
(\X_{4,22'}+4\X_{4,13}+8\X_{4,4}+2\X_{4,4'})+\O(\Psi^5),\label{eq:W41}\\
&\E_\a[K''(\X_\lambda\pa)(\lambda
V_1+(1-\lambda)\vV_1)(U_2\vU_1+U_1\vU_2)]\nonumber\\
&=s_\a \vs_\a(\lambda t_\a s_\a+(1-\lambda)\vt_\a \vs_\a)(t_\a s_\a+\vt_\a
\vs_\a)\cdot\nonumber\\
&\qquad(\X_{4,12'\t{2}'}+2\X_{4,12\t{2}}+2\X_{4,3\t{2}'}+2\X_{4,3'\t{2}}
+8\X_{4,3\t{2}})+\O(\Psi^5),\\
&\E_\a[K''(\X_\lambda\pa)(\lambda V_2+(1-\lambda)\vV_2)U_1\vU_1]\nonumber\\
&=s_\a \vs_\a(\lambda t_\a^2 s_\a^2+(1-\lambda)
\vt_\a^2\vs_\a^2)\cdot\nonumber\\
&\qquad(\X_{4,2'\t{12}'}+2\X_{4,2\t{12}}+2\X_{4,2'\t{3}}
+2\X_{4,2\t{3}'}+8\X_{4,2\t{3}})+\O(\Psi^5),\\
&\E_\a\left[\frac{K'''(\X_\lambda\pa)}{2}(\lambda V_1
+(1-\lambda) \vV_1)^2U_1\vU_1\right]\nonumber\\
&=\frac{s_\a \vs_\a}{2}(\lambda t_\a s_\a+(1-\lambda) \vt_\a
\vs_\a)^2\cdot\nonumber\\
&\qquad (\X_{4,2'\t{2}'\t{2}'}+2\X_{4,2'\t{2}\t{2}}+4\X_{4,2\t{2}\t{2}'}
+8\X_{4,2\t{2}\t{2}})+\O(\Psi^5),\label{eq:W44}
\end{align}
and $\E_\a[W_4]$ is the sum of (\ref{eq:W41}--\ref{eq:W44}).

The $\O(\Psi^5)$ remainder in (\ref{eq:Wexpansion}) is given by the difference
of the left side with $W_2,W_3,W_4$. As this is an integral over a
polynomial of entries of $G\pa$ and $X$, its partial expectation is still
$\O(\Psi^5)$ by Lemmas \ref{lemma:expectationdomination} and
\ref{lemma:crudeGXbound}.

Summarizing the results of Steps 1 and 2, we
collect (\ref{eq:Wexpansion}), (\ref{eq:EW2}),
(\ref{eq:W31}--\ref{eq:W34}), and (\ref{eq:W41}--\ref{eq:W44}):
{\allowdisplaybreaks
\begin{align}
&\E_\a\left[K'(\X_\lambda)\frac{G_{i\a}}{t_\a}\frac{\vG_{i\a}}{\vt_\a}\right]
\nonumber\\
&=s_\a\vs_\a K'(\X_\lambda\pa)\frac{1}{N}\sum_k G_{ik}\pa\vG_{ik}\pa
-t_\a s_\a^2 \vs_\a (\Y_{3,12'}\pa+2\Y_{3,3}\pa)
-\vt_\a \vs_\a^2 s_\a (\cZ_{3,12'}\pa+2\cZ_{3,3}\pa)\nonumber\\
&\;\;\;\;-\lambda t_\a s_\a^2 \vs_\a(\Y_{3,2'\t{2}'}\pa+2\Y_{3,2\t{2}}\pa)
-(1-\lambda) \vt_\a \vs_\a^2
s_\a(\cZ_{3,2'\t{2}'}\pa+2\cZ_{3,2\t{2}}\pa)\nonumber\\
&\;\;\;\;+s_\a\vs_\a(t_\a^2s_\a^2+t_\a s_\a\vt_\a\vs_\a+\vt_\a^2\vs_\a^2)
(\X_{4,22'}+4\X_{4,13}+8\X_{4,4}+2\X_{4,4'})\nonumber\\
&\;\;\;\;+s_\a \vs_\a(\lambda t_\a s_\a+(1-\lambda)\vt_\a \vs_\a)(t_\a s_\a
+\vt_\a \vs_\a)
(\X_{4,12'\t{2}'}+2\X_{4,12\t{2}}+2\X_{4,3\t{2}'}+2\X_{4,3'\t{2}}
+8\X_{4,3\t{2}})\nonumber\\
&\;\;\;\;+s_\a \vs_\a(\lambda t_\a^2 s_\a^2+(1-\lambda)
\vt_\a^2\vs_\a^2)
(\X_{4,2'\t{12}'}+2\X_{4,2\t{12}}+2\X_{4,2'\t{3}}
+2\X_{4,2\t{3}'}+8\X_{4,2\t{3}})\nonumber\\
&\;\;\;\;+\frac{s_\a \vs_\a}{2}(\lambda t_\a s_\a+(1-\lambda) \vt_\a \vs_\a)^2
(\X_{4,2'\t{2}'\t{2}'}+2\X_{4,2'\t{2}\t{2}}
+4\X_{4,2\t{2}\t{2}'}+8\X_{4,2\t{2}\t{2}})+\O(\Psi^5).\label{eq:Etotal1}
\end{align}}

{\bf Step 3:} In (\ref{eq:Etotal1}), we consider the first term on the right
(of size $\O(\Psi^2)$) and remove the superscripts $(\a)$,
keeping track of the $\O(\Psi^3)$ and $\O(\Psi^4)$ terms that arise.

Applying the resolvent identities and a Taylor expansion for $G_{\a\a}$, we
write
\begin{align}
G_{ik}\pa&=G_{ik}-\frac{G_{i\a}G_{k\a}}{G_{\a\a}}\nonumber\\
&=G_{ik}-G_{\a\a}\sum_{r,s} G_{ir}\pa X_{\a r}G_{ks}\pa X_{\a s}\nonumber\\
&=G_{ik}+t_\a s_\a\sum_{r,s} G_{ir}
\pa X_{\a r}G_{ks}\pa X_{\a s}\nonumber\\
&\quad-t_\a^2s_\a^2\left(\sum_{p,q} G_{pq}\pa X_{\a p}
X_{\a q}-m_*\right)\sum_{r,s} G_{ir}\pa X_{\a r}G_{ks}\pa X_{\a s}
+\O(\Psi^4)\nonumber\\
&\equiv G_{ik}+R_{2k}+R_{3k}+\O(\Psi^4),\label{eq:EW21}
\end{align}
where we defined the two remainder terms of sizes $\O(\Psi^2),\O(\Psi^3)$ as
$R_{2k},R_{3k}$. Similarly we write
\begin{align}
\vG_{ik}\pa=\vG_{ik}+\vR_{2k}+\vR_{3k}+\O(\Psi^4).\label{eq:EW22}
\end{align}
For $K'(\X_\lambda\pa)$, we apply the Taylor expansion
(\ref{eq:KXlambda}) and recall $V_1,\vV_1,V_2,\vV_2$
from (\ref{eq:XXa},\ref{eq:cXXa}) to obtain
\begin{align}
K'(\X_\lambda\pa)
&=K'(\X_\lambda)-K''(\X_\lambda\pa)
(\X_\lambda-\X_\lambda\pa)-\frac{K'''(\X_\lambda\pa)}{2}
(\X_\lambda-\X_\lambda\pa)^2+\O(\Psi^3)\nonumber\\
&=K'(\X_\lambda)-K''(\X_\lambda\pa)(\lambda V_1+(1-\lambda) \vV_1)
-K''(\X_\lambda\pa)(\lambda V_2+(1-\lambda) \vV_2)\nonumber\\
&\qquad-\frac{K'''(\X_\lambda\pa)}{2}(\lambda V_1+(1-\lambda) \vV_1)^2
+\O(\Psi^3).\label{eq:EW23}
\end{align}

Taking the product of (\ref{eq:EW21}), (\ref{eq:EW22}), and (\ref{eq:EW23}),
applying the identity
\[xyz=(x-\delta_x)(y-\delta_y)(z-\delta_z)
+xy\delta_z+x\delta_y z+\delta_x yz-x\delta_y\delta_z-\delta_x y\delta_z
-\delta_x\delta_y z+\delta_x\delta_y\delta_z\]
(with $x=G_{ik}\pa$, $x-\delta_x=G_{ik}$, and $\delta_x=R_{2k}+R_{3k}$, etc.),
and averaging over $k \in \I_N$, we obtain
\begin{align}
K'(\X_\lambda\pa)\frac{1}{N}\sum_k G_{ik}\pa\vG_{ik}\pa
&\equiv S_2+S_{3,1}+S_{3,2}+\sum_{j=1}^5 S_{4,j}+\O(\Psi^5),\label{eq:W2decomp}
\end{align}
where
\begin{align*}
    S_2&=K'(\X_\lambda)\frac{1}{N}\sum_k G_{ik}\vG_{ik},\\
S_{3,1}&=K'(\X_\lambda\pa)\frac{1}{N}\sum_k G_{ik}\pa \vR_{2k}+K'(\X_\lambda\pa)
\frac{1}{N}\sum_k R_{2k}\vG_{ik}\pa,\\
S_{3,2}&=-K''(\X_\lambda\pa)(\lambda V_1+(1-\lambda) \vV_1)
\frac{1}{N}\sum_k G_{ik}\pa \vG_{ik}\pa,\\
S_{4,1}&=K'(\X_\lambda\pa)\frac{1}{N}\sum_k G_{ik}\pa \vR_{3k}
+K'(\X_\lambda\pa)\frac{1}{N}\sum_k R_{3k}\vG_{ik}\pa,\\
S_{4,2}&=-K''(\X_\lambda\pa)(\lambda V_2+(1-\lambda)\vV_2)\frac{1}{N}\sum_k
G_{ik}\pa \vG_{ik}\pa,\\
S_{4,3}&=-\frac{K'''(\X_\lambda\pa)}{2}(\lambda V_1+(1-\lambda) \vV_1)^2
\frac{1}{N}\sum_k G_{ik}\pa \vG_{ik}\pa,\\
S_{4,4}&=-K'(\X_\lambda\pa)\frac{1}{N}\sum_k R_{2k}\vR_{2k},\\
S_{4,5}&=K''(\X_\lambda\pa)(\lambda V_1+(1-\lambda)\vV_1)
\frac{1}{N}\sum_k G_{ik}\pa\vR_{2k}\\
&\qquad+K''(\X_\lambda\pa)(\lambda V_1+(1-\lambda)\vV_1)
\frac{1}{N}\sum_k R_{2k}\vG_{ik}\pa.
\end{align*}

Recalling the definition of $R_{2k}$ and applying $\E_\a$ to the $\O(\Psi^3)$
terms,
\begin{align*}
\E_\a[S_{3,1}]&=t_\a s_\a\Y_{3,3}\pa+\vt_\a \vs_\a\cZ_{3,3}\pa,\\
\E_\a[S_{3,2}]&=\lambda t_\a s_\a\Y_{3,2'\t{2}'}\pa
+(1-\lambda)\vt_\a \vs_\a \cZ_{3,2'\t{2}'}\pa.
\end{align*}
Similarly, we apply $\E_\a$ to each of the $\O(\Psi^4)$ terms,
considering all pairings
of the four summation indices as in Step 2. Then applying Lemmas
\ref{lemma:removesuper} and \ref{lemma:removecheck} to remove superscripts
and convert $\vG$ to $G$, we obtain
\begin{align*}
\E_\a[S_{4,1}]&=-(t_\a^2s_\a^2+\vt_\a^2\vs_\a^2)
(\X_{4,13}+2\X_{4,4})+\O(\Psi^5),\\
\E_\a[S_{4,2}]&=-(\lambda t_\a^2s_\a^2+(1-\lambda)\vt_\a^2 \vs_\a^2)
(\X_{4,2'\t{12}'}+2\X_{4,2'\t{3}})+\O(\Psi^5),\\
\E_\a[S_{4,3}]&=-\frac{1}{2}(\lambda t_\a s_\a
+(1-\lambda)\vt_\a\vs_\a)^2
(\X_{4,2'\t{2}'\t{2}'}+2\X_{4,2'\t{2}\t{2}})+\O(\Psi^5),\\
\E_\a[S_{4,4}]&=-t_\a s_\a\vt_\a\vs_\a (\X_{4,4'}+2\X_{4,4})+\O(\Psi^5),\\
\E_\a[S_{4,5}]&=-(\lambda t_\a s_\a+(1-\lambda)\vt_\a\vs_\a)
(t_\a s_\a+\vt_\a \vs_\a)(\X_{4,3\t{2}'}+2\X_{4,3\t{2}})+\O(\Psi^5).
\end{align*}
Then applying $\E_\a$ to (\ref{eq:W2decomp}), noting that the remainder is again
$\O(\Psi^5)$ by Lemmas \ref{lemma:expectationdomination} and
\ref{lemma:crudeGXbound}, and substituting into (\ref{eq:Etotal1}),
\begin{align}
&\E_\a\left[K'(\X_\lambda)\frac{G_{i\a}}{t_\a}\frac{\vG_{i\a}}{\vt_\a}\right]\nonumber\\
&=s_\a\vs_\a \E_\a\left[K'(\X_\lambda)\frac{1}{N}\sum_k G_{ik}
\vG_{ik}\right]
-t_\a s_\a^2 \vs_\a(\Y_{3,12'}\pa+\Y_{3,3}\pa)\nonumber\\
&\quad-\vt_\a \vs_\a^2 s_\a(\cZ_{3,12'}\pa+\cZ_{3,3}\pa)
-2\lambda t_\a s_\a^2 \vs_\a \Y_{3,2\t{2}}\pa
-2(1-\lambda) \vt_\a \vs_\a^2 s_\a \cZ_{3,2\t{2}}\pa\nonumber\\
&\quad+s_\a\vs_\a(t_\a^2s_\a^2+\vt_\a^2\vs_\a^2)
(\X_{4,22'}+3\X_{4,13}+6\X_{4,4}+2\X_{4,4'})\nonumber\\
&\quad+s_\a\vs_\a(t_\a s_\a\vt_\a \vs_\a)
(\X_{4,22'}+4\X_{4,13}+6\X_{4,4}+\X_{4,4'})\nonumber\\
&\quad+s_\a \vs_\a(\lambda t_\a s_\a+(1-\lambda)\vt_\a \vs_\a)(t_\a s_\a+\vt_\a
\vs_\a)
(\X_{4,12'\t{2}'}+2\X_{4,12\t{2}}+\X_{4,3\t{2}'}+2\X_{4,3'\t{2}}
+6\X_{4,3\t{2}})\nonumber\\
&\quad+s_\a \vs_\a(\lambda t_\a^2 s_\a^2+(1-\lambda) \vt_\a^2\vs_\a^2)
(2\X_{4,2\t{12}}+2\X_{4,2\t{3}'}+8\X_{4,2\t{3}})\nonumber\\
&\quad+\frac{s_\a \vs_\a}{2}(\lambda t_\a s_\a+(1-\lambda) \vt_\a \vs_\a)^2
(4\X_{4,2\t{2}\t{2}'}
+8\X_{4,2\t{2}\t{2}})+\O(\Psi^5).\label{eq:Etotal2}
\end{align}

{\bf Step 4:} In (\ref{eq:Etotal2}), we remove the superscript $(\a)$ from
$\Y_{3,*}$ and $\cZ_{3,*}$, keeping track of the $\O(\Psi^4)$ errors that
arise. For each quantity $\Y_{3,*}\pa$ or $\cZ_{3,*}\pa$, let
$\Y_{3,*}$ or $\cZ_{3,*}$ be the analogous quantity with each instance of
$m_N\pa,G\pa,\tG\pa,\X_\lambda\pa$ replaced by $m_N,G,\tG,\X_\lambda$.

For $\Y_{3,12'}\pa$, recall from (\ref{eq:EW21})
and (\ref{eq:EW23}) that
\begin{align*}
    G_{ik}\pa&=G_{ik}+R_{2k}+\O(\Psi^3),\\
    K'(\X_\lambda\pa)&=K'(\X_\lambda)-K''(\X_\lambda\pa)(\lambda
V_1+(1-\lambda)\vV_1)+\O(\Psi^2).
\end{align*}
For $m_N\pa-m_*$, we apply the resolvent identities and write
\begin{align*}
m_N\pa-m_*&=m_N-m_*-\frac{1}{N}\sum_j \frac{G_{j\a}^2}{G_{\a\a}}\\
&=m_N-m_*-G_{\a\a}\frac{1}{N}\sum_{j,k,l} G_{jk}\pa X_{\a k}G_{jl}\pa X_{\a
l}\\
&=m_N-m_*+t_\a s_\a\frac{1}{N}\sum_{j,k,l} G_{jk}\pa X_{\a k}G_{jl}\pa X_{\a
l}+\O(\Psi^3)\\
&\equiv m_N-m_*+Q+\O(\Psi^3),
\end{align*}
where $Q$ is the $\O(\Psi^2)$ term. Multiplying the above
and averaging over $k$,
\begin{align*}
\Y_{3,12'}\pa&=\Y_{3,12'}+K'(\X_\lambda\pa)(m_N\pa-m_*)\frac{1}{N}\sum_k
G_{ik}\pa \vR_{2k}\\
&\quad+K'(\X_\lambda\pa)(m_N\pa-m_*)\frac{1}{N}\sum_k \vG_{ik}\pa R_{2k}
+K'(\X_\lambda\pa)Q\frac{1}{N}\sum_k G_{ik}\pa \vG_{ik}\pa\\
&\quad-K''(\X_\lambda\pa)(\lambda V_1+(1-\lambda)\vV_1)(m_N\pa-m_*)\frac{1}{N}
\sum_k G_{ik}\pa \vG_{ik}\pa+\O(\Psi^5),
\end{align*}
where each term except $\Y_{3,12'}$ on the right is of size $\O(\Psi^4)$.
Taking $\E_\a$ and applying
Lemmas \ref{lemma:removesuper} and \ref{lemma:removecheck} to remove
superscripts and checks,
\begin{align}
\Y_{3,12'}\pa&=\E_\a[\Y_{3,12'}]+(t_\a s_\a+\vt_\a \vs_\a)\X_{4,13}
+t_\a s_\a\X_{4,4'}\nonumber\\
&\qquad+(\lambda t_\a s_\a+(1-\lambda) \vt_\a \vs_\a)
\X_{4,12'\t{2}'}+\O(\Psi^5).
\end{align}
Similar arguments yield
\begin{align*}
\cZ_{3,12'}\pa&=\E_\a[\cZ_{3,12'}]+(t_\a s_\a+\vt_\a \vs_\a)\X_{4,13}
+\vt_\a \vs_\a\X_{4,4'}\nonumber\\
&\qquad+(\lambda t_\a s_\a+(1-\lambda) \vt_\a \vs_\a)
\X_{4,12'\t{2}'}+\O(\Psi^5),\\
\Y_{3,3}\pa&=\E_\a[\Y_{3,3}]+(2t_\a s_\a+\vt_\a \vs_\a)\X_{4,4}\nonumber\\
&\qquad+(\lambda t_\a s_\a+(1-\lambda) \vt_\a \vs_\a)
\X_{4,3\t{2}'}+\O(\Psi^5),\\
\cZ_{3,3}\pa&=\E_\a[\cZ_{3,3}]+(t_\a s_\a+2\vt_\a \vs_\a)\X_{4,4}\nonumber\\
&\qquad+(\lambda t_\a s_\a+(1-\lambda) \vt_\a \vs_\a)
\X_{4,3\t{2}'}+\O(\Psi^5),\\
\Y_{3,2\t{2}}\pa&=\E_\a[\Y_{3,2\t{2}}]+(t_\a s_\a+\vt_\a \vs_\a)
\X_{4,3\t{2}}+2t_\a s_\a \X_{4,2\t{3}}\nonumber\\
&\qquad+(\lambda t_\a s_\a+(1-\lambda) \vt_\a \vs_\a)
\X_{4,2\t{2}\t{2}'}+\O(\Psi^5),\\
\cZ_{3,2\t{2}}\pa&=\E_\a[\cZ_{3,2\t{2}}]+(t_\a s_\a+\vt_\a \vs_\a)
\X_{4,3\t{2}}+2\vt_\a \vs_\a \X_{4,2\t{3}}\nonumber\\
&\qquad+(\lambda t_\a s_\a+(1-\lambda) \vt_\a \vs_\a)
\X_{4,2\t{2}\t{2}'}+\O(\Psi^5).
\end{align*}

Substituting into (\ref{eq:Etotal2}),
\begin{align}
&\E_\a\left[K'(\X_\lambda)\frac{G_{i\a}}{t_\a}\frac{\vG_{i\a}}{\vt_\a}\right]\nonumber\\
&=s_\a\vs_\a \E_\a\left[K'(\X_\lambda)\frac{1}{N}\sum_k G_{ik}
\vG_{ik}\right]
-t_\a s_\a^2 \vs_\a\E_\a[\Y_{3,12'}+\Y_{3,3}]\nonumber\\
&\quad-\vt_\a \vs_\a^2 s_\a\E_\a[\cZ_{3,12'}+\cZ_{3,3}]
-2\lambda t_\a s_\a^2 \vs_\a \E_\a[\Y_{3,2\t{2}}]
-2(1-\lambda) \vt_\a \vs_\a^2 s_\a \E_\a[\cZ_{3,2\t{2}}]\nonumber\\
&\quad+s_\a\vs_\a(t_\a^2s_\a^2+t_\a s_\a \vt_\a \vs_\a+\vt_\a^2\vs_\a^2)
(\X_{4,22'}+2\X_{4,13}+4\X_{4,4}+\X_{4,4'})\nonumber\\
&\quad+s_\a \vs_\a(\lambda t_\a s_\a+(1-\lambda)\vt_\a \vs_\a)(t_\a s_\a+\vt_\a
\vs_\a)
(2\X_{4,12\t{2}}+2\X_{4,3'\t{2}}+4\X_{4,3\t{2}})\nonumber\\
&\quad+s_\a \vs_\a(\lambda t_\a^2 s_\a^2+(1-\lambda) \vt_\a^2\vs_\a^2)
(2\X_{4,2\t{12}}+2\X_{4,2\t{3}'}+4\X_{4,2\t{3}})\nonumber\\
&\quad+4s_\a \vs_\a(\lambda t_\a s_\a+(1-\lambda) \vt_\a \vs_\a)^2
\X_{4,2\t{2}\t{2}}+\O(\Psi^5).\label{eq:Etotal3}
\end{align}

{\bf Step 5:} We take the full expectation of both sides of (\ref{eq:Etotal3}),
applying Lemma \ref{lemma:removecheck2} to convert $\Y_{3,*}$ and $\cZ_{3,*}$
into $\X_{3,*}$. We illustrate the argument for $\cZ_{3,12'}$:
For $k \neq i$, denote
\[Y=K'(\X_\lambda)(\vm_N-\vm_*)G_{ik},\qquad
Y\pa=K'(\X_\lambda\pa)(\vm_N\pa-\vm_*)G_{ik}\pa.\]
Then $Y \prec \Psi^2$, and $Y-Y\pa \prec \Psi^3$ for all $\a \in \I_M$,
the latter from Lemma \ref{lemma:removesuper} and the second-derivative
bound for $K$. Then applying Lemma \ref{lemma:removecheck2},
$\E[Y \vG_{ik}]=\E[Y G_{ik}]+\O(\Psi^5)$. Hence
\begin{align}
\E\left[K'(\X_\lambda)(\vm_N-\vm_*)\frac{1}{N}\sum_k G_{ik}
(\vG_{ik}-G_{ik})\right]=\O(\Psi^5),\label{eq:W3first11}
\end{align}
where the $k=i$ term is controlled directly by Lemma \ref{lemma:removecheck}.
Applying this argument again with $Y=K'(\X_\lambda)G_{ik}^2$, 
together with the bound $\vm_*-m_* \leq C/N
\prec \Psi^3$, we may convert the term $\vm_N-\vm_*$:
\begin{align}
\E\left[K'(\X_\lambda)(\vm_N-\vm_*-m_N+m_*)\frac{1}{N}\sum_k G_{ik}^2
\right]=\O(\Psi^5).\label{eq:W3first12}
\end{align}
Finally, a Taylor expansion of $K'(x)$ around $\X$ yields
\begin{align}
K'(\X_\lambda)=K'(\X)+(1-\lambda)K''(\X)(\vX-\X)+\O(\Psi^2),\label{eq:W3first13}
\end{align}
where we have used $\vX-\X \prec \Psi$ by Lemma \ref{lemma:removecheck}.
Applying the third implication of Lemma \ref{lemma:removecheck2} with
$Y=K''(\X)(m_N-m_*)G_{ik}^2 \prec \Psi^3$ for $k \neq i$, we obtain
\begin{align}
\E\left[K''(\X)(\vX-\X)(m_N-m_*)
\frac{1}{N}\sum_k G_{ik}^2\right]=\O(\Psi^5).\label{eq:W3first14}
\end{align}
Then combining (\ref{eq:W3first11}--\ref{eq:W3first14}), we obtain
$\E[\cZ_{3,12'}]=\E[\X_{3,12'}]+\O(\Psi^5)$.

The same argument holds for the other terms $\Y_{3,*}$ and $\cZ_{3,*}$.
Then taking the full expectation of (\ref{eq:Etotal3}),
\begin{align}
&\E\left[K'(\X_\lambda)\frac{G_{i\a}}{t_\a}\frac{\vG_{i\a}}{\vt_\a}\right]\nonumber\\
&\quad=s_\a\vs_\a \E\left[K'(\X_\lambda)\frac{1}{N}\sum_k G_{ik}
\vG_{ik}\right]
-(t_\a s_\a^2 \vs_\a+\vt_\a \vs_\a^2 s_\a)
\E[\X_{3,12'}+\X_{3,3}]\nonumber\\
&\qquad-2(\lambda t_\a s_\a^2 \vs_\a
+(1-\lambda) \vt_\a \vs_\a^2 s_\a)\E[\X_{3,2\t{2}}]\nonumber\\
&\qquad+s_\a\vs_\a(t_\a^2s_\a^2+t_\a s_\a \vt_\a \vs_\a+\vt_\a^2\vs_\a^2)
\E[\X_{4,22'}+2\X_{4,13}+4\X_{4,4}+\X_{4,4'}]\nonumber\\
&\qquad+s_\a \vs_\a(\lambda t_\a s_\a+(1-\lambda)\vt_\a \vs_\a)(t_\a s_\a+\vt_\a
\vs_\a)
\E[2\X_{4,12\t{2}}+2\X_{4,3'\t{2}}+4\X_{4,3\t{2}}]\nonumber\\
&\qquad+s_\a \vs_\a(\lambda t_\a^2 s_\a^2+(1-\lambda) \vt_\a^2\vs_\a^2)
\E[2\X_{4,2\t{12}}+2\X_{4,2\t{3}'}+4\X_{4,2\t{3}}]\nonumber\\
&\qquad+4s_\a \vs_\a(\lambda t_\a s_\a+(1-\lambda) \vt_\a \vs_\a)^2
\E[\X_{4,2\t{2}\t{2}}]+\O(\Psi^5).\label{eq:Etotal4}
\end{align}

Finally, we integrate (\ref{eq:Etotal4}) over
$\lambda \in [0,1]$, applying $\int \lambda=\int(1-\lambda)=1/2$ and
$\int \lambda^2=\int 2\lambda(1-\lambda)=\int (1-\lambda)^2=1/3$.
Simplifying and identifying the terms
$\X_3$, $\X_4$, $\X_4^-$, $\cP_\a$, $\Q_\a$, and $\cR_\a$
concludes the proof of the lemma.

\subsection{Proof of optical theorems}\label{subsec:optical}
We discuss briefly the proof of Lemma \ref{lemma:optical}.
In the setting $K' \equiv 1$, Lemma \ref{lemma:optical} corresponds to
\cite[Lemma B.1]{leeschnelli} upon taking the imaginary part.

The proof for general $K$ is the same as that of \cite[Lemma B.1]{leeschnelli},
with additional terms arising from the Taylor expansion of $K'$ as in the proof
of Lemma \ref{lemma:decoupling}. The computation may be broken down into the
identities
\begin{align*}
    N^{-1}\left(\E[K'(\X)]+2m_*^{-1}\E[K'(\X)(m_N-m_*)]\right)\hspace{1in}&\\
    =2\E[\X_3]-2m_*^{-1}(z-E_*)\E[\X_2]-(A_4-2m_*^{-1}-m_*^{-4})
    \E[\X_4]&+\O(\Psi^5),\\
N^{-1}\E[K'(\X)(m_N-m_*)]-2\E[\X_{4,22'}+\X_{4,13}+\X_{4,4}+\X_{4,12\t{2}}]
    &=\O(\Psi^5),\\
    \E[2\X_{4,13}+3\X_{4,4}+\X_{4,4'}+2\X_{4,3\t{2}}]
    &=\O(\Psi^5),\\
    (z-E_*)\E[\X_2]-\E[\X_{4,22'}+4\X_{4,4}+\X_{4,4'}+2\X_{4,3'\t{2}}]
    &=\O(\Psi^5),\\
    \E[\X_{4,12\t{2}}+2\X_{4,3\t{2}}+\X_{4,3'\t{2}}+
\X_{4,2\t{12}}+2\X_{4,2\t{3}}+\X_{4,2\t{3}'}+2\X_{4,2\t{2}\t{2}}]
    &=\O(\Psi^5),
\end{align*}
where $\X_2=K'(\X)N^{-1}\sum_k G_{ik}^2$.
For $K' \equiv 1$, the first four identities above
reduce to \cite[eqs.\ (B.29), (B.33), (B.38),
(B.51)]{leeschnelli}. The fifth identity is
trivial for $K' \equiv 1$, as the left side is 0. It is analogous to
\cite[Eq.\ (C.42)]{leeschnelliwigner} in the full computation for the
deformed Wigner model, and may be derived as an ``optical theorem'' from
$\X_{3,2\t{2}}$.

Lemma \ref{lemma:optical} follows from substituting the second and fourth
identities into the first, adding $4m_*^{-1}$
times the third and fifth,
and taking the imaginary part (noting $K'$ is real-valued). This
concludes the proof of Theorem \ref{thm:resolventcompare}, and hence of
Theorem \ref{thm:TW}.

\appendix

\section{Properties of $\mu_0$}\label{appendix:deterministic}

\begin{figure}
\includegraphics[width=0.5\textwidth]{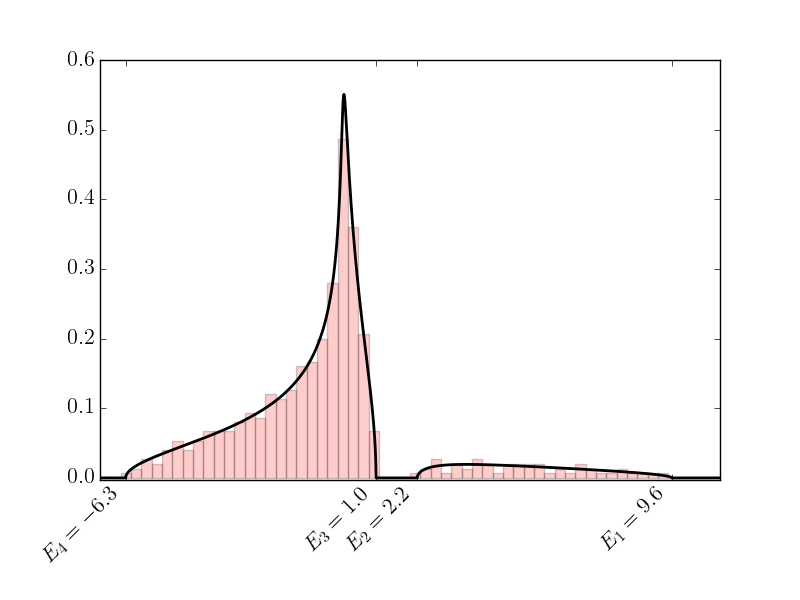}%
\includegraphics[width=0.5\textwidth]{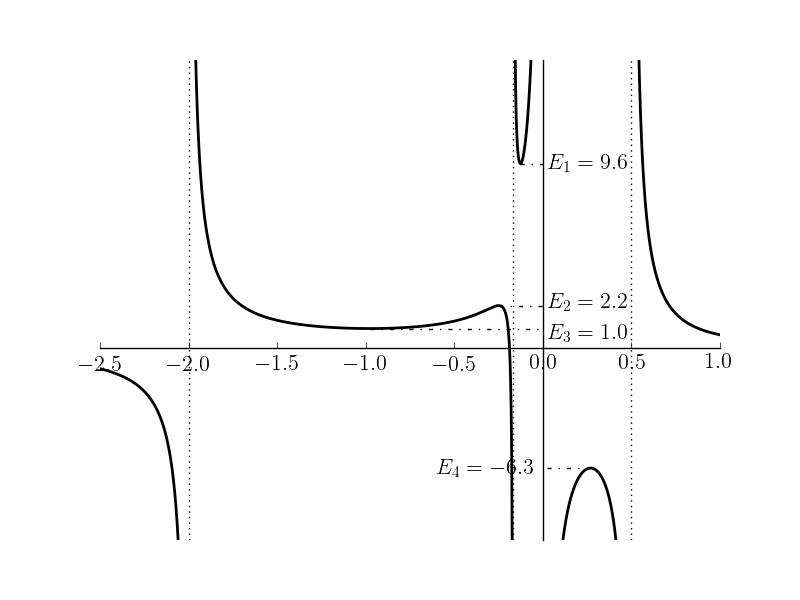}
\caption{Left: Density $f_0(x)$ of $\mu_0$ and simulated
eigenvalues of $\hSigma$, for $N=500$, $M=700$, and $T$ having 350
eigenvalues at -2, 300 at 0.5, and 50 at 6. The four soft
edges of $\mu_0$ are indicated by $E_1,\ldots,E_4$.
Right: The function $z_0(m)$, with two local
minima and two local maxima corresponding to the four edges of $\mu_0$.}
\label{fig:example1}
\end{figure}
\begin{figure}
\includegraphics[width=0.5\textwidth]{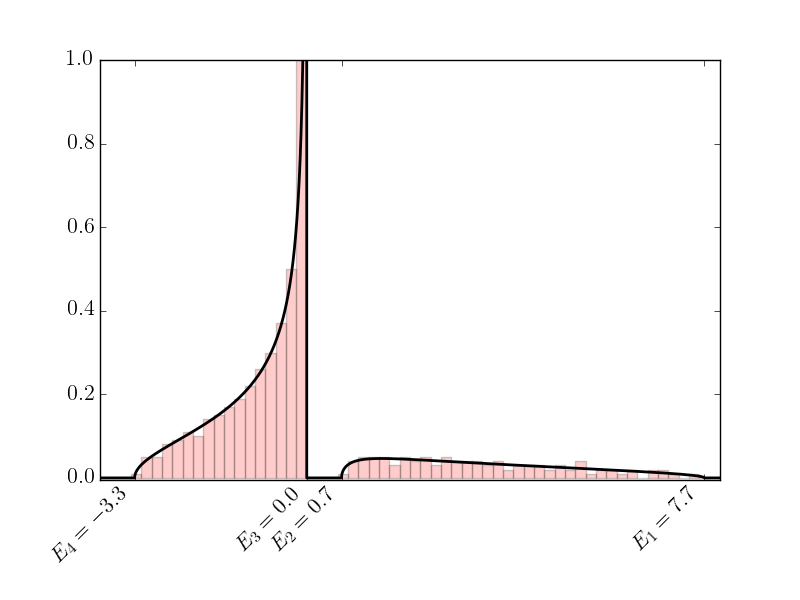}%
\includegraphics[width=0.5\textwidth]{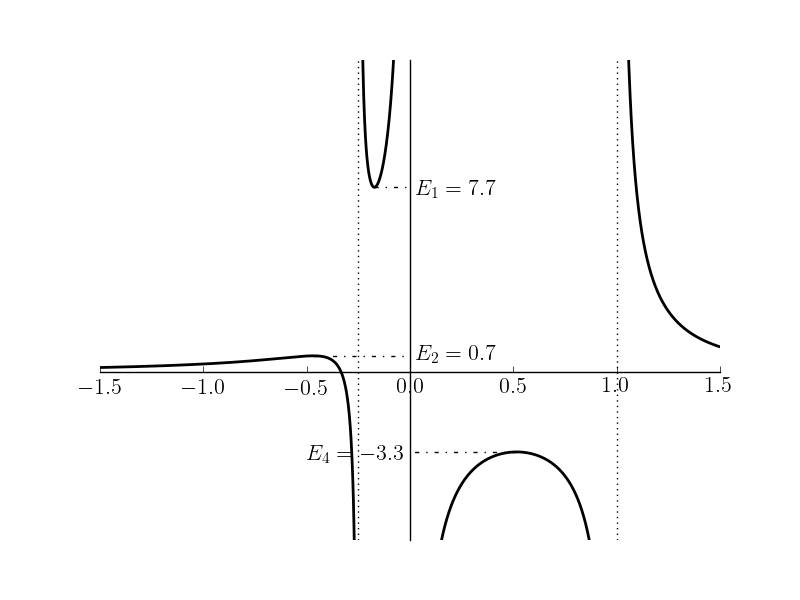}
\caption{Left: Density $f_0(x)$ of $\mu_0$ and simulated
eigenvalues of $\hSigma$, for $N=M=500$, and $T$ having 400
eigenvalues at -1 and 100 at 4. Here, $\mu_0$ has three soft edges
$E_1,E_2,E_4$ and one hard edge $E_3=0$.
Right: The function $z_0(m)$, with three indicated
local extrema, and also a local minimum at $m=\infty$
corresponding to the hard right edge $E_3=0$.}
\label{fig:example2}
\end{figure}

We verify the statements of Section \ref{subsec:support} and prove Proposition
\ref{prop:edges}.
The following characterization of the density and support of $\mu_0$
are from \cite{silversteinchoi}:

\begin{proposition}\label{prop:m0extension}
The limit
\begin{equation}\label{eq:m0x}
m_0(x)=\lim_{\eta \downarrow 0} m_0(x+i\eta)
\end{equation}
exists for each $x \in \R \setminus \{0\}$. At each such $x$, the law
$\mu_0$ admits a continuous density given by
\[f_0(x)=\frac{1}{\pi} \Im m_0(x).\]
\end{proposition}
\begin{proof}
See \cite[Theorem 1.1]{silversteinchoi}.
\end{proof}

\begin{proposition}\label{prop:mu0support}
Let
$S=\{m \in \R \setminus P:z_0'(m)>0\}$ and $z_0(S)=\{z_0(m):m\in S\}$. Then
\[\R \setminus \supp(\mu_0)=z_0(S).\]
Furthermore, $z_0:S \to \R \setminus \supp(\mu_0)$ is a bijection
with inverse $m_0:\R \setminus \supp(\mu_0) \to S$.
\end{proposition}
\begin{proof}
See \cite[Theorems 4.1 and 4.2]{silversteinchoi}.
\end{proof}

Proposition \ref{prop:mu0support} implies that $\mu_0$ has bounded support:

\begin{proposition}\label{prop:boundedsupport}
Under Assumption \ref{assump:dT}, $\supp(\mu_0) \subset [-C,C]$
for a constant $C>0$.
\end{proposition}
\begin{proof}
Proposition \ref{prop:mu0support} and the behavior of $z_0(m)$ as $m \to 0$
implies that $\mu_0$ has compact support for each $N$.
Furthermore, each non-zero boundary point of $\supp(\mu_0)$ is given by
$z_0(m_*)$ for
some $m_* \in \R$ satisfying $z_0'(m_*)=0$. Rearranging this condition yields
\[1=\frac{1}{N}\sum_{\a=1}^M \frac{m_*^2t_\a^2}{(1+m_*t_\a)^2}.\]
Since $\|T\|<C$, this condition implies $|m_*|>c$ for a constant $c>0$.
Furthermore, Cauchy-Schwarz yields
\[\left(\frac{1}{M}\sum_{\a=1}^M \frac{t_\a}{1+m_*t_\a}\right)^2
\leq \frac{1}{M}\sum_{\a=1}^M \frac{t_\a^2}{(1+m_*t_\a)^2}=\frac{N}{Mm_*^2}.\]
Combining these yields $|z_0(m_*)|<C$ for a constant $C>0$,
so each non-zero boundary point of $\supp(\mu_0)$ belongs to $[-C,C]$.
\end{proof}

We next extend Proposition \ref{prop:m0extension} to handle the case
$x=0$ (cf.\ Proposition \ref{prop:m0extension0} below). We provide this
extension so as to distinguish the behavior of a hard edge at $x=0$ from a soft
edge at $x=0$ (which may occur if $T$ is indefinite).

\begin{lemma}\label{lemma:z0decreasing}
Denote $m_0(\C^+)=\{m_0(z):z \in \C^+\}$. For any $m \in \R \setminus P$ such 
that $z_0'(m)<0$, $m$ cannot belong to the closure of $m_0(\C^+)$.
\end{lemma}
\begin{proof}
$z_0$ defines an analytic function on $\C \setminus P$.
For any such $m$, the inverse function theorem implies $z_0$ has an
analytic inverse in a neighborhood $B$ of $m$ in $\C \setminus P$.
If $m$ belongs to the closure of $m_0(\C^+)$, then $B \cap m_0(\C^+)$ is
non-empty. As $z_0(m_0(z))=z$ for $z \in \C^+$ by definition of
$m_0$, the inverse of $z_0$ on $B$ is an analytic extension of $m_0$ to
$z_0(B)$. By the open mapping theorem, $z_0(B)$ is an open set in $\C$
containing $m$. On the other hand, as $m_0$ is the Stieltjes transform of
$\mu_0$, it permits an analytic extension only to
$\C \setminus \supp(\mu_0)$, and this extension is real-valued and increasing
on $\R \setminus \supp(\mu_0)$. Then $z_0(B) \cap \R$ must belong to
$\R \setminus \supp(\mu_0)$ and $z_0$ must be increasing on $B \cap \R$, but
this contradicts that $z_0'(m)<0$.
\end{proof}

\begin{lemma}\label{lemma:disjointincreasing}
Define
\begin{equation}\label{eq:gq}
g(q)=z_0(1/q)
=-q+\frac{1}{N}\sum_{\a=1}^M \left(t_\a-\frac{t_\a^2}{q+t_\a}\right).
\end{equation}
Then for any $c \in \R$, there is at most one value $q \in \R$
for which $g(q)=c$ and $g'(q) \leq 0$.
\end{lemma}
\begin{proof}
Denote by $P'=\{-t_\a:t_\a \neq 0\}$ the distinct poles of $g$, and let
$I_1,\ldots,I_{|P'|+1}$ be the intervals of $\R \setminus P'$ in increasing
order. For any $c \in \R$, boundary conditions of $g$ at $P'$ imply that
$g(q)=c$ has at least one root $q$ in each interval $I_2,\ldots,I_{|P'|}$,
and hence at least $|P'|-1$ total roots. In addition, every
$q \in \R$ where $g(q)=c$ and $g'(q) \leq 0$ contributes two additional
roots to $g(q)=c$, counting multiplicity.
As $g(q)=c$ may be written as a polynomial equation in $q$ of
degree $|P'|+1$ by clearing denominators, it can have at most $|P'|+1$ total
roots counting multiplicity, and hence there is at most one such $q$.
\end{proof}

\begin{proposition}\label{prop:m0extension0}
If $\rank(T)>N$, then the limit (\ref{eq:m0x}) exists also at $x=0$,
and $\mu_0$ has continuous density $f_0(x)=(1/\pi)\Im m_0(x)$ at $x=0$.

If $\rank(T) \leq N$, then for any
sequence $z_n \to 0$ with $z_n \in \overline{\C^+} \setminus \{0\}$, we have
$|m_0(z_n)| \to \infty$.
\end{proposition}
\begin{proof}
Suppose $\rank(T)>N$. Taking imaginary parts of (\ref{eq:MPdiag}) yields
\begin{equation}\label{eq:MPimag}
\Im z=\frac{\Im m_0(z)}{|m_0(z)|^2}
\left(1-\frac{1}{N}\sum_{\a=1}^M \frac{|t_\a m_0(z)|^2}{|1+t_\a m_0(z)|^2}
\right).
\end{equation}
Both $\Im z>0$ and $\Im m_0(z)>0$ for $z \in \C^+$,
whereas if $|m_0(z_n)| \to \infty$ along any sequence $z_n \in \C^+$, then
\[\left(1-\frac{1}{N}\sum_{\a=1}^M \frac{|t_\a m_0(z_n)|^2}{|1+t_\a m_0(z_n)|^2}
\right) \to 1-\frac{\rank(T)}{N}.\]
When $\rank(T)>N$, this implies $m_0(z)$ is bounded on all of $\C^+$.
In particular, it is bounded in a neighborhood of $x=0$, and the result follows
from the same proof as \cite[Theorem 1.1]{silversteinchoi}.

Suppose now $\rank(T) \leq N$. Note (\ref{eq:MPdiag}) holds for $z \in
\overline{\C^+} \setminus \{0\}$ by continuity of $m_0$. If
$m_0(z_n) \to m$ for some finite $m$ along any sequence $z_n \in \overline{\C^+}
\setminus \{0\}$ with $z_n \to 0$, then
$z_0(m)=\lim_n z_0(m_0(z_n))=0$, and $m \notin P$. 
Rearranging (\ref{eq:MPdiag}) yields
\[zm_0(z)=-1+\frac{\rank(T)}{N}-\frac{1}{N}\sum_{\a:t_\a \neq 0} \frac{1}{1+t_\a
m_0(z)},\]
and taking real and imaginary parts followed by $z_n \to 0$ yields
\[1-\frac{\rank(T)}{N}=-\frac{1}{N}\sum_{\a:t_\a \neq 0}
\frac{1+t_\a \Re m}{|1+t_\a m|^2},\qquad
0=\frac{1}{N} \sum_{\a:t_\a \neq 0} \frac{t_\a \Im m}{|1+t_\a m|^2}.\]
When $\rank(T) \leq N$, the first equation implies
$\Re m \neq 0$ and $\sum_{\a:t_\a \neq 0}
t_\a/|1+t_\a m|^2 \neq 0$, and the second
equation then implies $\Im m=0$. Thus $m \in \R \setminus P$. But 
recalling $g(q)$ from (\ref{eq:gq}), we have $g(0)=0$ and $g'(0) \leq 0$ when
$\rank(T) \leq N$, so Lemma \ref{lemma:disjointincreasing} implies $g'(q)>0$ for
every other $q$ where $g(q)=0$. Thus $z_0'(m)<0$, but this contradicts 
Lemma \ref{lemma:z0decreasing}. Hence $|m_0(z_n)| \to \infty$.
\end{proof}

Recall $\R_*$ from (\ref{eq:Rstar}) and the notion of a soft edge from
Definition \ref{def:mvaluescale}. We record the following consequence of the
above.
\begin{proposition}\label{prop:m0softedge}
If $E_*$ is a soft edge of $\mu_0$ with $m$-value $m_*$, then $E_* \in \R_*$,
$m_0$ extends continuously to $E_*$, and $m_0(E_*)=m_*$.
\end{proposition}
\begin{proof}
Recalling $g(q)$ from (\ref{eq:gq}),
if $E_*=0$ is a soft edge, then $g(1/m_*)=0$ and $g'(1/m_*)=0$.
Hence Lemma \ref{lemma:disjointincreasing} implies $g'(0)>0$, so
$\rank(T)>N$. Thus any soft edge $E_*$ belongs to $\R_*$. Propositions
\ref{prop:m0extension} and \ref{prop:m0extension0} then imply continuous
extension of $m_0$ to $E_*$.
Considering $m \in \R$ with $z_0'(m)>0$ and $m \to m_*$, Proposition
\ref{prop:mu0support} implies $m_0(z_0(m))=m$, while continuity of $z_0$ and
$m_0$ yield $z_0(m) \to z_0(m_*)=E_*$ and $m_0(z_0(m)) \to m_0(E_*)$. Hence
$m_0(E_*)=m_*$.
\end{proof}

We now establish the characterization of edges of $\mu_0$ given in
Proposition \ref{prop:edges}, following arguments similar to
\cite{silversteinchoi,knowlesyin}.

\begin{proof}[Proof of Proposition \ref{prop:edges}]
Let $g(q)$ be as in Lemma \ref{lemma:disjointincreasing}.
If $m_j$ is a local minimum (or maximum) of $z_0$, then $q_j=1/m_j$ is a local
minimum (resp.\ maximum) of $g$, where $q_j=0$ if $m_j=\infty$. Furthermore
these are the only local extrema of $g$, and they are ordered as
$q_1<\ldots<q_n$. We have $E_j=g(q_j)$ for each $j=1,\ldots,n$.

Let $P'=\{-t_\a:t_\a \neq 0\}$ be the poles of $g$, and let
$I_1,\ldots,I_{|P'|+1}$ be the intervals of $\R \setminus P'$ in increasing
order. Denoting
\[S'=\{q \in \R \setminus P':g'(q)<0\},\]
Proposition \ref{prop:mu0support} is rephrased in terms of $g$ as
\begin{equation}\label{eq:supportcondition}
\R \setminus \supp(\mu_0)=g(S' \setminus \{0\}).
\end{equation}
(We must remove 0 from $S'$, as $m=\infty$ is not included in $S$.)
As $g'''(q)>0$ for all $q
\in \R \setminus P'$, we have that $g'(q)$ is convex on each $I_j$. Together
with the boundary conditions $g'(q) \to \infty$ as $q \to P'$ and
$g'(q) \to -1$ as $q \to \pm \infty$, this implies $I_1$ contains the single
local extremum $q_1$ (a minimum), $I_{|P'|+1}$ contains the single local
extremum $q_n$ (a maximum),
and each $I_j$ for $j=2,\ldots,|P'|$ contains either 0 or 2 local
extrema (a maximum followed by a minimum). Hence $S'$
is a union of open intervals, say $J_1,\ldots,J_r$,
with at most one such interval contained in each $I_j$.
Lemma \ref{lemma:disjointincreasing} verifies
\begin{equation}\label{eq:disjointgJ}
\overline{g(J_j)} \cap \overline{g(J_k)}=\emptyset
\end{equation}
for all $j \neq k$. Together with (\ref{eq:supportcondition}), this verifies
that the edges of $\mu_0$ are precisely the values $g(q_j)$, with a local
maximum $q_j$ corresponding to a left edge and a local minimum $q_j$
corresponding to a right edge. If $0 \in S'$, then it belongs to the interior
of some open interval $J_j$, and $\supp(\mu_0)$ contains an isolated point
at 0 which is not considered an edge. This establishes (a) and (b).

The ordering in part (c) follows from a continuity argument as in
\cite[Lemma 2.5]{knowlesyin}: Define for $\lambda \in (0,1]$
\[g_\lambda(q)=-q+\frac{\lambda}{N}\sum_{\a=1}^M
\left(t_\a-\frac{t_\a^2}{q+t_\a}\right).\]
Note that $g_\lambda'(q)$ is increasing in $\lambda$ for each fixed
$q \in \R \setminus P'$. Hence for each local minimum (or maximum) $q_j$ of $g$,
we may define a path $q_j(\lambda)$, continuous and increasing (resp.\
decreasing) in $\lambda$, such that
$q_j(1)=q_j$ and $q_j(\lambda)$ remains a local minimum (resp.\ maximum) of 
$g_\lambda$ for each $\lambda \in (0,1]$. As $\lambda \searrow 0$,
each $q_j(\lambda)$ converges to a pole $-t_\a$ in $P'$, with
$g_\lambda(q_j(\lambda)) \searrow
t_\a$ if $q_j(\lambda) \nearrow -t_\a$ and 
$g_\lambda(q_j(\lambda)) \nearrow t_\a$ if $q_j(\lambda) \searrow
-t_\a$. Hence for sufficiently small $\lambda>0$,
\[g_\lambda(q_1(\lambda))>\ldots>g_\lambda(q_n(\lambda)).\]
Lemma \ref{lemma:disjointincreasing} applies to $g_\lambda$ for each fixed 
$\lambda$, implying in particular that $g_\lambda(q_j(\lambda)) \neq
g_\lambda(q_k(\lambda))$ for any $j \neq k$. Hence by continuity in $\lambda$,
the above ordering is preserved for all
$\lambda \in (0,1]$. In particular it holds at $\lambda=1$, which
establishes (c).

Finally, for part (d), suppose $E_j$ is a soft right edge.
Proposition \ref{prop:m0softedge} yields $m_j \in \R_*$
and $m_0(E_j)=m_j$. The previous
convexity argument implies $g''(q_j) \neq 0$ for any local extremum $q_j$,
and hence $z_0''(m_j) \neq 0$.
Taking $x \nearrow E_j$, continuity of $m_0$ implies
$m_0(x) \to m_j$. As $z_0$ is analytic at $m_j$ and $z_0'(m_j)=0$,
a Taylor expansion yields, as $x \nearrow E_j$,
\[x-E_j=z_0(m_0(x))-z_0(m_j)=\frac{z_0''(m_j)}{2}(1+o(1))(m_0(x)-m_j)^2.\]
Since $\Im m_0(x)>0$ and $\Im m_j=0$, this yields
\[m_0(x)-m_j=\sqrt{\frac{2}{z_0''(m_j)}(x-E_j)(1+o(1))},\]
where we take the square root with branch cut on the positive real axis and
having positive imaginary part. Taking imaginary parts
and recalling $f_0(x)=(1/\pi)\Im m_0(x)$ yields (d). The case of a left edge is
similar.
\end{proof}

\section{Behavior of Stieltjes transform}\label{appendix:m0}

We establish some estimates involving the Stieltjes transform $m_0(z)$ in
spectral domains with constant separation from $\supp(\mu_0)$. We then prove the
consequences of edge regularity stated in Section \ref{subsec:regularity}.
Many arguments are similar to those of \cite[Appendix A]{knowlesyin}, although there are differences in the technical details to handle indefinite $T$.

First consider $z \in U_\delta=\{z \in \C:\dist(z,\supp(\mu_0)) \geq \delta\}$
for a constant $\delta>0$. We establish some basic bounds on $m_0$
and $\Im m_0$ in this domain.

\begin{proposition}\label{prop:m0regularoutside}
Suppose Assumption \ref{assump:dT} holds. Fix any constant $\delta>0$. Then
for some constant $c>0$, all $z \in U_\delta$, and each eigenvalue
$t_\a$ of $T$,
\[|1+t_\a m_0(z)|>c.\]
\end{proposition}
\begin{proof}
For each $z \in U_\delta$, we have
\begin{equation}\label{eq:m0stieltjes}
\Im m_0(z)=\int \frac{\Im z}{|x-z|^2}\mu_0(dx),
\qquad |m_0(z)| \leq \int \frac{1}{|x-z|}\mu_0(dx) \leq \frac{1}{\delta}.
\end{equation}
The second statement implies the
result holds for $|t_\a|<\delta/2$.
Since $\|T\|<C_0$ for a constant $C_0>0$, the result also holds when
$|m_0(z)|<1/(2C_0)$. Proposition \ref{prop:boundedsupport} shows that
$\supp(\mu_0)$ is uniformly bounded, so there is a constant $R>0$ such that
$|m_0(z)|<1/(2C_0)$ when $|z|>R$. Thus it remains to consider
the case
\begin{equation}\label{eq:zconditions}
|t_\a| \geq \delta/2, \qquad |m_0(z)| \geq 1/(2C_0), \qquad |z| \leq R.
\end{equation}
For this case, consider first $z \in U_\delta \cap \R$, so that $m_0(z) \in \R$.
The result is immediate if $t_\a m_0(z)>0$. Otherwise, note that
$\operatorname{sign}(m_0(z))=\operatorname{sign}(-1/t_\a)$.
Since $z \notin \supp(\mu_0)$, Proposition \ref{prop:mu0support}
implies $z_0'(m_0(z))>0$. By the behavior of $z_0$ at its poles,
there exists $m_* \in \R$ between $m_0(z)$ and $-1/t_\a$
such that $z_0'(m_*)=0$ and $z_0'(m)>0$ for each $m$ between $m_*$ and
$m_0(z)$. Note that $|1/t_\a|>1/C_0$, so
$|m|>1/(2C_0)$ for each such $m$. Also, differentiating (\ref{eq:z0}) yields
$z_0'(m) \leq 1/m^2$. So $0<z_0'(m)<4C_0^2$ for each such $m$. Then,
since $z=z_0(m_0(z))$, we have
\[|m_0(z)+1/t_\a|>|m_0(z)-m_*|>|z-z_0(m_*)|/(4C_0^2).\]
Since $z_0(m_*)$ is a boundary of $\supp(\mu_0)$ and $z \in U_\delta$, we have
$|z-z_0(m_*)|>\delta$. Multiplying by $|t_\a|$ and applying $|t_\a|
\geq \delta/2$ yields the result when $z \in U_\delta \cap \R$.

To extend to all $z \in U_\delta$ satisfying (\ref{eq:zconditions}), let
us apply the validity of this result for $z \in U_{\delta/2} \cap \R$. Note
that for any $z,z' \in U_{\delta/2}$, we have
\[|m_0(z)-m_0(z')| \leq \int
\left|\frac{1}{x-z}-\frac{1}{x-z'}\right|\mu_0(dx) \leq C|z-z'|.\]
Thus $|1+t_\a m_0(z)|>c$ for all $z \in U_\delta \subset U_{\delta/2}$
belonging to an $\eps$-neighborhood of $U_{\delta/2} \cap \R$,
for a sufficiently small constant $\eps>0$. On the other hand,
if $\dist(z,U_{\delta/2} \cap \R)>\eps$ and $z \in U_\delta$, then it is easy
to check that $|\Im z|>\eps$ when $\eps$ is sufficiently small. So the
bound $|z|<R$ in (\ref{eq:zconditions}) and the first statement of
(\ref{eq:m0stieltjes}) yields $|\Im m_0(z)|>c$.
Then $|1+t_\a m_0(z)| \geq |t_\a| \cdot |\Im m_0(z)|>c$.
\end{proof}

\begin{proposition}\label{prop:m0basicoutside}
Suppose Assumption \ref{assump:dT} holds.
Fix $\delta,R>0$. Then there exist constants $C,c>0$ such that for all $z \in
U_\delta$,
\[|m_0(z)|<C, \qquad |\Im m_0(z)| \leq C|\Im z|,\]
and for all $z \in U_\delta$ with $|z|<R$,
\[|m_0(z)|>c, \qquad |\Im m_0(z)| \geq c|\Im z|.\]
\end{proposition}
\begin{proof}
From (\ref{eq:m0stieltjes}),
we obtain both bounds on $\Im m_0(z)$ and the upper bound on $|m_0(z)|$.
The lower bound on $|m_0(z)|$
follows from (\ref{eq:MPdiag}) together with $|z|<R$, $|t_\a|<C$,
and $|1+t_\a m_0(z)|>c$.
\end{proof}

We now turn to the implications of edge regularity, and prove
Propositions \ref{prop:regular}, \ref{prop:basicregbounds},
\ref{prop:z0secondderivative}, and
\ref{prop:m0estimates}. 

\begin{remark}
One may check, via Proposition \ref{prop:basicregbounds}, that
Definition \ref{def:regular} is equivalent to the definition of a regular edge
in \cite{knowlesyin} when $T$ is positive definite.
The condition $|m_*+t_\a^{-1}|>\tau$ is similar to that
introduced for the rightmost edge in \cite{elkaroui}. 
In a simple spiked model \cite{johnstone} where 
$(t_1,\ldots,t_M)=(\theta,1,1,\ldots,1)$ for
fixed $\theta>1$, this condition for the rightmost edge is equivalent to
$\theta$ falling below the phase transition
threshold $1+\sqrt{M/N}$ studied in \cite{baiketal}.
\end{remark}

\begin{proof}[Proof of Proposition \ref{prop:basicregbounds}]
The bounds $|m_*|<\tau^{-1}$ and $\gamma<\tau^{-1}$
are assumed in Definition \ref{def:regular}. From (\ref{eq:dz0mstar})
and the condition $|m_*+t_\a^{-1}|>\tau$ for each $\a$, the bound
$|m_*|>c$ follows. The bounds
$|E_*|<C$ and $\gamma>c$ then follow from the definitions $E_*=z_0(m_*)$ and
$\gamma^{-2}=|z_0''(m_*)|/2$.
For $|1+t_\a m_*|$, take $C>0$ such that $|m_*|<C$.
If $|t_\a|>1/(2C)$, then $|1+t_\a m_*|>\tau/(2C)$ by the condition
$|m_*+t_\a^{-1}|>\tau$, whereas if $|t_\a| \leq 1/(2C)$,
then $|1+t_\a m_*|>1/2$.

From (\ref{eq:dz0mstar}) and the conditions $|m_*|<C$ and $|1+t_\a m_*|>c$,
we have $M^{-1}\sum_\a t_\a^2>c$. Together with the assumption $|t_\a|<C$ for
all $\a$, this implies (\ref{eq:nondegenerate}).
Finally, note that $0=z_0'(m_*)$ implies
$m_*^{-1}=N^{-1}\sum_\a t_\a^2m_*/(1+t_\a m_*)^2$, and hence
\[E_*=z_0(m_*)=\frac{1}{N}\sum_{\a=1}^M \frac{t_\a}{(1+t_\a m_*)^2}.\]
If $T$ is positive semi-definite, then $E_*>c$ follows from
$|1+t_\a m_*|<C$ and (\ref{eq:nondegenerate}).
\end{proof}

The remaining results heuristically follow from
the Taylor expansion
\[z_0(m)-E_*=z_0(m)-z_0(m_*)=\frac{z_0''(m_*)}{2}(m-m_*)^2+O((m-m_*)^3),\]
where there is no first-order term because $0=z_0'(m_*)$. Consequently,
\[m_0(z) \approx m_*+\sqrt{\frac{2}{z_0''(m_*)}(z-E_*)}\]
for $z \in \C^+$ near $E_*$ and an appropriate choice of square-root.
Edge regularity implies uniform control of the above Taylor expansion.

We first quantify continuity of $m_0$, uniformly in $N$,
near a regular edge $E_*$. In particular this implies that when $|z-E_*|$ is
small, $|m_0(z)-m_*|$ is also small. (We believe that uniform control of this
continuity may have been
omitted from the analysis in \cite[Appendix A]{knowlesyin}.)

\begin{lemma}\label{lemma:m0uniformcontinuity}
Suppose Assumption \ref{assump:dT} holds and $E_*$ is a
regular edge with $m$-value $m_*$.
Then there exist constants $C,\delta>0$ such that
\[(E_*-\delta,E_*+\delta) \subset \R_*,\]
and for every $z \in \overline{\C^+}$ with $|z-E_*|<\delta$,
\[|m_0(z)-m_*|^2<C|z-E_*|.\]
\end{lemma}
\begin{proof}
Applying Proposition \ref{prop:basicregbounds}, take a constant
$\nu>0$ such that $|m_*|>\nu$.
Fix a constant $c<\min(\nu,\tau)$ to be determined later,
and define
\begin{align*}
\delta_N&=\min\Big(c,\;\inf(\delta>0:|m_0(z)-m_*|<c \text{ for all }
z \in \C^+ \cup \R_*
 \text{ such that } |z-E_*| \leq \delta)\Big).
\end{align*}
As $m_0(E_*)=m_*$, continuity of $m_0$ at $E_*$ implies $\delta_N>0$.
Furthermore, if $\rank(T) \leq N$ so that $0 \notin \R_*$, then the divergence
of $m_0$ at 0 from Proposition \ref{prop:m0extension0} implies
$(E_*-\delta_N,E_*+\delta_N) \subset \R_*$.
A priori, $\delta_N$ may depend on $N$. We will first establish 
that $|m_0(z)-m_*|^2<C|z-E_*|$ when $|z-E_*| \leq \delta_N$. This will then
imply that $\delta_N$ is bounded below by a constant $\delta$.

Consider $z \in \overline{\C^+}$ with $|z-E_*| \leq \delta_N$. Let us
write as shorthand $m=m_0(z)$. Then
\begin{align}
|z-E_*|&=|z_0(m)-z_0(m_*)|\nonumber\\
&=|m-m_*| \left|-\frac{1}{mm_*}+\frac{1}{N}
\sum_{\a=1}^M \frac{t_\a^2}{(1+t_\a m)(1+t_\a m_*)}\right|\nonumber\\
&=|m-m_*|^2 \left|-\frac{1}{mm_*^2}+\frac{1}{N}
\sum_{\a=1}^M \frac{t_\a^3}{(1+t_\a m)(1+t_\a m_*)^2}\right|,
\label{eq:deltastar}
\end{align}
where the last line adds to the quantity inside the modulus
\[0=z_0'(m_*)=\frac{1}{m_*^2}-\frac{1}{N}\sum_{\a=1}^M
\frac{t_\a^2}{(1+t_\a m_*)^2}.\]
As $|m-m_*|<c$ by definition of $\delta_N$, we have for each non-zero
$t_\a$
\[\left|\frac{1}{m}-\frac{1}{m_*}\right|<\frac{c}{\nu(\nu-c)},
\qquad \left|\frac{1}{m+t_\a^{-1}}
-\frac{1}{m_*+t_\a^{-1}}\right|<\frac{c}{\tau(\tau-c)}.\]
Applying this to (\ref{eq:deltastar}) and recalling
$\gamma^{-2}=|z_0''(m_*)|/2$ yields
\[|z-E_*|>|m-m_*|^2\left(\gamma^{-2}-
\frac{c}{\nu^3(\nu-c)}-\frac{M}{N}\frac{c}{\tau^3(\tau-c)}\right).\]
As $\gamma^{-2}>\tau^2$, this implies
$|m_0(z)-m_*|^2<C|z-E_*|$ when $c$ is chosen sufficiently small, as desired.

By continuity of $m_0$ and definition of $\delta_N$, either $\delta_N=c$ or
there must exist $z \in \overline{\C^+}$ such that $|z-E_*|=\delta_N$ and
$|m_0(z)-m_*|=c$. In the latter case, for this $z$
we have $c^2=|m_0(z)-m_*|^2<C|z-E_*|=C\delta_N$, implying $\delta_N>c^2/C$.
Thus in both cases $\delta_N$ is bounded below by a constant,
yielding the lemma.
\end{proof}

Next we bound the third derivative of $z_0$ near the $m$-value of a
regular edge.
\begin{lemma}\label{lemma:z0thirdderivative}
Suppose Assumption \ref{assump:dT} holds and $E_*$ is a
regular edge with $m$-value $m_*$. Then there exist
constants $C,\delta>0$ such that $z_0$ is analytic on the
disk $\{m \in \C:|m-m_*|<\delta\}$, and for every $m$ in this disk,
\[|z_0'''(m)|<C.\]
\end{lemma}
\begin{proof}
Proposition \ref{prop:basicregbounds} ensures $|m_*|>\nu$ for a constant
$\nu>0$. Taking $\delta<\min(\nu,\tau)$, the disk
$D=\{m \in \C:|m-m_*|<\delta\}$ does not contain any pole of $z_0$,
and hence $z_0$ is analytic on $D$. We compute
\[z_0'''(m)=\frac{6}{m^4}-\frac{1}{N}\sum_{\a:t_\a \neq 0}
\frac{6}{(t_\a^{-1}+m)^4},\]
so $|z_0'''(m)|<C$ for $m \in D$ and sufficiently small $\delta$
by the bounds $|m_*|>\nu$ and $|m_*+t_\a^{-1}|>\tau$.
\end{proof}

Propositions \ref{prop:regular}, \ref{prop:z0secondderivative}, and
\ref{prop:m0estimates} now follow:

\begin{proof}[Proof of Proposition \ref{prop:z0secondderivative}]
This follows from Taylor expansion of $z_0''$ at $m_*$,
the condition $|z_0''(m_*)|=2\gamma^{-2}>2\tau^2$ implied by regularity,
and Lemma \ref{lemma:z0thirdderivative}.
\end{proof}

\begin{proof}[Proof of Proposition \ref{prop:regular}(a)]
Let $C,\delta>0$ be as in Lemma \ref{lemma:m0uniformcontinuity}.
Reducing $\delta$ as necessary and
applying Lemma \ref{lemma:z0thirdderivative}, we may assume $z_0$ is analytic
with $|z_0'''(m)|<C'$ over the disk
\[D=\{m \in \C:|m-m_*|<\sqrt{C\delta}\},\]
for a constant $C'>0$.

Let $E^*$ be the closest other edge to $E_*$, and suppose
$E^* \in (E_*-\delta,E_*+\delta)$. Let $m^*$ be the $m$-value for $E^*$.
Then Lemma \ref{lemma:m0uniformcontinuity} implies $m^* \in D$.
Applying a Taylor expansion of $z_0'$,
\[z_0'(m^*)=z_0'(m_*)+z_0''(m_*)(m^*-m_*)+\frac{z_0'''(m)}{2}(m^*-m_*)^2\]
for some $m$ between $m_*$ and $m^*$. Applying $0=z_0'(m^*)=z_0'(m_*)$,
$|z_0''(m_*)|=2\gamma^{-2}>2\tau^2$, and $|z_0'''(m)|<C'$,
we obtain $|m^*-m_*|>4\tau^2/C'$. Then Lemma
\ref{lemma:m0uniformcontinuity} yields $|E^*-E_*|>c$ for a constant
$c>0$. Reducing $\delta$ to $c$ if necessary, we ensure
$(E_*-\delta,E_*+\delta)$ contains no other edge $E^*$. The
condition $(E_*-\delta,E_*+\delta) \subset \R_*$ was established in Lemma
\ref{lemma:m0uniformcontinuity}.
\end{proof}

\begin{proof}[Proof of Propositions \ref{prop:m0estimates}
and \ref{prop:regular}(b)]
For any constant $\delta>0$, if $\eta=\Im z \geq \delta$,
then all claims follow from
Propositions \ref{prop:m0regularoutside} and \ref{prop:m0basicoutside}. Hence
let us consider $\eta=\Im z<\delta$.

Taking $\delta$ sufficiently small, Lemma \ref{lemma:m0uniformcontinuity}
implies $|m_0(z)-m_*|<\sqrt{C\delta}$ for all $z \in \bD_0$.
Then $|m_0(z)| \asymp 1$ and $|1+t_\a m_0(z)| \asymp 1$ by
Proposition \ref{prop:basicregbounds}. Reducing $\delta$ if necessary,
by Lemma \ref{lemma:z0thirdderivative} we may also ensure
$z_0$ is analytic with $|z_0'''(m)|<C'$ on
\[D=\{m \in \C:|m-m^*|<\sqrt{C\delta}\}.\]
Note $z=z_0(m_0(z))$ by (\ref{eq:MPdiag}) while
$E_*=z_0(m_*)$. Then taking a Taylor expansion of $z_0$ and applying the
conditions $z_0'(m_*)=0$, $z_0''(m_*)=2\gamma^{-2}$,
and $|z_0'''(\tilde{m})|<C'$ for all $\tilde{m} \in D$, we have
\begin{equation}\label{eq:Taylorzm}
z-E_*=z_0(m_0(z))-z_0(m_*)=(\gamma^{-2}+r(z))(m_0(z)-m_*)^2
\end{equation}
where $|r(z)|<C'\sqrt{C\delta}/6$. Taking $\delta$ sufficiently small, we
ensure
\begin{equation}\label{eq:gammarz}
|\gamma^{-2}+r(z)| \asymp 1,\qquad\arg(\gamma^{-2}+r(z)) \in (-\eps,\eps)
\end{equation}
for an arbitrarily small constant $\eps>0$, where $\arg(z)$ denotes the complex
argument. Taking the modulus of 
(\ref{eq:Taylorzm}) on both sides yields
$|m_0(z)-m_*| \asymp \sqrt{|z-E_*|} \asymp \sqrt{\kappa+\eta}$.

For $\Im m_0(z)$, suppose $E_*$ is a right edge. (The case of a
left edge is similar.) By Proposition \ref{prop:regular}(a), we may assume
$(E_*-\delta,E_*) \subset \supp(\mu_0)$ and
$(E_*,E_*+\delta) \subset \R \setminus \supp(\mu_0)$.
First suppose $\Im z>0$ and $E \equiv \Re z \leq E_*$. As $\Im m_0(z)>0$ by
definition, (\ref{eq:Taylorzm}) yields
\[m_0(z)-m_*=\sqrt{(z-E_*)/(\gamma^{-2}+r(z))}\]
where the square-root has branch cut on the positive real axis
and positive imaginary part. Applying $\arg(z-E_*) \in [\pi/2,\pi)$ and
(\ref{eq:gammarz}), we have
$\Im m_0(z) \asymp \Im \sqrt{z-E_*} \asymp |\sqrt{z-E_*}|
\asymp \sqrt{\kappa+\eta}$. By continuity of $m_0$, this extends to $z \in
(E_*-\delta,E_*)$ on the real axis. Hence Proposition \ref{prop:regular}(b)
also follows, as $f_0(x)=\pi^{-1} \Im m_0(x)$.

Now, suppose $E \equiv \Re z>E_*$. Let us write
\begin{align*}
\Im m_0(z)&=\int_{|\lambda-E_*|<\delta} \frac{\eta}{(\lambda-E)^2+\eta^2} 
\mu_0(d\lambda)+\int_{|\lambda-E_*| \geq \delta}
\frac{\eta}{(\lambda-E)^2+\eta^2}\mu_0(d\lambda)\\
&\equiv \mathrm{I}+\mathrm{II}.
\end{align*}
Reducing $\delta$ to $\delta/2$, we may assume the closest edge to $E$ is $E_*$.
Then we have $\mathrm{II} \in [0,\eta/\delta^2]$.
For $\mathrm{I}$, as $\mu_0$ has density $f_0(x) \asymp \sqrt{E_*-x}$ for
$x \in (E_*-\delta,E_*)$ while
$(E_*,E_*+\delta) \subset \R \setminus \supp(\mu_0)$,
\[\mathrm{I} \asymp \int_{E_*-\delta}^{E_*} \frac{\eta}{(\lambda-E)^2+\eta^2}
\sqrt{E_*-\lambda}\,d\lambda
=\int_0^\delta \frac{\eta}{\eta^2+(\kappa+x)^2}\sqrt{x}\,dx.\]
Considering separately the integral over $x \in [0,\kappa+\eta]$
and $x \in [\kappa+\eta,\delta]$, we obtain $\mathrm{I} \asymp
\eta/\sqrt{\eta+\kappa}$. Then $\mathrm{II} \leq C\cdot \mathrm{I}$, and this
yields $\Im m_0(z) \asymp \eta/\sqrt{\eta+\kappa}$.
\end{proof}

\section{Proof of local law}\label{appendix:locallaw}
We verify that the proof of the entrywise local law in \cite{knowlesyin} does
not require positivity of $T$.
Indeed, Theorem \ref{thm:generallocallaw} below, which is a slightly modified
version of \cite[Theorem 3.22]{knowlesyin}, holds in our setting.
We deduce from this Theorems
\ref{thm:sticktobulk}, \ref{thm:regedgeconcentration}, and \ref{thm:locallaw}.

We use the following notion of stability, analogous to \cite[Definition
5.4]{knowlesyin} and \cite[Lemma 4.5]{bloemendaletal}.
\begin{definition}\label{def:stability}
Fix a bounded set $S \subset \R$ and a constant $a>0$, and let
\begin{equation}\label{eq:Dgeneral}
\bD=\{z \in \C^+:\Re z \in S,\,\Im z \in [N^{-1+a},1]\}.
\end{equation}
For $z=E+i\eta \in \bD$, denote
\[L(z)=\{z\} \cup \{w \in \bD:\Re w=E,\,\Im w \in [\eta,1] \cap (N^{-5}
\mathbb{N})\}.\]
For a function $g:\bD \to (0,\infty)$, the Marcenko-Pastur equation
(\ref{eq:MPdiag}) is {\bf $\pmb{g}$-stable} on $\bD$ if the following holds
for some constant $C>0$:
Let $u:\C^+ \to \C^+$ be the Stieltjes transform of any probability measure, and
let $\Delta:\bD \to (0,\infty)$ be any function satisfying
\begin{itemize}
\item (Boundedness) $\Delta(z) \in [N^{-2},(\log N)^{-1}]$ for all $z \in \bD$,
\item (Lipschitz) $|\Delta(z)-\Delta(w)| \leq N^2|z-w|$ for all $z,w \in \bD$,
\item (Monotonicity) $\eta \mapsto \Delta(E+i\eta)$ is non-increasing for
each $E \in S$ and $\eta>0$.
\end{itemize}
If $z \in \bD$ is such that $|z_0(u(w))-w| \leq \Delta(w)$
for all $w \in L(z)$, then
\begin{equation}\label{eq:stability}
|u(z)-m_0(z)| \leq \frac{C\Delta(z)}{g(z)+\sqrt{\Delta(z)}}.
\end{equation}
\end{definition}

\begin{theorem}[Abstract local law]\label{thm:generallocallaw}
Suppose Assumptions \ref{assump:dT} and \ref{assump:X} hold.
Fix a bounded set $S \subset \R$ and a constant $a>0$, and define $\bD$ by
(\ref{eq:Dgeneral}).
Suppose, for some constants $C,c>0$ and a bounded function $g:\bD \to (0,C)$,
that (\ref{eq:MPdiag}) is $g$-stable on $\bD$, and furthermore
\[c<|m_0(z)|<C, \qquad c\eta<\Im m_0(z)<Cg(z), \qquad
|1+t_\a m_0(z)|>c\]
for all $z=E+i\eta \in \bD$ and all $\a \in \I_M$. Then, letting
$m_N(z),G(z),\Pi(z)$ be as in (\ref{eq:mN}), (\ref{eq:Galt}), and
(\ref{eq:Pi0}), and denoting
\[\Psi(z)=\sqrt{\frac{\Im m_0(z)}{N\eta}}+\frac{1}{N\eta},\]
\begin{enumerate}[(a)]
\item (Entrywise law)
For all $z \in \bD$ and $A,B \in \I$,
\[\frac{G_{AB}(z)-\Pi_{AB}(z)}{t_A t_B} \prec \Psi(z).\]
\item (Averaged law) For all $z \in \bD$,
\[m_N(z)-m_0(z) \prec
\min\left(\frac{1}{N\eta},\frac{\Psi(z)^2}{g(z)}\right).\]
\end{enumerate}
\end{theorem}
\begin{proof}
The proof is the same as for \cite[Theorem 3.22]{knowlesyin},
with only cosmetic differences which we indicate here. The notational
identification with \cite{knowlesyin} is $T \leftrightarrow \Sigma$ and 
$t_\a \leftrightarrow \sigma_i$. (We continue to use
Greek indices for $\I_M$ and Roman indices for $\I_N$, although this is reversed
from the convention in \cite{knowlesyin}.)
As in \cite{knowlesyin}, we may assume $T$ is invertible. The non-invertible
case follows by continuity.

We follow \cite[Section 5]{knowlesyin}, which in turn is based on
\cite{bloemendaletal}. Define
\[Z_i=\sum_{\a,\b \in \I_M} G_{\a\b}^{(i)}X_{\a i}X_{\b i}-N^{-1}\Tr
G_M^{(i)},\]
\[Z_\a=\sum_{i,j \in \I_N}
G_{ij}^{(\a)}X_{\a i}X_{\a j}-N^{-1}\Tr G_N^{(\a)},\]
\[[Z]=\frac{1}{N}\left(\sum_{i \in I_N} Z_i
+\sum_{\a \in \I_M} \frac{t_\a^2}{(1+t_\a m_0)^2}Z_\a\right),\]
\[\Theta=N^{-1}\left|\sum_{i \in \I_N} (G-\Pi)_{ii}\right|+M^{-1}
\left|\sum_{\a \in \I_M} (G-\Pi)_{\a\a}\right|,\qquad
\Psi_\Theta=\sqrt{\frac{\Im m_0+\Theta}{N\eta}},\]
\[\Lambda_o=\max_{A \neq B \in \I} \frac{|G_{AB}|}{|t_A t_B|},\qquad
\Lambda=\max_{A,B \in \I} \frac{|(G-\Pi)_{AB}|}{|t_A t_B|},\qquad
\Xi=\{\Lambda \leq (\log N)^{-1}\}.\]
These all implicitly depend on an argument $z \in \bD$. Then
the same steps as in \cite[Section 5]{knowlesyin} yield, either for
$\eta=1$ or on the event $\Xi$, for all $z \in \bD$ and $A \in \I$,
\begin{align}
|Z_A|,\Lambda_o &\prec \Psi_\Theta,\label{eq:ZLambdao}\\
z_0(m_N(z))-z-[Z] &\prec \Psi_\Theta^2
\prec (N\eta)^{-1}.\label{eq:z0mNz}
\end{align}
(In the argument for $\eta=1$, the use of \cite[Eq.\ (4.16)]{knowlesyin}
may be replaced by \cite[Lemmas 4.8 and 4.9]{knowlesyin}. Various bounds using
$\sigma_i$, for example
\cite[Eqs.\ (5.4), (5.11)]{knowlesyin}, may be replaced by ones using
the positive quantity $|t_\a|$.) Applying (\ref{eq:ZLambdao})
and the resolvent identities
for $G_{ii}$ and $G_{\a\a}$, we may also obtain on the event $\Xi$
\begin{equation}\label{eq:LambdaThetaN}
\Theta \prec |m_N-m_0|+|[Z]|+(N\eta)^{-1},\qquad
\Lambda \prec |m_N-m_0|+\Psi_\Theta.
\end{equation}

The bound (\ref{eq:ZLambdao}) yields the initial estimate
$[Z] \prec \Psi_\Theta \prec (N\eta)^{-1/2}$ on $\Xi$.
The conditions of Definition \ref{def:stability} hold for
$\Delta=(N\eta)^{-1/2}$, so (\ref{eq:z0mNz}), the assumed stability
of (\ref{eq:MPdiag}),
and the stochastic continuity argument of \cite[Section 4.1]{bloemendaletal}
yield that $\Xi$ holds with high probability (i.e.\ $1 \prec \1\{\Xi\}$) and
$\Lambda \prec (N\eta)^{-1/4}$ on all of $\bD$. Next, applying the fluctuation 
averaging result of \cite[Lemma 5.6]{knowlesyin},
we obtain for any $c \in (0,1]$ the implications
\begin{align*}
\Theta \prec (N\eta)^{-c} &\Rightarrow 
\Psi_\Theta \prec \sqrt{\frac{\Im m_0+(N\eta)^{-c}}{N\eta}}\\
&\Rightarrow [Z] \prec \frac{\Im m_0+(N\eta)^{-c}}{N\eta}
\equiv \Delta(z).
\end{align*}
The conditions of Definition \ref{def:stability} hold for this
$\Delta(z)$, so applying (\ref{eq:z0mNz}), stability of (\ref{eq:MPdiag}),
and $1 \prec \1\{\Xi\}$, we have the implications
\begin{align}
\Theta \prec (N\eta)^{-c} &\Rightarrow 
|m_N-m_0| \prec \frac{\Delta(z)}{g(z)+\sqrt{\Delta(z)}}\nonumber\\
&\Rightarrow \Theta \prec 
\frac{\Delta(z)}{g(z)+\sqrt{\Delta(z)}}+\Delta(z)+(N\eta)^{-1}.
\label{eq:selfimproving}
\end{align}
We bound $\Delta(z) \leq C(N\eta)^{-1}$ and
\[\frac{\Delta(z)}{g(z)+\sqrt{\Delta(z)}}
\leq \frac{\Im m_0(z)}{N\eta\,g(z)}
+(N\eta)^{-(1+c)/2}<C(N\eta)^{-1}+(N\eta)^{-(1+c)/2},\]
where this applies $\Im m_0(z)<Cg(z)$. Hence
\[\Theta \prec (N\eta)^{-c} \Rightarrow \Theta \prec (N\eta)^{-(1+c)/2}.\]
Initializing to $c=1/4$ and iterating,
we obtain $\Theta \prec (N\eta)^{-1+\eps}$
for any $\eps>0$, so $|m_N-m_0| \leq \Theta \prec (N\eta)^{-1}$. Applying
(\ref{eq:selfimproving}) once more with $c=1$, we have for $c=1$ that $\Delta(z)
\leq \Psi(z)^2$ and hence also $|m_N-m_0| \prec \Psi^2/g$. This yields
both bounds in the averaged law. The entrywise law $\Lambda \prec \Psi$
follows from (\ref{eq:LambdaThetaN}).
\end{proof}

We now verify the stability condition in Definition \ref{def:stability} near a
regular edge and outside the spectrum. Define
\[\supp(\mu_0)_\delta=\{x \in \R: \text{there exists } y \in \supp(\mu_0)
\text{ such that } |x-y|<\delta\}.\]
The proofs are the same as
\cite[Lemmas A.5 and A.8]{knowlesyin}, which are based on
\cite[Lemma 4.5]{bloemendaletal}. For convenience, we reproduce the argument
here.

\begin{lemma}\label{lemma:stability}
Suppose Assumption \ref{assump:dT} holds.
\begin{enumerate}[(a)]
\item Fix any constants $\delta,a,C_0>0$, and let
\[\bD=\{z \in \C^+:\Re z \in [-C_0,C_0] \setminus \supp(\mu_0)_\delta,\,
\Im z \in [N^{-1+a},1]\}.\]
Then (\ref{eq:MPdiag}) is $g$-stable on $\bD$ for $g(z) \equiv 1$.
\item Let $E_*$ be a regular edge, and let $\bD$ be the domain
(\ref{eq:bD}), depending on constants $\delta,a>0$.
For $z=E+i\eta \in \bD$, denote $\kappa=|E-E_*|$ and let
$g(z)=\sqrt{\kappa+\eta}$. Then, for any constant $a>0$ and any constant
$\delta>0$ sufficiently small, (\ref{eq:MPdiag}) is $g$-stable on $\bD$.
\end{enumerate}
\end{lemma}
\begin{proof}
Writing $u=u(z)$, $m=m_0(z)$, and $\Delta_0=\Delta_0(z)=z_0(u(z))-z$, we have
\begin{align*}
\Delta_0=z_0(u)-z_0(m)&=\frac{m-u}{um}\left(-1+
\frac{1}{N}\sum_{\a=1}^M \frac{t_\a^2um}{(1+t_\a u)(1+t_\a m)}\right)\\
&=\alpha(z)(m-u)^2+\beta(z)(m-u)
\end{align*}
for
\[\alpha(z)=-\frac{1}{u}\cdot
\frac{1}{N}\sum_{\a=1}^M \frac{t_\a^2}{(1+t_\a u)(1+t_\a m)^2},\]
\[\beta(z)=\frac{1}{um}\left(-1+\frac{1}{N}\sum_{\a=1}^M
\frac{t_\a^2 m^2}{(1+t_\a m)^2}\right)=-\frac{m}{u}z_0'(m).\]
Viewing this a quadratic equation in $m-u$ and
denoting the two roots
\begin{equation}\label{eq:R1R2}
R_1(z),R_2(z)=
\frac{-\beta(z) \pm \sqrt{\beta(z)^2+4\alpha(z)\Delta_0(z)}}{2\alpha(z)},
\end{equation}
we obtain $m_0(z)-u(z) \in \{R_1(z),R_2(z)\}$ for each $z \in \bD$. Note that
(\ref{eq:R1R2}) implies
\begin{equation}\label{eq:R1R2diff}
|R_1(z)-R_2(z)|=\frac{\sqrt{|\beta(z)^2+4\alpha(z)\Delta_0(z)|}}{|\alpha(z)|}.
\end{equation}
Also, we have $|R_1R_2|=|\Delta_0/\alpha|$ and $|R_1+R_2|=|\beta/\alpha|$. The
first statement yields $\min(|R_1|,|R_2|) \leq
\sqrt{|\Delta_0/\alpha|}=2|\Delta_0|/\sqrt{4|\alpha \Delta_0|}$. The second
yields $\max(|R_1|,|R_2|) \geq |\beta/(2\alpha)|$, so the first
then yields $\min(|R_1|,|R_2|) \leq 2|\Delta_0|/|\beta|$. Combining these,
\begin{equation}\label{eq:minR1R2}
\min(|R_1(z)|,|R_2(z)|) \leq
\frac{4|\Delta_0(z)|}{|\beta(z)|+\sqrt{4|\alpha(z)\Delta_0(z)|}}.
\end{equation}

We first show part (a). Let $\Delta(z)$ satisfy the conditions of Definition
\ref{def:stability}. We claim that for any constant $\nu>0$,
there exist constants $C_0,c>0$ such that
\begin{enumerate}
\item If $\Im z \geq \nu$ and $|\Delta_0(z)| \leq \Delta(z)$, then
\begin{equation}\label{eq:stabilityoutside}
|m_0(z)-u(z)| \leq C_0\Delta(z).
\end{equation}
\item If $|\Delta_0(z)| \leq \Delta(z)$ and $|m_0(z)-u(z)|<(\log N)^{-1/2}$,
then
\begin{equation}\label{eq:gapoutside}
\min(|R_1(z)|,|R_2(z)|) \leq C_0\Delta(z), \qquad
|R_1(z)-R_2(z)| \geq c.
\end{equation}
\end{enumerate}
Indeed, if $\Im z \geq \nu$ and $|\Delta_0(z)| \leq \Delta(z) \leq (\log
N)^{-1}$, then $\Im z_0(u(z)) \geq \nu/2$. In particular $z_0(u(z)) \in \C^+$,
so $m_0(z_0(u(z)))=u(z)$ as this is the unique
root $m \in \C^+$ to the equation $z_0(m)=z_0(u(z))$. Applying $|m_0'(z)| \leq
1/(\Im z)^2$, we obtain
\[|m_0(z)-u(z)|=|m_0(z)-m_0(z_0(u(z)))| \leq (4/\nu^2)|\Delta_0(z)|
\leq (4/\nu^2)\Delta(z),\]
and hence (\ref{eq:stabilityoutside}) holds for $C_0=4/\nu^2$.
On the other hand, if $|m_0(z)-u(z)|<(\log N)^{-1/2}$, then Propositions
\ref{prop:m0basicoutside} and
\ref{prop:m0regularoutside} imply $|\alpha(z)|<C$ and $|\beta(z)|<C$.
Taking imaginary parts of (\ref{eq:MPdiag}) as in
(\ref{eq:MPimag}), we also have $|u(z)m(z)\beta(z)|
\geq (\Im z)|m_0(z)|^2/\Im m_0(z)>c$, so $|\beta(z)|>c$. Applying this
to (\ref{eq:R1R2diff}) and (\ref{eq:minR1R2}), and increasing $C_0$ if
necessary, we obtain (\ref{eq:gapoutside}).

A continuity argument now concludes the proof of part (a): Consider any $z \in
\bD$ with $|\Delta_0(w)| \leq \Delta(w)$ for all $w \in L(z)$. If $\Im z \geq
\nu$, the result follows from (\ref{eq:stabilityoutside}). If $\Im z<\nu$,
let $w \in L(z)$ be such that
$\Im z<\Im w\leq \Im z+N^{-5}$. Suppose inductively that we have shown
(\ref{eq:stabilityoutside}) holds at $w$.
Applying $|u'(z)| \leq 1/(\Im z)^2 \leq N^2$ for any Stieltjes transform
$u(z)$ and $z \in \bD$, we obtain
\[|m_0(z)-u(z)| \leq C_0\Delta(w)+2N^{-3}<(\log N)^{-1/2}.\]
So (\ref{eq:gapoutside}) implies
$\max(|R_1(z)|,|R_2(z)|)>c/2$. Then
\[|m_0(z)-u(z)|=\min(|R_1(z)|,|R_2(z)|),\]
so (\ref{eq:gapoutside}) also shows
that (\ref{eq:stabilityoutside}) holds at $z$. Starting the induction at
$\Im z \geq \nu$, we obtain (\ref{eq:stabilityoutside}) for all
$w \in L(z)$, and in particular at $w=z$. This establishes part (a).

For part (b), let $g(z)=\sqrt{\kappa+\eta}$.
We claim that when $\delta>0$ is sufficiently small, there exist
constants $\nu,C_0,C_1>0$ such that
\begin{enumerate}
\item If $\Im z \geq \nu$ and $|\Delta_0(z)| \leq \Delta(z)$, then
\begin{equation}\label{eq:stabilityedge}
|m_0(z)-u(z)| \leq \frac{C_0\Delta(z)}{g(z)+\sqrt{\Delta(z)}}.
\end{equation}
\item If $\Im z<\nu$, $|\Delta_0(z)| \leq \Delta(z)$, and
$|m_0(z)-u(z)|<(\log N)^{-1/3}$, then
\begin{equation}\label{eq:minR1R2edge}
\min(|R_1(z)|,|R_2(z)|) \leq 
\frac{C_0\Delta(z)}{g(z)+\sqrt{\Delta(z)}},
\end{equation}
\begin{equation}\label{eq:gapedge}
C_1^{-1}(g(z)-\sqrt{\Delta(z)}) \leq
|R_1(z)-R_2(z)| \leq C_1(g(z)+\sqrt{\Delta(z)}).
\end{equation}
\end{enumerate}
We verify the second claim first: If $\Im z<\nu$ and $|m_0(z)-u(z)|<(\log
N)^{-1/3}$, then for $\nu$ and $\delta$ sufficiently small,
Lemma \ref{lemma:m0uniformcontinuity} implies
\begin{equation}\label{eq:muclose}
|m_0(z)-m_*|<C\sqrt{\nu+\delta},\qquad |u(z)-m_*|<C\sqrt{\nu+\delta}
\end{equation}
for a constant $C>0$ independent of $\nu,\delta$. We have
\[\frac{m_*z_0''(m_*)}{2}=-\frac{1}{m_*^2}+\frac{1}{N}\sum_{\a=1}^M
\frac{t_\a^3 m_*}{(1+t_\a m_*)^3}
=-\frac{1}{N}\sum_{\a=1}^M \frac{t_\a^2}{(1+t_\a m_*)^3},\]
where the second equality applies the identity $0=z_0'(m_*)$.
Comparing the right side with $u(z)\alpha(z)$, and applying (\ref{eq:muclose})
together with the bounds $|m_*| \asymp 1$, $|z_0''(m_*)| \asymp 1$,
and $|1+t_\a m_*| \asymp 1$ from Proposition \ref{prop:basicregbounds}, we
obtain $c<|\alpha(z)|<C$ for constants $C,c>0$ and sufficiently small
$\nu,\delta$. Next, applying again $0=z_0'(m_*)$, we have
\[z_0'(m)=\int_{m_*}^m z_0''(x)dx
=(m-m_*)z_0''(m_*)+\int_{m_*}^m \int_{m_*}^x z_0'''(y) dy\,dx.\]
Applying (\ref{eq:muclose}), $|z_0''(m_*)| \asymp 1$ from
Proposition \ref{prop:basicregbounds}, $|m_0(z)-m_*| \asymp \sqrt{\kappa+\eta}$
from Proposition \ref{prop:m0estimates}, and $|z_0'''(y)|<C$ from
Lemma \ref{lemma:z0thirdderivative}, we obtain
$cg(z)<|\beta(z)|<Cg(z)$ for $\nu,\delta$ sufficiently small.
Applying these bounds and $|\Delta_0(z)| \leq \Delta(z)$
to (\ref{eq:minR1R2}) and (\ref{eq:R1R2diff}) yields
(\ref{eq:minR1R2edge}) and (\ref{eq:gapedge}). Letting $\nu$ be small enough
such that this holds, for $\Im z \geq \nu$,
the same argument as in part (a) implies $|m_0(z)-u(z)| \leq (4/\nu^2)
\Delta(z)$. Noting $g(z) \geq \sqrt{\nu}$ and increasing $C_0$ if necessary,
we obtain (\ref{eq:stabilityedge}).

We again apply a continuity argument to conclude the proof: 
Consider any $z \in \bD$ with $|\Delta_0(w)| \leq \Delta(w)$ for all $w \in
L(z)$. If $\Im z \geq \nu$, the result follows from (\ref{eq:stabilityedge}).
If $\Im z<\nu$, suppose first that
\begin{equation}\label{eq:inductivez}
\frac{C_0\Delta(z)}{g(z)+\sqrt{\Delta(z)}}+2N^{-3}
<(2C_1)^{-1}(g(z)-\sqrt{\Delta(z)}).
\end{equation}
Note that by monotonicity of $\Delta$, the left side is decreasing in $\Im z$
while the right side is increasing in $\Im z$. Thus if
(\ref{eq:inductivez}) holds at $z$, then it holds at all $w \in L(z)$.
Let $w \in L(z)$ be such that $\Im z<\Im w
\leq \Im z+N^{-5}$, and suppose inductively that we have established
(\ref{eq:stabilityedge}) at $w$. Then
\[|m_0(z)-u(z)| \leq \frac{C_0\Delta(w)}{g(w)+\sqrt{\Delta(w)}}
+2N^{-3}<(\log N)^{-1/3}.\]
Then (\ref{eq:gapedge}) and (\ref{eq:inductivez}) imply
$|m_0(z)-u(z)|=\min(|R_1(z)|,|R_2(z)|)$, so (\ref{eq:minR1R2edge}) implies
(\ref{eq:stabilityedge}) holds at $z$.
Starting the induction at $\Im z \geq \nu$, this establishes
(\ref{eq:stabilityedge}) if $z$ satisfies (\ref{eq:inductivez}).

If $z$ does not satisfy (\ref{eq:inductivez}), then rearranging
(\ref{eq:inductivez}) and applying $\Delta(z)>N^{-3}$ yields
$g(z)^2 \leq C\Delta(z)$ for a constant $C>0$. Then
\[\frac{C_0\Delta(z)}{g(z)+\sqrt{\Delta(z)}}
+C_1(g(z)+\sqrt{\Delta(z)}) \leq 
\frac{C_2\Delta(z)}{g(z)+\sqrt{\Delta(z)}}\]
for a constant $C_2>0$. We claim
\begin{equation}\label{eq:stabilityedge2}
|m_0(z)-u(z)| \leq \frac{C_2\Delta(z)}{g(z)+\sqrt{\Delta(z)}}.
\end{equation}
Indeed, let $w \in L(z)$ be such that
$\Im z<\Im w \leq \Im z+N^{-5}$, and suppose inductively that
we have established (\ref{eq:stabilityedge2}) at $w$. This implies
in particular $|m_0(z)-u(z)|<(\log N)^{-1/3}$ as before, so
(\ref{eq:stabilityedge2}) holds at $z$ by (\ref{eq:minR1R2edge}) and
(\ref{eq:gapedge}). Starting the induction at the value
$w \in L(z)$ satisfying (\ref{eq:inductivez}) which has the smallest
imaginary part, this concludes the proof in all cases.
\end{proof}

We now verify Theorems \ref{thm:sticktobulk}, \ref{thm:regedgeconcentration}, 
and \ref{thm:locallaw}.

\begin{proof}[Proof of Theorem \ref{thm:sticktobulk}]
By the bound $\|\hSigma\| \leq \|T\|\|X\|^2$, we may take $C_0>0$
sufficiently large such that $\|\hSigma\| \leq C_0$
with probability at least $1-N^{-D}$.
Define
\[\bD=\{z \in \C^+:\,\Re \in [-C_0,C_0] \setminus \supp(\mu_0)_\delta,\,
\Im z \in [N^{-2/3},1]\}.\]
Then Propositions \ref{prop:m0regularoutside}, \ref{prop:m0basicoutside},
and Lemma \ref{lemma:stability}(a) check the conditions of Theorem
\ref{thm:generallocallaw} for $g(z) \equiv 1$ over $\bD$.

Applying the second bound of
Theorem \ref{thm:generallocallaw}(b), $|m_N(z)-m_0(z)|
\prec \Psi(z)^2 \asymp N^{-1}+(N\eta)^{-2}$ for any $z \in \bD$. Taking
$\eta=N^{-2/3}$ and applying also $\Im m_0(z) \asymp
\eta$, we obtain $\Im m_N(z) \prec N^{-2/3}<1/(2N\eta)$. 
As the number of eigenvalues of $\hSigma$ in $[E-\eta,E+\eta]$ is at most
$2N\eta \cdot \Im m_N(z)$, this implies $\hSigma$ has no eigenvalues in this
interval with probability $1-N^{-D}$ for all $N \geq N_0(D)$. 
The result follows from a union
bound over a grid of values $E \in [-C_0,C_0] \setminus \supp(\mu_0)_\delta$
of cardinality at most $CN^{2/3}$,
together with the bound $\|\hSigma\| \leq C_0$.
\end{proof}

\begin{proof}[Proof of Theorem \ref{thm:regedgeconcentration}]
The argument follows \cite[Eq.\ (3.4)]{pillaiyin}.
Consider the case of a right edge $E_*$. (A left edge is analogous.)
For each $E \in [E_*+N^{-2/3+\eps},E_*+\delta]$, denoting $\kappa=E-E_*$,
consider $z=E+i\eta$ for
\[\eta=N^{-1/2-\eps/4}\kappa^{1/4} \in [N^{-2/3},1],\]
where the inclusion holds for all large $N$ because $\kappa \in
[N^{-2/3+\eps},\delta]$. Proposition \ref{prop:m0estimates} implies
\[\Im m_0(z) \leq \frac{C\eta}{\sqrt{\kappa+\eta}}
\leq \frac{C\eta}{\sqrt{\kappa}}=C(N\eta)^{-1}N^{-\eps/2}.\]
Also by Proposition \ref{prop:m0estimates} and
Lemma \ref{lemma:stability}(b), we may apply
Theorem \ref{thm:generallocallaw} with $g(z)=\sqrt{\kappa+\eta}$.
The above bound on $\Im m_0(z)$ yields $\Psi(z)^2 \leq C/(N\eta)^2$, and hence
Theorem \ref{thm:generallocallaw}(b) implies
\[|m_N(z)-m_0(z)| \prec \frac{1}{(N\eta)^2\sqrt{\kappa+\eta}}
\leq \frac{1}{(N\eta)^2\sqrt{\kappa}}
=\frac{1}{N^{3+\eps/2}\eta^4} \leq (N\eta)^{-1}N^{-\eps/2},\]
where the last bound uses $\eta \geq N^{-2/3}$. Thus we obtain
\[\Im m_N(z) \prec C(N\eta)^{-1}N^{-\eps/2}.\]
Then $\hSigma$ has no eigenvalues in $[E-\eta,E+\eta]$ with probability
$1-N^{-D}$ for all $N \geq N_0(D)$, and the result follows from a union bound
over a grid of such values $E$.
\end{proof}

\begin{proof}[Proof of Theorem \ref{thm:locallaw}]
This follows from
Theorem \ref{thm:generallocallaw} applied with $g(z)=\sqrt{\kappa+\eta}$,
and Proposition \ref{prop:m0estimates} and Lemma \ref{lemma:stability}(b).
\end{proof}

\begin{proof}[Proof of Corollary \ref{cor:locallawunionbound}]
This follows from Lemmas \ref{lemma:unionbound}(a) and
\ref{lemma:infiniteunion}. For a large enough constant $C>0$ and any $D>0$,
on an event of probability $1-N^{-D}$, we have $\|X\|<C$ for all $N \geq
N_0(D)$. The required boundedness and Lipschitz continuity
properties for Lemma \ref{lemma:infiniteunion} then follow from
(\ref{eq:locallawalt}), (\ref{eq:mNbound}), (\ref{eq:mNlipschitz}), and
Proposition \ref{prop:m0estimates}.
\end{proof}

\section{Stochastic domination and resolvent approximation}\label{appendix:tools}

We state several known elementary properties about stochastic
domination, and also prove Lemma \ref{lemma:resolventapprox} on the resolvent
approximation. (This follows the argument
of \cite[Lemma 6.1 and Corollary 6.2]{erdosyauyin}; we provide a
self-contained exposition, as we do not first establish eigenvalue rigidity.)

\begin{lemma}\label{lemma:unionbound}
Let $U$ be any index set, and suppose $\xi(u) \prec \Psi(u)$ for all $u \in U$.
\begin{enumerate}[(a)]
\item For any constant $C>0$, if $|U| \leq N^C$, then
$\sup_{u \in U} |\xi(u)|/\Psi(u) \prec 1$.
\item For any constant $C>0$, if $|U| \leq N^C$, then
$\sum_{u \in U} \xi(u) \prec \sum_{u \in U} \Psi(u)$.
\item If $u_1,u_2 \in U$, then
$\xi(u_1)\xi(u_2) \prec \Psi(u_1)\Psi(u_2)$.
\end{enumerate}
\end{lemma}
\begin{proof}
All three parts follow from a union bound, as $\eps,D>0$ in
(\ref{eq:domination}) are arbitrary.
\end{proof}

\begin{lemma}\label{lemma:expectationdomination}
Suppose $\xi \prec \Psi$ and $\Psi$ is deterministic. Suppose furthermore that
there are constants $C,C_1,C_2,\ldots>0$ such that
$\Psi>N^{-C}$ and $\E[|\xi|^\ell]<N^{C_\ell}$ for each integer $\ell>0$.
Then $\E[\xi|\G] \prec \Psi$ for any sub-$\sigma$-field $\G$.
\end{lemma}
\begin{proof}
If $\G$ is trivial so $\E[\xi|\G]=\E[\xi]$, then this follows
from Cauchy-Schwarz: For any $\eps>0$ and all $N \geq N_0(\eps)$,
\begin{align*}
|\E \xi| &\leq \E\left[|\xi|\1\{|\xi| \leq N^{\eps/2}\Psi\}\right]
+\E\left[|\xi|\1\{|\xi|>N^{\eps/2}\Psi\}\right]\\
&\leq N^{\eps/2}\Psi+\E[|\xi|^2]^{1/2}\P[|\xi|>N^{\eps/2}\Psi]^{1/2}\\
&<N^{\eps}\Psi,
\end{align*}
where the last inequality applies $\xi \prec \Psi$.
For general $\G$, consider any $\eps,D>0$ and fix an integer
$k>(D+\eps)/\eps$. Then the above argument
yields $\E[|\xi|^k]<N^\eps \Psi^k$ for all $N \geq N_0(\eps,D)$, so
\[\P\Big[|\E[\xi|\G]|>N^\eps\Psi\Big]
\leq \frac{\E[|\E[\xi|\G]|^k]}{N^{k\eps}\Psi^k}
\leq \frac{\E[|\xi|^k]}{N^{k\eps}\Psi^k}<N^{\eps-k\eps}<N^{-D}.\]
\end{proof}

\begin{lemma}\label{lemma:infiniteunion}
Suppose $\xi(z) \prec \Psi(z)$ for all $z \in U$, where $U \subset \C$ is
uniformly bounded in $N$.
Suppose that for any $D>0$, there exists $C \equiv C(D)>0$
and an event of probability $1-N^{-D}$ on which
$\Psi(z)>N^{-C}$ for all $z \in U$, and also
$|\xi(z_1)-\xi(z_2)| \leq N^C|z_1-z_2|$ and
$|\Psi(z_1)-\Psi(z_2)| \leq N^C|z_1-z_2|$ for all $z_1,z_2 \in U$.
Then $\sup_{z \in U} |\xi(z)|/\Psi(z) \prec 1$.
\end{lemma}
\begin{proof}
For any $\eps,D>0$, set $C=C(D)$ and $\Delta=N^{-3C}$.
Take a net $\N \subset U$ with $|\N| \leq N^{6C+1}$ such that
for every $z \in U$, there exists $z' \in \N$ with $|z-z'|<\Delta$.
By Lemma \ref{lemma:unionbound}(a), $|\xi(z')|<N^\eps \Psi(z')$ for all $z' \in
\N$ with probability $1-N^{-D}$. Then with probability $1-2N^{-D}$, for all
$z \in U$,
\[|\xi(z)| \leq |\xi(z')|+\Delta N^C
<N^\eps \Psi(z')+\Delta N^C
\leq N^\eps \Psi(z)+2\Delta N^{\eps+C}
<3N^\eps \Psi(z).\]
\end{proof}

\begin{proof}[Proof of Lemma \ref{lemma:resolventapprox}]
Denote
\[\#(a,b)=\text{number of eigenvalues of } \hSigma \text{ in }
[a,b].\]
For any $E_1<E_2$, any $m>0$, and any $\lambda \in \R$, we have the
casewise bound
\begin{align*}
&\left|\1_{[E_1,E_2]}(\lambda)-\int_{E_1}^{E_2}
\frac{1}{\pi}\frac{\eta}{\eta^2+(x-\lambda)^2}dx\right|
\leq \begin{cases}
\frac{E_2-E_1}{\pi}\frac{\eta}{\eta^2+(E_1-\lambda)^2} & \text{if }
\lambda<E_1-m\\
1 & \text{if } E_1-m \leq \lambda \leq E_1+m \\
\frac{2}{\pi}\frac{\eta}{m} & \text{if } E_1+m<\lambda<E_2-m\\
1 & \text{if } E_2-m \leq \lambda \leq E_2+m \\
\frac{E_2-E_1}{\pi} \frac{\eta}{\eta^2+(\lambda-E_2)^2}
& \text{if } \lambda>E_2+m,
\end{cases}
\end{align*}
where the middle case $E_1+m<\lambda<E_2-m$ follows from
\begin{align*}
1-\int_{E_1}^{E_2} \frac{1}{\pi}\frac{\eta}{\eta^2+(x-\lambda)^2}dx
&\leq 1-\int_{\lambda-m}^{\lambda+m} \frac{1}{\pi}\frac{\eta}
{\eta^2+(x-\lambda)^2}dx\\
&=1-\frac{2}{\pi}\tan^{-1}\left(\frac{m}{\eta}\right) \leq \frac{2}{\pi}
\frac{\eta}{m}.
\end{align*}
For the first case, we apply also the bound
\[\frac{\eta}{\eta^2+(E_1-\lambda)^2} \leq \frac{\eta}{(E_1-\lambda)^2}
\leq \frac{2\eta}{m} \cdot \frac{m}{m^2+(E_1-\lambda)^2},\]
and similarly for the last case.
Hence, summing over $\lambda$ as the eigenvalues of $\hSigma$,
\begin{equation}\label{eq:countapprox}
\left|\#(E_1,E_2)-\frac{N}{\pi}\int_{E_1}^{E_2} \Im m_N(x+i\eta)dx\right|
\leq R(E_1,E_2,m)+S(E_1,E_2,m)
\end{equation}
where we set
\begin{align*}
R(E_1,E_2,m)&=\#(E_1-m,E_1+m)+\#(E_2-m,E_2+m),\\
S(E_1,E_2,m)&=\frac{2}{\pi}\frac{\eta}{m}\Big((E_2-E_1)N\Im m_N(E_1+im)\\
&\hspace{0.5in}+(E_2-E_1)N\Im m_N(E_2+im)+\#(E_1+m,E_2-m)\Big).
\end{align*}

We apply the above with $E_1,E_2 \in [E_*-2s_+,E_*+2s_+]$, and with
$m=N^{-2/3-3\eps}$. To bound $S(E_1,E_2,m)$, note that
Proposition \ref{prop:m0estimates} and Theorem \ref{thm:locallaw} yield,
for $j=1,2$,
\[\Im m_N(E_j+im) \prec N^{-1/3+3\eps}.\]
For $z=E_*+i(2s_+)$, Proposition \ref{prop:m0estimates} and
Theorem \ref{thm:locallaw} also yield $Ns_+\Im m_N(z) \prec N^{3\eps/2}$.
Applying $\#(E_*-v,E_*+v) \leq 2Nv\,\Im m_N(E_*+iv)$ for any $v>0$, this yields
\begin{equation}\label{eq:roughrigidity}
\#(E_*-2s_+,E_*+2s_+) \prec N^{3\eps/2}.
\end{equation}
Then applying $\#(E_1+m,E_2-m) \leq \#(E_*-2s_+,E_*+2s_+)$ and
$\eta/m=N^{-6\eps}$, we obtain $S(E_1,E_2,m) \prec N^{-2\eps}$.
By (\ref{eq:mNlipschitz}) and Lemma \ref{lemma:infiniteunion}, we may take a
union bound over all such $E_1,E_2$: For any $\eps',D>0$,
\begin{align}
&\P\bigg[\text{there exist } E_1,E_2 \in [E_*-2s_+,E_*+2s_+]
\text{ such that }
S(E_1,E_2,m)>N^{-2\eps+\eps'}\bigg] \leq N^{-D} \label{eq:countR2}
\end{align}
for all $N \geq N_0(\eps',D)$.

Now let $E=E_*+s$ and $E_+=E_*+s_+-l$. Then
\begin{align*}
\#(E,E_+) &\leq \frac{1}{l^2}\int_{E-l}^{E}\left(\int_{E_+}^{E_++l}
\#(E_1,E_2)dE_2\right)dE_1\\
&\leq \frac{N}{\pi}\int_{E-l}^{E_++l} \Im m_N(x+i\eta)dx
+\frac{1}{l^2}\int_{E-l}^{E} \int_{E_+}^{E_++l}
R(E_1,E_2,m)dE_2\,dE_1+\O(N^{-2\eps}),
\end{align*}
where we have applied (\ref{eq:countapprox}) and (\ref{eq:countR2}).
The first term is $\pi^{-1}\X(s-l,s_+,\eta)$.
For the second term, we obtain from the definition of $R(E_1,E_2,m)$
\begin{align*}
&\frac{1}{l^2}\int_{E-l}^{E} \int_{E_+}^{E_++l}
R(E_1,E_2,m)dE_2\,dE_1\\
&\quad \leq \frac{2m}{l}\#(E-l-m,E+m)
+\frac{2m}{l}\#(E_+-m,E_++l+m).
\end{align*}
Applying (\ref{eq:roughrigidity}) to crudely bound
$\#(E-l-m,E+m)$ and $\#(E_+-m,E_++l+m)$ by $\#(E_*-2s_+,E_*+2s_+)$,
and noting $m/l=N^{-2\eps}$, we obtain
\[\#(E,E_+) \leq \pi^{-1} \X(s-l,s_+,\eta)+\O(N^{-\eps/2}).\]
Theorem \ref{thm:regedgeconcentration} yields $\#(E_+,E_*+\delta)=0$ with
probability $1-N^{-D}$ for $N \geq N_0(\eps,D)$, so
\begin{equation}\label{eq:Edeltaupper}
\#(E,E_*+\delta) \leq \pi^{-1} \X(s-l,s_+,\eta)+\O(N^{-\eps/2}).
\end{equation}
Similarly, setting $E_+=E_*+s_++l$, we have
\begin{align}
\#(E,E_*+\delta) &\geq \frac{1}{l^2} \int_E^{E+l}
\left(\int_{E_+-l}^{E_+} \#(E_1,E_2)dE_2\right)dE_1\nonumber\\
&\geq \pi^{-1}\X(s+l,s_+,\eta)-\O(N^{-\eps/2}).
\label{eq:Edeltalower}
\end{align}

For any $D>0$ and all $N \geq N_0(\eps,D)$, (\ref{eq:Edeltaupper}) implies that
$\pi^{-1} \X(s-l,s_+,\eta) \geq 2/3$ whenever $\#(E_*+s,E_*+\delta) \geq 1$, 
except possibly on an event of probability $N^{-D}$. Similarly
(\ref{eq:Edeltalower}) implies $\pi^{-1} \X(s+l,s_+,\eta) \leq 1/3$ whenever
$\#(E_*+s,E_*+\delta)=0$, except possibly on
an event of probability $N^{-D}$. The result then follows from the definition
and boundedness of $K$.
\end{proof}

\section{Testing in random effects models}\label{appendix:testing}

We discuss further the application of Theorem \ref{thm:TW} for testing the
global sphericity null hypothesis in linear mixed models. In Example
\ref{ex:oneway} below, we describe explicitly the form of this test for a
balanced one-way classification design, including an additional fixed-effect
mean vector as is common in applications of this model.

In balanced classification designs, regularity of the rightmost edge
may be verified from the following simple sufficient condition,
noted also in \cite{elkaroui}.

\begin{proposition}\label{prop:balancedregular}
Suppose there exists a constant
$c>0$ such that the largest diagonal value of $T$ is at least $c$ and has
multiplicity at least $cM$. Then the rightmost edge $E_*$ of $\mu_0$ is
$\tau$-regular for a constant $\tau>0$.
\end{proposition}
\begin{proof}
Let $t_1$ be the maximum diagonal value of $T$, and let $K$ be its multiplicity.
The $m$-value $m_*$ for the rightmost edge satisfies $m_* \in (-t_1^{-1},0)$.
As $t_1>c$ for a constant $c>0$, this implies $|m_*|<1/c$.
Furthermore, we have
\begin{equation}\label{eq:dz0mstar}
0=z_0'(m_*)=\frac{1}{m_*^2}-\frac{1}{N}\sum_{\a:t_\a \neq 0}
\frac{1}{(m_*+t_\a^{-1})^2}.
\end{equation}
As $|t_\a^{-1}|>c$ for a constant $c>0$ and each $\a$, this implies $|m_*|>c$
for a constant $c>0$. The condition (\ref{eq:dz0mstar}) also implies
\[0 \leq \frac{1}{m_*^2}-\frac{K}{N}\frac{1}{(m_*+t_1^{-1})^2}.\]
As $K$ is proportional to $N$, this yields $|m_*+t_1^{-1}|>c$ for a constant
$c>0$. Then by the condition $m_* \in (-t_1^{-1},0)$, we obtain
$|m_*+t_\a^{-1}|>\tau$ for all non-zero $\a$ and some constant
$\tau>0$. Finally, we have
\[z_0''(m_*)=-\frac{2}{m_*^3}+\frac{2}{N}\sum_{\a:t_\a \neq 0}
\frac{1}{(m_*+t_\a^{-1})^3}
=\sum_{\a:t_\a \neq 0} -\frac{2}{m_*N} \cdot
\frac{t_\a^{-1}}{(m_*+t_\a^{-1})^3},\]
where the second equality applies (\ref{eq:dz0mstar}). Note that $m_*<0$, and
$m_*+t_\a^{-1}>0$ if $t_\a>0$ and $m_*+t_\a^{-1}<0$ if $t_\a<0$. Thus each
summand on the right side above is positive, and in particular
\[z_0''(m_*) \geq -\frac{2K}{m_*N} \cdot \frac{t_1^{-1}}{(m_*+t_1^{-1})^3}.\]
Thus $\gamma<\tau^{-1}$ for a constant $\tau>0$.
\end{proof}

In testing applications where the variances
$\sigma_r^2$ are unknown, they may be estimated as follows.

\begin{proposition}\label{prop:bulkestimates}
Fix $r \in \{1,\ldots,k\}$ and let $\hSigma=Y'BY$ be an unbiased estimator for
$\Sigma_r$ in the mixed effects linear model (\ref{eq:MANOVAmodel}).
Suppose the null hypothesis (\ref{eq:nullhypothesis}) holds, and there is a
constant $C>0$ such that $\sigma_s \leq C$, $\|U_s\| \leq C$, and
$\|B\| \leq C/n$ for all $s \in \{1,\ldots,k\}$.

Let $\hsigma^2=p^{-1}\Tr \hSigma$. 
Then for any $\eps,D>0$ and all $n \geq n_0(\eps,D)$,
\[\P[|\hsigma^2-\sigma_r^2|>n^{-1+\eps}]<n^{-D}.\]
\end{proposition}
\begin{proof}
Note that $\E[\hsigma^2]=\sigma_r^2$. Writing $\hSigma=X'FX$ where $X$ has
$\N(0,1/N)$ entries and $F$ is defined by (\ref{eq:F}), we have
\[\hat{\sigma}^2=N^{-1}\Tr X'FX=\vec(X)'A\vec(X)\]
where $A=N^{-1}\Id_N \otimes F$ and
$\vec(X)$ is the column-wise vectorization of $X$. The condition
$\E[\hsigma^2]=\sigma_r^2$ implies $N^{-1}\Tr A=\sigma_r^2$.
We have $\|A\|_\HS^2=N^{-1}\|F\|_\HS^2<C$ for a constant $C>0$ under the above
conditions, so the result
follows from the Hanson-Wright inequality.
\end{proof}
Replacing any $\sigma_1^2,\ldots,\sigma_k^2$ that are unknown
by $\hsigma_1^2,\ldots,\hsigma_k^2$ and computing
$\widehat{E}_*$ and $\hat{\gamma}$ using these estimated variances,
one may check that when $E_*$ is regular,
\[\P[|\widehat{E}_*-E_*|>n^{-1+\eps}]<n^{-D}, \qquad
\P[|\hat{\gamma}-\gamma|>n^{-1+\eps}]<n^{-D}.\]
This follows from an argument similar to Lemma \ref{lemma:singleswap}, which we
omit for brevity. Then the conclusion of Theorem \ref{thm:TW} remains
asymptotically valid using the estimated center and scale
$\widehat{E}_*,\hat{\gamma}$.

\begin{table}
\centering
\caption{Empirical cumulative probabilities for
$(\gamma p)^{2/3}(\lambda_{\max}(\hSigma_1)-E_*)$ at the theoretical
90th, 95th, and 99th percentiles of the Tracy-Widom $F_1$ law, estimated
across 10000 simulations.
Here, $\hSigma_1$ is the MANOVA estimator of $\Sigma_1$ in the
balanced one-way classification model, for various $n,p,J$ when $\Sigma_1=0$ and
$\Sigma_2=\Id$. The final column gives approximate standard errors based on
binomial sampling.}
\label{table:simulation}
\begin{tabular}{l|l|SSS|SSS|c}
\toprule
&\multirow{2}{*}{$F_1$}
& \multicolumn{3}{c}{$n=p$} & \multicolumn{3}{c}{$n=4 \times p$} & \\
& & {$J=2$} & {$J=5$} & {$J=10$} & {$J=2$} & {$J=5$} & {$J=10$} &
$2 \times \text{SE}$ \\
\midrule
\multirow{3}{*}{$p=20$}
& 0.90 & 0.941 & 0.949 & 0.959 & 0.931 & 0.934 & 0.940 & (0.005) \\ 
& 0.95 & 0.973 & 0.977 & 0.983 & 0.968 & 0.969 & 0.971 & (0.003) \\
& 0.99 & 0.995 & 0.997 & 0.997 & 0.994 & 0.994 & 0.993 & (0.002) \\
\midrule
\multirow{3}{*}{$p=100$}
& 0.90 & 0.926 & 0.928 & 0.934 & 0.920 & 0.916 & 0.919 & (0.005) \\
& 0.95 & 0.964 & 0.967 & 0.968 & 0.960 & 0.958 & 0.961 & (0.004) \\
& 0.99 & 0.993 & 0.995 & 0.995 & 0.992 & 0.991 & 0.992 & (0.002) \\
\midrule
\multirow{3}{*}{$p=500$}
& 0.90 & 0.914 & 0.920 & 0.919 & 0.916 & 0.915 & 0.921 & (0.006) \\
& 0.95 & 0.958 & 0.961 & 0.960 & 0.957 & 0.957 & 0.962 & (0.004) \\
& 0.99 & 0.992 & 0.993 & 0.993 & 0.992 & 0.992 & 0.993 & (0.002) \\
\bottomrule
\end{tabular}
\end{table}

\begin{example}\label{ex:oneway}
As a concrete example, consider the balanced one-way classification model
\[\y_{i,j}=\bmu+\balpha_i+\beps_{i,j} \in \R^p\]
with $I$ groups of $J$ samples per group, as discussed in the introduction and
with an additional deterministic mean vector $\bmu \in \R^p$.
This model is expressed in matrix form as
\[Y=\one_n\bmu'+U\alpha+\eps,\]
where the rows of $Y \in \R^{n \times p}$, $\alpha \in \R^{I \times p}$, and
$\eps \in \R^{n \times p}$ are the above vectors, and where
$\one_n$ denotes the all-1's column vector of length $n$ and
\begin{equation}\label{eq:U}
U=\Id_I \otimes \one_J=\begin{pmatrix} \one_J & & \\
& \ddots & \\
& & \one_J \end{pmatrix} \in \{0,1\}^{n \times I}
\end{equation}
is an incidence matrix encoding the group 
memberships. Denoting by $\pi_1,\pi_2 \in \R^{n \times n}$ the
orthogonal projections onto $\col(U) \ominus \col(\one_n)$ (the orthogonal
complement of $\one_n$ in the column span of $U$) and onto
$\R^n \ominus \col(U)$ (the orthogonal complement of the column span of $U$
in $\R^n$), the classical MANOVA estimators \cite{searleetal,searlerounsaville}
are $\hSigma_1=Y'B_1Y$ and $\hSigma_2=Y'B_2Y$ for
\[B_1=\frac{1}{J}\frac{\pi_1}{I-1}-\frac{1}{J}\frac{\pi_2}{n-I},
\qquad B_2=\frac{\pi_2}{n-I}.\]

Let us consider a test of
\[H_0:\Sigma_1=\sigma_1^2\Id,\;\Sigma_2=\sigma_2^2\Id\]
using the largest observed eigenvalue of $\hSigma_1$. To obtain a more explicit
form for $F$, set $B \equiv B_1$ and
write the singular value decomposition of $U$ as
\[U=\sqrt{J}V_0W_0'+\sqrt{J}V_1W_1'\]
where $V_0=\one_n/\sqrt{n}$ and the columns of $V_1 \in \R^{n \times (I-1)}$
collect the left singular vectors of $U$, and $W_0 \in \R^{I \times 1}$
and $W_1 \in \R^{I \times (I-1)}$ are the corresponding right singular vectors.
Letting $V_2 \in \R^{n \times (n-I)}$ have orthonormal
columns spanning $\R^n \ominus \col(U)$, we have $\pi_1=V_1V_1'$ and
$\pi_2=V_2V_2'$. Then, after some simplification,
\[F=Q \begin{pmatrix} \frac{p\sigma_1^2}{I-1}\Id_{I-1} & 
\frac{p\sigma_1\sigma_2}{\sqrt{J}(I-1)}\Id_{I-1} & 0 \\
\frac{p\sigma_1\sigma_2}{\sqrt{J}(I-1)}\Id_{I-1} &
\frac{p\sigma_2^2}{J(I-1)}\Id_{I-1} & 0 \\
0 & 0 & -\frac{p\sigma_2^2}{J(n-I)}\Id_{n-I} \end{pmatrix} Q',\]
\[Q=\begin{pmatrix} W_1 & 0 & 0 \\ 0 & V_1 & V_2 \end{pmatrix}.\]
As $Q$ has orthonormal columns,
the nonzero eigenvalues of $F$ are the same as those of $Q'FQ$.
Diagonalization yields that $F$ has $I-1$ eigenvalues equal to $t_1$,
$n-I$ eigenvalues equal to $t_2$, and remaining eigenvalues 0, where
\[t_1=\frac{p}{I-1}(\sigma_1^2+\sigma_2^2/J),\qquad
t_2=-\frac{p}{J(n-I)}\sigma_2^2.\]
Then the Marcenko-Pastur equation (\ref{eq:MPdiag}) is cubic in $m_0(z)$, and we
have the explicit form
\[z_0(m)=-\frac{1}{m}+\frac{I-1}{p}\cdot \frac{1}{m+t_1^{-1}}+\frac{n-I}{p}
\cdot \frac{1}{m+t_2^{-1}}.\]

Table \ref{table:simulation} displays the accuracy of the Tracy-Widom
approximation for the standardized largest eigenvalue
$(\gamma p)^{2/3}(\lambda_{\max}(\hSigma_1)-E_*)$, under $\sigma_1^2=0$,
$\sigma_2^2=1$, and various settings
of $n$, $p$, and group size $J$. The center and scale $E_*$ and $\gamma$ are
computed from $z_0(m)$ above, where we have assumed that $\sigma_1^2$ and
$\sigma_2^2$ are known. We observe that the approximation is reasonably
accurate but has a conservative bias, particularly for small sample sizes.
\end{example}

\newcommand{\etalchar}[1]{$^{#1}$}

\end{document}